\documentclass[UTF-8,reqno]{amsart}
\usepackage{enumerate}
\usepackage{verbatim,listings}
\usepackage{amssymb,url,color, booktabs}

\usepackage{mathrsfs}
\usepackage{amsmath}
\usepackage{verbatim}
\usepackage{fancyhdr}
\usepackage[nobysame]{amsrefs}
\BibSpec{article}{%
+{}{\PrintAuthors} {author}
+{,}{ \textrm} {title}
+{.}{ \textit} {journal}
+{,}{ \textbf} {volume}
+{}{ \parenthesize} {date}
+{,}{ } {pages}
+{.}{ arXiv:} {eprint}
+{.}{} {transition}
}
\BibSpec{book}{%
+{}{\PrintAuthors} {author}
+{,}{ \textit} {title}
+{.}{ \textrm} {series} 
+{,}{ Vol.} {volume}
+{.}{ } {publisher}
+{,}{ } {date}
+{.}{} {transition}
}
\usepackage{color}
\usepackage{tikz}
\usepackage{graphicx}
\usepackage[colorlinks=true]{hyperref}
\hypersetup{
    linkcolor=blue,          
    citecolor=red,        
    filecolor=blue,      
    urlcolor=cyan
}

\numberwithin{equation}{section}
\def\theequation{\arabic{section}.\arabic{equation}}
\newcommand{\be}{\begin{eqnarray}}
\newcommand{\ee}{\end{eqnarray}}
\newcommand{\ce}{\begin{eqnarray*}}
\newcommand{\de}{\end{eqnarray*}}
\newtheorem{theorem}{Theorem}[section]
\newtheorem{condition}{Condition}[section]
\newtheorem{remark}{Remark}[section]
\newtheorem{conjecture}{Conjecture}[section]
\newtheorem{lemma}[theorem]{Lemma}
\newtheorem{definition}[theorem]{Definition}
\newtheorem{proposition}[theorem]{Proposition}
\newtheorem{Examples}[theorem]{Example}
\newtheorem{corollary}[theorem]{Corollary}

\newenvironment{proof of theorem 1.4}{{\it Proof of Theorem 1.4}.}{{\hfill 	
		$\square$\hskip - \parfillskip}}
\newenvironment{proof of theorem 1.5}{{\it Proof of Theorem 1.5}.}{{\hfill 	
		$\square$\hskip - \parfillskip}}
\newenvironment{proof of theorem 1.6}{{\it Proof of Theorem 1.6}.}{{\hfill 	
		$\square$\hskip - \parfillskip}}

\makeatletter

\newcommand{\Rmnum}[1]{\expandafter\@slowromancap\romannumeral #1@}
\makeatother

\usepackage{caption}
\def\u{\mathbf{u}}

\def\p{\partial}

\def\[{{\Big[}}
\def\]{{\Big]}}
\def\<{{\langle}}
\def\>{{\rangle}}
\def\({{\Big(}}
\def\){{\Big)}}

\def\bx{{\mathbf{x}}}
\def\tr{\mathrm {tr}}

\def\sgn{\mbox{\rm sgn}}

\def\min{{\mathord{{\rm min}}}}
\def\Vol{\mathord{{\rm Vol}}}

\def\={&\!\!=\!\!&}

\def\1{{\mathbf{1}}}

\def\geq{\geqslant}
\def\leq{\leqslant}

\def\k{\kappa}

\def\div{\mathord{{\rm div}}}

\def\u{\mathbf{u}}

\def\p{\partial}

\def\[{{\Big[}}
\def\]{{\Big]}}
\def\<{{\langle}}
\def\>{{\rangle}}
\def\({{\Big(}}
\def\){{\Big)}}

\def\bx{{\mathbf{x}}}
\def\tr{\mathrm {tr}}
\def\W{{\mathcal W}}

\def\sgn{\mbox{\rm sgn}}

\def\min{{\mathord{{\rm min}}}}
\def\Vol{\mathord{{\rm Vol}}}

\def\={&\!\!=\!\!&}
\def\bt{\begin{theorem}}
\def\et{\end{theorem}}
\def\bl{\begin{lemma}}
\def\el{\end{lemma}}
\def\br{\begin{remark}}
\def\er{\end{remark}}
\def\bx{\begin{Examples}}
\def\ex{\end{Examples}}
\def\bd{\begin{definition}}
\def\ed{\end{definition}}
\def\bp{\begin{proposition}}
\def\ep{\end{proposition}}
\def\bc{\begin{corollary}}
\def\ec{\end{corollary}}

\def\geq{\geqslant}
\def\leq{\leqslant}

\def\div{\mathord{{\rm div}}}

 \def\R{\mathbb R}
 \def\R{\mathbb R}

\def\<{\langle} \def\>{\rangle}

\def\bpf{\begin{proof}}
\def\epf{\end{proof}}

\allowdisplaybreaks

	\tikzset{
	pattern size/.store in=\mcSize,
	pattern size = 5pt,
	pattern thickness/.store in=\mcThickness,
	pattern thickness = 0.3pt,
	pattern radius/.store in=\mcRadius,
	pattern radius = 1pt}
\makeatletter
\pgfutil@ifundefined{pgf@pattern@name@_ak5ydppfi}{
	\pgfdeclarepatternformonly[\mcThickness,\mcSize]{_ak5ydppfi}
	{\pgfqpoint{0pt}{0pt}}
	{\pgfpoint{\mcSize+\mcThickness}{\mcSize+\mcThickness}}
	{\pgfpoint{\mcSize}{\mcSize}}
	{
		\pgfsetcolor{\tikz@pattern@color}
		\pgfsetlinewidth{\mcThickness}
		\pgfpathmoveto{\pgfqpoint{0pt}{0pt}}
		\pgfpathlineto{\pgfpoint{\mcSize+\mcThickness}{\mcSize+\mcThickness}}
		\pgfusepath{stroke}
}}
\makeatother
\tikzset{every picture/.style={line width=0.75pt}} 

\begin{document}
\title{Anisotropic  mean curvature type flow and  capillary Alexandrov-Fenchel inequalities}\thanks{\it {The research is partially supported by NSFC (Nos. 11871053 and 12261105).}}
\author{Shanwei Ding, Jinyu Gao and Guanghan Li}

\thanks{Email address: dingsw@whu.edu.cn;  jinyugao@whu.edu.cn; ghli@whu.edu.cn}

\address{School of Mathematics and Statistics, Wuhan University, Wuhan 430072, China.
}

\begin{abstract}
In this paper, an anisotropic volume-preserving mean curvature type flow for star-shaped anisotropic $\omega_0$-capillary hypersurfaces in the half-space is studied, and the long-time existence and  smooth convergence to a capillary Wulff shape are obtained. If the initial hypersurface is strictly convex, 
the solution of this flow remains to be strictly convex for all $t>0$ by adopting a new approach applicable to anisotropic capillary setting.
In analogy with closed hypersurfaces,
 if the $\omega_0$-capillary Wulff shape is a $\theta$-capillary hypersurface with constant contact angle $\theta$,  the  quermassintegrals for anisotropic capillary hypersurfaces match 
 the mixed volume of two $\theta$-capillary convex bodies.
  Thus, generalized quermassintegrals for anisotropic capillary hypersurfaces with general Wulff shapes (i.e., the  $\omega_0$-capillary Wulff shape has a variable contact angle) can be defined, which satisfy certain monotonicity properties along the flow.
  As applications, we establish an anisotropic capillary isoperimetric inequality for star-shaped anisotropic capillary hypersurfaces and a family of new Alexandrov-Fenchel inequalities for strictly convex anisotropic capillary hypersurfaces. 
  In particular, we  provide a flow's method to derive the Alexandrov-Fenchel inequalities for two $\theta$-capillary hypersurfaces,  demonstrated 
   in \cite{Xia-arxiv} from the view of point in convex geometry.
\end{abstract}

\maketitle
\setcounter{tocdepth}{2}
\tableofcontents

\section{Introduction}
In convex  geometry, there are some classical results (such as Alexandrov-Fenchel inequalities)  for mixed volume $V(K,\cdots,K,L,\cdots,L)$ of two different convex bodies $K,L\in\mathbb{R}^{n+1}$ (i.e. $\partial K,\partial L$ are strictly convex hypersurfaces). Particular interest is the case that one of the hypersurfaces is not a strictly convex hypersurface.
However, weakening this convexity condition presents significant challenges when using traditional convex geometric methods.
 Guan and Li in \cite{Guan-Li-2009} used  flow's methods to prove the Alexandrov-Fenchel inequalities for the case where  $\partial K$ is a $k$-convex hypersurface and $\partial L=\mathbb{S}^n$ is a unit sphere. Xia \cite{Xia2017} considered the anisotropic case, also used  flow's methods to prove the Alexandrov-Fenchel inequalities for the case where  $\partial K$ is $F$-mean convex hypersurface and $\partial L$ is a  strictly convex hypersurface. 
 Thus anisotropic curvature flow is a significant tool in the study of Alexandrov-Fenchel inequalities.


	The mean curvature flow is one of the classical curvature flows where the hypersurfaces evolve by their mean curvature vector, which has been introduced and studied  by Mullins \cite{Mullins1956}, Brakke \cite{Brakke1978}, Huisken \cite{Huisken1984}, and so on.
	Later,  Huisken  \cite{Huisken1987} modified the mean curvature flow and introduced a volume-preserving mean curvature flow with an extra global term to study the isoperimetric problem.

Inspired by the Minkowski formulas for closed hypersurfaces, Guan and Li \cite{Guan-Li-2015} introduced the locally constrained  volume-preserving  mean curvature type flow $\frac{\partial X}{\partial t}=(n-Hu)\nu$, where $X,H,\nu,$ and $u=\<X,\nu\>$ are the position vector, mean curvature, unit outward normal vector field,  and support function of the evolving hypersurfaces, respectively. It is proved that all the quermassintegrals have  monotonicity, and then a new proof of some Alexandrov-Fenchel inequalities from convex geometry is obtained.
Later, Wei and Xiong \cite{Wei-Xiong2021} considered the anisotropic analogue of locally constrained  volume-preserving mean curvature flow $\frac{\partial X}{\partial t}=(n-H_F\hat{ u})\nu_F$, where $H_F,\nu_F,$ and $\hat{ u}=\<X,\nu\>/F(\nu)$ are the anisotropic mean curvature, anisotropic outward normal vector field,  and anisotropic support function of the evolving hypersurfaces, respectively. For the exact definition of anisotropic geometric quantities we refer to Section \ref{sec 2}.

For the case of hypersurfaces with capillary boundary in the unit Euclidean $(n+1)$-ball, based on the Minkowski type formula \cite{Weng-Xia-2022}*{Proposition 2.8}, 
Wang and Xia \cite{Wang-Xia-2022} introduced the mean curvature type flow for hypersurfaces in the unit Euclidean ball with $\theta$-capillary boundary for the case $\theta=\frac{\pi}{2}$. Wang and Weng \cite{Wang-Weng-2020} considered the general case  $|\cos\theta|<\frac{3n+1}{5n-1}$, and then solved an  isoperimetric problem for hypersurfaces in the unit Euclidean ball with $\theta$-capillary boundary. In \cite{Hu-Wei-yang-zhou2023}, Hu, Wei, Yang and Zhou studied the  mean curvature type flow for general case $\theta\in (0,\frac{\pi}{2}]$, proved the preservation of strict convexity along the flow, and then established a family of new Alexandrov-Fenchel inequalities for strictly convex hypersurfaces in the unit Euclidean ball with capillary boundary.

Recently, Mei,  Wang, and Weng \cite{Mei-Wang-Weng} studied a new locally constrained mean curvature type flow $(\frac{\partial X}{\partial t})^{\perp}=\( (n+n \cos \theta\<\nu,-E_{n+1}\>)-Hu\)\nu$ for hypersurfaces with capillary boundary at a constant angle $\theta\in(0,\pi )$ in the half-space. Here $E_{n+1}$ denotes the $(n+1)$th-coordinate unit vector.
It is shown that the flow preserves the convexity for $\theta\in (0,\frac{\pi}{2}]$, 
and the capillary quermassintegrals evolve monotonically along the flow,  hence a class of new Alexandrov-Fenchel inequalities is established for convex hypersurfaces with capillary boundary in the half-space.
 For more works about mean curvature type flow, one refers to \cite{Guan-Li-Wang-2019,Li-Ma2020,Gao-Luo-Xu2024}.

 As we know that, the isotropic quermassintegrals for closed hypersurfaces match the mixed volume $V(K,\cdots,K,\mathbb{B}^{n+1},\cdots,\mathbb{B}^{n+1})$ with convex body $K$ and unit ball $\mathbb{B}^{n+1}$;  isotropic capillary quermassintegrals $\mathcal{V}_{k+1,\theta}(K)$ for hypersurfaces $\partial K$ with $\theta$-capillary boundary in the half space as introduced by \cite{Wang-Weng-Xia} can be reinterpreted as the mixed volume $V(K,\cdots,K,\widehat{\mathcal{C}_{\theta}},\cdots,\widehat{\mathcal{C}_{\theta}})$ of $(n-k)$ capillary convex bodies $K$ and $(k+1)$ capillary spherical caps $\mathcal{C}_{\theta}$ (see \cite{Xia-2017-convex}*{Proposition 2.15}).

 For anisotropic case, the anisotropic quermassintegrals for closed hypersurfaces match the mixed volume $V(K,\cdots,K,L,\cdots,L)$ with convex bodies $K$ and  $L$ (see \cite{Xia2017} or \cite{Wei-Xiong2022}). It is a natural consideration to construct  suitable anisotropic capillary quermassintegrals which match the mixed volume $V(K_{\theta},\cdots,K_{\theta},L_{\theta},\cdots,L_{\theta})$ of two $\theta$-capillary convex bodies $K_{\theta}$ and $L_{\theta}$.
Then, by establishing the relationship between anisotropic geometric quantities and convex bodies  geometric quantities, anisotropic curvature flow serves as an effective tool for extending the conclusions of convex  geometry.

At present, there are many studies on the curvature flows of anisotropic closed hypersurfaces or hypersurfaces with isotropic capillary boundary. There are also some issues related to hypersurfaces with  anisotropic capillary boundary (ref. \cite{Jia-Wang-Xia-Zhang2023,Guo-Xia,Koiso2023,Koiso2006,Koiso2007-1,Koiso2007-2,Ma-Ma-Yang}), but there is currently no research on flows for   anisotropic capillary hypersurfaces.

Inspired by \cite{Wei-Xiong2021,Hu-Wei-yang-zhou2023,Wang-Xia-2022,Wang-Weng-2020,Mei-Wang-Weng}, in this paper, we mainly consider the flow for  anisotropic capillary hypersurfaces in the half-space
\begin{align*}
	\mathbb{R}_+^{n+1}=\{x\in \mathbb{R}^{n+1} :\<x,E_{n+1}\>>0 \}.
\end{align*}

 Let $\mathcal{W}\subset \mathbb{R}^{n+1}$ be a given smooth closed strictly convex hypersurface containing the origin $O$ with a support function $F:\mathbb{S}^{n}\rightarrow\mathbb{R}_+$. The Cahn-Hoffman map associated with $F$ is given by
 \begin{align*}
 \Psi:\mathbb{S}^n\rightarrow\mathbb{R}^{n+1},\quad	\Psi(x)=F(x)x+\nabla^{\mathbb{S}} {F}(x),
 \end{align*}
where $\nabla^{\mathbb{S}}$ denotes the covariant derivative on $\mathbb{S}^n$ with standard round metric. Let $\Sigma$ be a smooth compact orientable hypersurface in $\overline{\mathbb{R}_+^{n+1}}$ with boundary $\partial\Sigma\subset\partial\mathbb{R}_+^{n+1}$, which encloses a bounded domain $\widehat{\Sigma}$. Let $\nu$ be the unit normal of $\Sigma$ pointing outward of  $\widehat{\Sigma}$. Given a constant  $\omega_0 \in(-F(E_{n+1}), F(-E_{n+1}))$, we say $\Sigma$ is an anisotropic $\omega_0$-capillary hypersurface (see \cite{Jia-Wang-Xia-Zhang2023}) if
\begin{align}
	\label{equ:w0-capillary}
	\<\Psi(\nu),-E_{n+1}\>=\omega_0,\quad \text{on}\ \partial\Sigma.
\end{align}
 The anisotropic contact angle $\hat{\theta}: \partial \Sigma \rightarrow(0, \pi)$ (see \cite{Jia-Wang-Xia-Zhang2023}) is defined by
\begin{equation}
	-\cos \hat{\theta}=\left\{\begin{array}{ll}
	F\left(E_{n+1}\right)^{-1}\left\langle \Psi(\nu),-E_{n+1}\right\rangle, & \text { if }\left\langle \Psi(\nu),-E_{n+1}\right\rangle<0, \\
	0, & \text { if }\left\langle \Psi(\nu),-E_{n+1}\right\rangle=0, \\
	F\left(-E_{n+1}\right)^{-1}\left\langle \Psi(\nu),-E_{n+1}\right\rangle, & \text { if }\left\langle \Psi(\nu),-E_{n+1}\right\rangle>0,
\end{array}\right. \label{equ:def-theta}
\end{equation}
which is a natural generalization of the contact angle in the isotropic case.
If $\hat{\theta}=\pi / 2$, 
it is called a free boundary anisotropic hypersurface. For anisotropic $\omega_0$-capillary hypersurface, $\omega_0\leq 0$ means $\hat{\theta}\in (0,\frac{\pi}{2}]$. If $\W=\mathbb{S}^n$, this corresponds to the isotropic case, where $\hat{\theta}=\theta$, representing an isotropic constant contact angle.


For the convenience of description, we need a constant vector $E_{n+1}^F \in \mathbb{R}^{n+1}$ (ref. \cite{Jia-Wang-Xia-Zhang2023}) defined as
 $$
 E_{n+1}^F= \begin{cases}\frac{\Psi\left(E_{n+1}\right)}{F\left(E_{n+1}\right)}, & \text { if } \omega_0<0, \\[5pt]
  -\frac{\Psi\left(-E_{n+1}\right)}{F\left(-E_{n+1}\right)}, & \text { if } \omega_0>0.\end{cases}
 $$
 Note that $E_{n+1}^F$ is the unique vector in the direction $\Psi\left(E_{n+1}\right)$, whose scalar product with $E_{n+1}$ is 1 (see \cite{Jia-Wang-Xia-Zhang2023}*{eq. (3.2)}). When $\omega_0=0$, one can define it by any unit vector. 

 The model example of anisotropic $\omega_0$-capillary hypersurface is the $\omega_0$-capillary Wulff shape (see \cite{Jia-Wang-Xia-Zhang2023}), which is a part of a Wulff shape in $\overline{\mathbb{R}^{n+1}_+}$ such that the anisotropic capillary boundary condition \eqref{equ:w0-capillary} holds. Denote by ${\W_{r_0,\omega_0}}=\W_{r_0}(x_0)\cap\overline{\mathbb{R}_{+}^{n+1}}$, where $ x_0=r_0\omega_0E^F_{n+1}$ and  $\W_{r_0}(x_0)$ is defined in \eqref{equ:w_r(x)}.

Let $\Sigma_t$ be a family of smooth, embedded anisotropic capillary hypersurfaces in $\overline{\mathbb{R}^{n+1}_+}$ with boundary supported on $\partial\mathbb{R}^{n+1}_+$, which are given by $X(\cdot,t):M\rightarrow\mathbb{R}^{n+1}_+$, and satisfy $\operatorname{int}(\Sigma_t)=X(\operatorname{int}(M),t)\subset \mathbb{R}_+^{n+1}, \partial\Sigma_t=X(\partial M,t)\subset\partial\mathbb{R}^{n+1}_+$.
We study the following locally constrained anisotropic mean curvature flow for anisotropic capillary hypersurfaces defined by
\begin{equation}\label{equ:flow}
	\left\{
	\begin{array}{ll}
		\partial_tX(\cdot,t)=f\nu_F(\cdot,t)+T(\cdot,t), & \text{in}\ M\times[0,T),
		\\[3pt]
		\<\Psi(\nu(X(\cdot,t))),E_{n+1}\>=-\omega_0, & \text{on}\ \partial M\times [0,T),
		\\[3pt]
		X(\cdot,0)=X_0(\cdot) &\text{in}\ M,
	\end{array}	
	\right.
\end{equation}
where $f=nF(\nu)^{-1}\(
F(\nu)+\omega_0\<\nu,E^F_{n+1}\>-\<X,\nu\>H^F_1
\)$ and $T\in T_X\Sigma_t$ is added such that the hypersurfaces preserve the capillary condition along this flow. By \eqref{equ:G(vF,Y)=<v,Y>/F}, $f$ can be also  written as
\begin{align}\label{equ:pfL3.2-4}
	f=n
	+n\omega_0G(\nu_F)(\nu_F,E^F_{n+1})
	-\hat{u}H_F,
\end{align}
where 
the new metric $G(\nu_F(X))$ is defined in Section \ref{sec 2}.
\begin{theorem}\label{thm:flow}
	Let $\omega_0 \in\left(-F\left(E_{n+1}\right), F\left(-E_{n+1}\right)\right)$. If the initial anisotropic $\omega_0$-capillary hypersurface $\Sigma_0$ is star-shaped with respect to the origin in  $\overline{\mathbb{R}_{+}^{n+1}}$, then flow \eqref{equ:flow} has a unique smooth solution for all time. Moreover, the evolving hypersurface smoothly converges to an $\omega_0$-capillary Wulff shape ${\W_{r_0,\omega_0}}$, for some $r_0$ determined by ${\Vol(\widehat{\W_{r_0,\omega_0}})}= \Vol(\widehat{\Sigma_0})$ as $t\rightarrow \infty $.
\end{theorem}

Geometric inequality is one of the fundamental topics in differential geometry and convex  geometry. The classical isoperimetric inequality in the Euclidean space says that among all bounded domains in $\mathbb{R}^{n+1}$ with fixed enclosed volume, the minimum of the area functional is achieved by round balls. 
Some geometric inequalities can be proved by using  curvature flows,
e.g. for closed hypersurfaces in \cite{Guan-Li-2015}, and for anisotropic analogue in \cite{Wei-Xiong2021}.
In \cite{Ding-Li-2022},  the first and third authors  extended the Alexandrov-Fenchel inequalities in \cite{Guan-Li-2015}.
In \cite{Gao-Li2024},  the second and third authors obtained  the generalization of anisotropic Alexandrov-Fenchel type inequalities,  which extended the  inequalities in \cite{Wei-Xiong2021,Ding-Li-2022}.
In \cite{Wang-Xia-2022,Wang-Weng-2020}, a class of isoperimetric inequalities for hypersurface with capillary boundary in a ball is obtained by the volume-preserving mean curvature type flows. In \cite{Hu-Wei-yang-zhou2023,Mei-Wang-Weng}, it is proved that the mean curvature type flow  preserves the convexity under some range of $\theta$, and then the corresponding Alexandrov-Fenchel inequalities are also obtained. In order to obtain Alexandrov-Fenchel inequalities for anisotropic capillary hypersurfaces, we need to prove that the flow \eqref{equ:flow} preserves the convexity. 

 Due to anisotropy, some classical approaches to estimating the positive lower bound of principal curvatures fail. Motivated by \cite{Xia-2017-convex,Wei-Xiong2022}, we consider using the inverse  anisotropic Gauss map to reparameterize the flow, and then prove the upper bound of the anisotropic principal curvature radii. For isotropic case, the corresponding region  of the reparameterized flow equation is a spherical cap determined by $\omega_0$ (i.e.  $-\cos\theta$). However, for anisotropic case, the corresponding region  of the reparameterized flow equation  is not as regular as the spherical cap, since different Wulff shapes lead to different images of inverse Gaussian maps.
 The region with boundary is determined by $\omega_0$ and  Wulff shape $\W$, since different Wulff shapes can cause different initial boundary value equations. So the necessary condition of convexity preservation may require constraints on $\omega_0$ and Wulff shape $\W$. Inspired by the formula $\eqref{equ:haan}$, we naturally pose the following condition to process boundary estimation: 
 \begin{condition}\label{condition}
 	We assume that
 \begin{align}
 	\label{equ:condition}	\omega_0\leq \frac{Q(z)(Y,Y,A_F(\nu(z))\mu(z))\cdot \<E_{n+1},\mu(z)           \> \cdot F(\nu)}{G(z)(Y,Y)},
 \end{align}
  for all $z\in\W\cap \{x_{n+1}=-\omega_0\}$ and $Y(\neq 0)\in T_{z}\( \W\cap \{x_{n+1}=-\omega_0\} \)$, where $\nu(z)$ is the unit outward normal vector of  $\W $ in $\mathbb{R}^{n+1}$,  $\mu(z)$ is the unit outward co-normal of $\W\cap \{x_{n+1}=-\omega_0\}$ in $\W$,  $A_F$ is the matrix defined by \eqref{equ:A_F>0}  which acts on linear space $T_{\nu(z)}\mathbb{S}^n=T_z\W$, $Q$ is a $(0,3)$-tensor defined by \eqref{equ:Q}, and $\{x_{n+1}=-\omega_0\}:=\{x=(x_1,\cdots,x_{n+1})\in \mathbb{R}^{n+1}: x_{n+1}=-\omega_0\}$.
 \end{condition}
 \begin{remark}\label{rk1}
 	(1) 
 	  Inequality \eqref{equ:condition}
 	  only depends on  Wulff shape $\W$ (or $F$) and $\omega_0$, which means that this condition only imposes some restrictions of the Wulff shape on the sufficiently small neighborhood around the hyperplane $\{x_{n+1}=-\omega_0\}$. Besides, we can see that $F(z), G(z)(Y,Y)$, and $\<-E_{n+1},\mu(z)\>$ are all positive. 
 	
 	(2) If  $\W_{1,\omega_0}$  is an isotropic $\theta$-capillary hypersurface (see Definition  \ref{def:capillary-condex-body}), its Gauss map and that of the anisotropic $\omega_0$-capillary hypersurface $\Sigma$ both have their images  $\mathbb{S}_\theta$, a spherical cap with a constant contact angle $\theta$. 
 	Then the flow can be formulated as an initial boundary value problem on this special domain.  
 	Lemma \ref{lemma:7.2} demonstrates that, Condition \ref{condition} reduces to the requirement $\theta \in (0, \frac{\pi}{2}]$ in this special case.
 	
 	In particular, if $\W=\mathbb{S}^n$, then $\W_{1,\omega_0}=\(\mathbb{S}^{n}+\omega_0E_{n+1}\)\cap\overline{\mathbb{R}^{n+1}_+}$ is a $\theta$-capillary hypersurface, with $\omega_0=-\cos\theta$ and  $Q\equiv 0$. Inequality \eqref{equ:condition}  is equivalent to $\omega_0\leq 0$, that means $\theta\in(0,\frac{\pi}{2}]$.
 	
 	If $\W$ is  any given convex body, then $\W_{1,\omega_0}$ is just the general hypersurface in $\overline{\mathbb{R}_+^{n+1}}$. On its boundary,  the contact angle is not constant and may vary, and the Gauss mapping image $\Sigma$ may exhibit a complicated structure. To address this, certain restrictions are imposed near the boundary of $\W_{1,\omega_0}$, as specified in Condition \ref{condition}.

 	(3) We list some  Wulff shapes and  the corresponding range of values for $\omega_0$ which satisfy Inequality \eqref{equ:condition} in Appendix \ref{Appendix} (see Example \ref{exam:x^4} (2),  Remark \ref{rk:A.2}, Example \ref{exam:x^4T}, and Remark \ref{rk:A.3}).
  	
 	
 	(4) Inequality \eqref{equ:condition} is equivalent to $\tilde{Q}(\tilde{z})\(Y,Y,A_F(\nu(z))\mu(z)\)\leq0$ by Lemma \ref{lem:Q-translate}, where $\tilde{Q}$ is the corresponding $(0,3)$-tensor of the  translated  Wulff shape $\widetilde{\W}=\W+\omega_0E_{n+1}^F$, $\tilde{z}\in\widetilde{\W}\cap\{x_{n+1}=0\}$. More details about the  translated  Wulff shape and related geometric quantities
 	 can be seen in Section \ref{subsec 5.2}.
 \end{remark}
  Due to the extra term $\omega_0G(\nu_{{F}})(\nu_{{F}},E_{n+1}^F)$ in  flow \eqref{equ:flow}, we cannot directly apply the method in  \cite{Wei-Xiong2022} to prove convexity preservation, even with the additional condition of \eqref{equ:condition}. We utilize a new method to deal with this problem then prove the convexity preservation.
\begin{theorem}\label{thm:convex preservation}
 Assume that $\omega_0 \in\left(-F\left(E_{n+1}\right), F\left(-E_{n+1}\right)\right)$, and  Condition \ref{condition} holds.
If the initial anisotropic $\omega_0$-capillary hypersurface $\Sigma_0$ is strictly convex, then the solution of  flow \eqref{equ:flow} is strictly convex for all positive time.
\end{theorem}

Next, we shall prove that the enclosed volume $\mathcal{V}_{0,\omega_0}(\Sigma_t):=|\widehat{\Sigma}_t|$ is preserved and the anisotropic capillary area (or free energy functional \cite{Koiso2023}) $\mathcal{V}_{1,\omega_0}(\Sigma_t)$ (defined by \eqref{equ:v1}) is monotonic along flow \eqref{equ:flow}, then derive the following anisotropic capillary isoperimetric inequality for star-shaped hypersurfaces.

\begin{theorem}\label{thm:iso-neq}
	Let $ \omega_0 \in\left(-F\left(E_{n+1}\right), F\left(-E_{n+1}\right)\right), n\geq 2$, and $\Sigma\subset\overline{\mathbb{R}^{n+1}_+}$ be a star-shaped anisotropic $\omega_0$-capillary hypersurface, then there holds
	\begin{align}\nonumber
		\frac{|\Sigma|_F+\omega_0|\widehat{\partial\Sigma}|}{|\W_{1,\omega_0}|_F+\omega_0|\widehat{\partial\W_{1,\omega_0}}|}
		\geq
		\left(\frac{|\widehat{\Sigma}|}{|\widehat{\W_{1,\omega_0}}|}\right)^{\frac{n}{n+1}},
	\end{align}
	where $\widehat{\partial\Sigma}=\widehat{\Sigma}\cap \partial\mathbb{R}^{n+1}_+$ denotes the domain enclosed  by $\partial\Sigma\subset\mathbb{R}^n=\partial\mathbb{R}^{n+1}_+$, $|\cdot |$ denotes the volume of the domain in $ \mathbb{R}^n$ or $ \mathbb{R}^{n+1}$, and $|\cdot |_F$ denotes the anisotropic area defined by \eqref{equ:F-area}.
	Moreover, equality holds if and only if $\Sigma$ is an $\omega_0$-capillary Wulff shape, and $|\W_{1,\omega_0}|_F+\omega_0|\widehat{\partial\W_{1,\omega_0}}|={(n+1)}\mathcal{V}_{1,\omega_0}({\W_{1,\omega_0}})=(n+1)\Vol(\widehat{\W_{1,\omega_0}}):=(n+1)|\widehat{\W_{1,\omega_0}}|$ (by Lemma \ref{lemma:7.2} (i)).
\end{theorem}
Recently, Wang, Weng, and Xia \cite{Wang-Weng-Xia} constructed some suitable quermassintegrals for capillary hypersurface in the half-space. For anisotropic case, we are inspired to define the anisotropic capillary quermassintegrals $\mathcal{V}_{k,\omega_0}(\Sigma_t)$ in Section \ref{sec 6}. Combining the monotonicity of $\mathcal{V}_{k,\omega_0}(\Sigma_t)$ derived in Theorem  \ref{Thm6}, we have the new Alexandrov-Fenchel inequalities as follows.
\begin{theorem}
		\label{thm:AF-neq}
	Let $ \omega_0 \in\left(-F\left(E_{n+1}\right), F\left(-E_{n+1}\right)\right), n\geq 2$, and $\Sigma\subset\overline{\mathbb{R}^{n+1}_+}$ be a strictly convex anisotropic $\omega_0$-capillary hypersurface. If Condition \ref{condition} holds, then there holds
	\begin{align*}
		\left(\frac{\mathcal{V}_{k, \omega_0}({\Sigma})}{\mathcal{V}_{k, \omega_0}( \W_{1,\omega_0} )}\right) ^{\frac{1}{n+1-k}}
		\geq
		\left(\frac{\mathcal{V}_{0, {\omega_0}}({\Sigma})}{\mathcal{V}_{0, \omega_0}( \W_{1,\omega_0} )}\right)^{\frac{1}{n+1}}, \quad 1 \leq k< n
		,
	\end{align*}
		where equality holds if and only if $\Sigma$ is an $\omega_0$-capillary Wulff shape, and  $\mathcal{V}_{k, \omega_0}( \W_{1,\omega_0} )=\mathcal{V}_{0, \omega_0}( \W_{1,\omega_0} )=\Vol(\widehat{\W_{1,\omega_0}})$ (by Lemma \ref{lemma:7.2} (i)).
\end{theorem}

\begin{remark}
	
(1) If we take $\W=\mathbb{S}^n$, Theorem \ref{thm:iso-neq} matches the capillary isoperimetric inequality in \cite{Mei-Wang-Weng}*{Corollary 1.3}.

(2) If we take $\W=\mathbb{S}^n$, Theorem \ref{thm:AF-neq} matches the  Alexandrov-Fenchel inequalities in \cite{Mei-Wang-Weng}*{Theorem 1.4}, since Remark \ref{rk1} (2).

(3) The Wulff shape of the isoperimetric inequality in Theorem \ref{thm:iso-neq} does not need to satisfy Condition \ref{condition}, and  $\hat{\theta}\in(0,\pi)$.

\end{remark}

Mei, Wang, Weng, and Xia  studied the  theory for capillary convex bodies in the half-space for the first time in \cite{Xia-arxiv}, where the capillary convex bodies, the  Minkowski  sum in $\mathcal{K}_{\theta}$, and the mixed volume  $V(K_1,\cdots,K_{n+1})$ for $K_i\in\mathcal{K}_{\theta}$  as given in formula \eqref{equ:VKKKLLL} were introduced, and the general Alexandrov-Frenchel inequalities (see \cite{Xia-2017-convex}*{Theorem 1.1}) were proved from the view of point in convex geometry. Here $\mathcal{K}_{\theta}$ denotes the set of all $\theta$-capillary convex bodies (see Section \ref{sec 7}).  Specifically, the isotropic capillary Alexandrov-Frenchel inequalities which involve strictly convex $\theta$-capillary hypersurface ${\Sigma}$ and capillary spherical cap $\mathcal{C}_{\theta}$ were obtained. 

Inspired by \cite{Xia-arxiv}, we derive  Lemma \ref{lemma:7.2}, which further explains and expresses the anisotropic geometric quantity   $\mathcal{V}_{k,\omega_0}$  defined in Section \ref{sec 6}. For a general Wulff shape, the image of  Gauss map of $\W_{1,\omega_0}$ may be an irregular domain $\mathcal{A}$ on $\mathbb{S}^n$,  corresponding to a wider range of geometric problems. 
Specifically, if $\W_{1,\omega_0}$ is an isotropic $\theta$-capillary hypersurface, then Condition \ref{condition} is equivalent to $\theta\in(0,\frac{\pi}{2}]$ (Lemma \ref{lemma:7.1}), and Theorem \ref{thm:AF-neq} implies the following inequalities \ref{equ:cor-7.3} (concerning  $\theta$-capillary hypersurfaces $\partial K$ and $\partial L$, $\partial L$ does not need to be  the capillary spherical cap). Theorem \ref{thm:iso-neq} implies the inequality \eqref{equ:cor7.3-1} which we call  the capillary $L_1$-Minkowski inequality for non-convex case.

\begin{corollary}
	\label{cor:7.3}
	(i) Let $\theta\in(0, \pi)$ and $ n\geq 2$. Given a star-shaped $\theta$-capillary hypersurface $\Sigma$ with $K=\widehat{\Sigma}$ and a capillary convex body  $L\in\mathcal{K}_{\theta}$, we have
	\begin{align}\label{equ:cor7.3-1}
		\int_{\Sigma}\ell(\nu(X)) {\rm d}\mu_g\geq(n+1) \Vol(K)^{\frac{n}{n+1}}\Vol(L)^{\frac{1}{n+1}}
		,
	\end{align}
		where $\ell(z)=F(z)+\omega_0\<z,E_{n+1}^F\>\ (z\in\mathbb{S}^n)$ is the  support function of $\mathcal{C}_{\theta}$.
	Equality holds if and only if there exists a constant $c\in\mathbb{R}^+$ and a constant vector $b\in\partial\overline{\mathbb{R}^{n+1}_+}$ such that  $K=cL+b$.
	
	(ii) Let $\theta\in(0, \frac{\pi}{2}]$ and $ n\geq 2$. Given two capillary convex bodies $K,L\in\mathcal{K}_{\theta}$, 
	we have
	\begin{align}\label{equ:cor-7.3}
		V\(\underbrace{K, \cdots, K}_{(n-k) \text { copies }}, \underbrace{L, \cdots, L}_{(k+1) \text { copies }}\)\geq \Vol(K)^{\frac{n-k}{n+1}}\Vol(L)^{\frac{k+1}{n+1}}, \quad k=0,\cdots,n-1.
	\end{align}
	Equality holds if and only if there exists a constant $c\in\mathbb{R}^+$ and a constant vector $b\in\partial\overline{\mathbb{R}^{n+1}_+}$ such that  $K=cL+b$.
\end{corollary}

\begin{remark}
	Inequality \eqref{equ:cor-7.3}  corresponds to the inequality involving two different convex bodies in \cite{Xia-arxiv}. If we denote by $\mathcal{K}_{\omega_0}$ the class of anisotropic $\omega_0$-capillary convex bodies, then we may define  mixed volumes for 	
	 $m$ different anisotropic $
	 \omega_0$-capillary convex bodies. 
	 We conjecture that Alexandrov-Fenchel inequalities analogous to \cite{Xia-arxiv}*{Theorem 1.1} hold in this case by Theorem \ref{thm:iso-neq} and \ref{thm:AF-neq}.
\end{remark}
\begin{conjecture}
	Let $\widehat{\Sigma}_i \in \mathcal{K}_{\omega_0}$ for $1 \leq i \leq n+1$. Then
$$
V^2\left(\widehat{\Sigma}_1, \widehat{\Sigma}_2, \widehat{\Sigma}_3, \cdots, \widehat{\Sigma}_{n+1}\right) \geq V\left(\widehat{\Sigma}_1, \widehat{\Sigma}_1, \widehat{\Sigma}_3, \cdots, \widehat{\Sigma}_{n+1}\right) V\left(\widehat{\Sigma}_2, \widehat{\Sigma}_2, \widehat{\Sigma}_3, \cdots, \widehat{\Sigma}_{n+1}\right),
$$
with equality holding if and only if there exists a constant $c\in\mathbb{R}^+$ and a constant vector $b\in\partial\overline{\mathbb{R}^{n+1}_+}$ such that  $\widehat{\Sigma}_1=c\widehat{\Sigma}_2+b$.



\end{conjecture}

The rest of this paper is organized as follows. In Section \ref{sec 2}, we briefly introduce some preliminaries on the anisotropic capillary hypersurface. In Section \ref{sec 3}, we first obtain the uniform estimates for flow \eqref{equ:flow} and prove its long-time existence and smooth convergence, then we complete the proof of Theorem \ref{thm:flow}.
In Section \ref{sec 4}, we derive some monotonic quantities $\mathcal{V}_{0,\omega_0},\mathcal{V}_{1,\omega_0}$ and show the anisotropic capillary isoperimetric inequality (i.e., Theorem \ref{thm:iso-neq}).
In Section \ref{sec 5}, we prove the flow preserves the convexity under Condition \ref{condition} (i.e., Theorem \ref{thm:convex preservation}).
 In Section \ref{sec 6}, we construct the general anisotropic capillary quermassintegrals $\mathcal{V}_{k,\omega_0}$, and show their monotonicity, then prove the Alexandrov-Fenchel inequalities in Theorem \ref{thm:AF-neq}.
 In Section \ref{sec 7}, we mainly consider the case when $\W_{1,\omega_0}$ is an isotropic $\theta$-capillary hypersurface and prove  Corollary \ref{cor:7.3}.
   In Appendix \ref{Appendix}, we study the relationship between some geometric quantities of the Wulff shape and the translated Wulff shape,  which will be used in the proof of  convexity preserving result in Section \ref{subsec 5.2}. We also provide some specific examples, which are related to  Condition \ref{condition}.

\section{Preliminary}\label{sec 2}
Some preliminaries on the anisotropic geometry are briefly introduced in this section, for more details please refer to \cite{Xia-phd}.
\subsection{The Wulff shape and dual Minkowski norm}
Let $F$ be a smooth positive function on the standard sphere $(\mathbb{S}^n, g^{\mathbb{S}^n} ,\nabla^{\mathbb{S}})$ such that the matrix
\begin{equation}\label{equ:A_F>0}
A_{F}(x)~=~\nabla^{\mathbb{S}}\nabla^{\mathbb{S}} {F}(x)+F(x)g^{\mathbb{S}^n}, 
\quad x\in \mathbb{S}^n,
\end{equation}
is positive definite on $  \mathbb{S}^n$,
where 
 $g^{\mathbb{S}^n}$ denotes the round metric on $\mathbb{S}^n$.
  Then there exists a unique smooth strictly convex hypersurface $\W$ defined by
\begin{align*}
\W=\{\Psi(x)|\Psi(x):=F(x)x+\nabla^{\mathbb{S}} {F}(x),~x\in \mathbb{S}^n\},
\end{align*}
whose support function is given by $F$. We call $\W$ the Wulff shape determined by the function $F\in C^{\infty}(\mathbb{S}^n)$. When $F$ is a constant, the Wulff shape is just a round sphere.

The smooth function $F$ on $\mathbb{S}^n$ can be extended  to a $1$-homogeneous function on $\mathbb{R}^{n+1}$ by
\begin{equation*}
F(x)=|x|F(\frac{x}{|x|}), \quad x\in \mathbb{R}^{n+1}\setminus\{0\},
\end{equation*}
and setting $F(0)=0$. Then it is easy to show that $\Psi(x)=DF(x)$ for $x\in \mathbb{S}^n$, where $D$ denotes the standard gradient on $\mathbb{R}^{n+1}$.


The homogeneous extension $F$ defines a Minkowski norm on $\mathbb{R}^{n+1}$, that is, $F$ is a norm on $\mathbb{R}^{n+1}$ and $D^2(F^2)$ is uniformly positive definite on $\mathbb{R}^{n+1}\setminus\{0\}$. We can define a dual Minkowski norm $F^0$ on $\mathbb{R}^{n+1}$ by
\begin{align*}
F^0(\xi):=\sup_{x\neq 0}\frac{\langle x,\xi\rangle}{{F}(x)},\quad \xi\in \mathbb{R}^{n+1}.
\end{align*}
We call $\W$ the unit Wulff shape since
 $$\W=\{x\in \R^{n+1}: F^0(x)=1\}.$$
A Wulff shape of radius $r_0$ centered at $x_0$ is given by
\begin{align}\label{equ:w_r(x)}
	\W_{r_0}(x_0)=\{x\in\mathbb{R}^{n+1}:F^0(x-x_0)=r_0\}.
\end{align}
An $\omega_0$-capillary Wulff shape of radius $r_0$ is given by
\begin{align}\label{equ:crwo}
	\W_{r_0,\omega_0}(E_{n+1}^F):=\{x\in\overline{\mathbb{R}_+^{n+1}}:F^0(x-r_0\omega_0E_{n+1}^F)=r_0\},
\end{align}
 which is a part of a Wulff shape cut by a hyperplane  $\{x_{n+1}=0\}$, here $E_{n+1}^F$ satisfies $\<E_{n+1}^F,E_{n+1}\>=1$ which guarantees that  $\W_{r_0,\omega_0}(E_{n+1}^F)$ satisfies capillary condition \eqref{equ:w0-capillary}. If there is no confusion, we just write $\W_{r_0,\omega_0}:=\W_{r_0,\omega_0}(E^F_{n+1})$ in this paper.

\subsection{Anisotropic curvature}
Let $(\Sigma,g,\nabla)\subset \overline{\mathbb{R}^{n+1}_+}$ be a $C^2$ hypersurface with $\partial\Sigma\subset\partial\mathbb{R}^{n+1}_+$, which encloses a bounded domain  $\widehat{\Sigma}$. Let $\nu$ be the unit normal of $\Sigma$ pointing outward of  $\widehat{\Sigma}$.
The anisotropic Gauss map of $\Sigma$  is defined by $$\begin{array}{lll}\nu_F: &&\Sigma\to  \W\\
&&X\mapsto \Psi(\nu(X))=F(\nu(X))\nu(X)+\nabla^\mathbb{S} F(\nu(X)).\end{array} $$
The anisotropic principal curvatures $\k^F=(\k^F_1,\cdots, \k^F_n)$ of $\Sigma$ with respect to $\W$ at $X\in \Sigma$  are defined as the eigenvalues of
 \begin{align*}
S_F=\mathrm{d}\nu_F=\mathrm{d}(\Psi\circ\nu)=A_F\circ \mathrm{d}\nu : T_X \Sigma\to T_{\nu_F(X)} \W=T_X \Sigma.
\end{align*}
The anisotropic principal curvature radii $\tau=(\tau_1,\cdots, \tau_n)$ of $\Sigma$ with respect to $\W$ at $X\in \Sigma$  are defined as $\tau_i=(\kappa^F_i)^{-1}$, if $\kappa^F_i\ne 0$, $i=1,\cdots,n$.

We define the normalized $k$-th elementary symmetric function $H^F_k=\sigma_k(\kappa^F)/\binom{n}{k}$ of the anisotropic principal curvature $\kappa^F$:
\begin{align}\label{equ:Hk}
H^F_k:=\binom{n}{k}^{-1}\sum_{1\leq {i_1}<\cdots<{i_k}\leq n} \kappa^F_{i_1}\cdots \kappa^F_{i_k},\quad k=1,\cdots,n,
\end{align}
where $\binom{n}{k}=\frac{n!}{k!(n-k)!}$.
Setting $H^F_{0}=1$, $H^F_{n+1}=0$, and $H_F=nH_1^F$
for convenience. To simplify the notation, from now on, we use  $\sigma_k$ to denote function $\sigma_k(\kappa^F)$ of the anisotropic principal curvatures, if there is no confusion.

Next, we introduce two important operators $P_k$ and $T_k$ which will be used in Section \ref{sec 6},
\begin{align} 
	& P_k=\sigma_k I-\sigma_{k-1} S_F+\cdots+(-1)^k S_F^k, \quad k=0, \cdots, n, \nonumber\\
	& T_k=P_k A_F, \quad k=0, \cdots, n-1.\label{equ:T_k=P_kAF}
\end{align}
Obviously, $P_n=0$, $ P_k=\sigma_k I-P_{k-1} S_F=\sigma_k I-T_{k-1} \mathrm{~d} \nu,\ k=1,\cdots, n$, and $T_k$ is symmetric for each $k$.
He and Li in \cite{HL08} proved that
\begin{lemma} For each $0\leq k\leq n$, there holds	
	\begin{align}
		\operatorname{div}\left(P_k\left(\nabla^{\mathbb{S}} F\right) \circ \nu\right)+F\left(\nu\right) \operatorname{tr}\left(P_k \circ \mathrm{d} \nu\right)=(k+1) \sigma_{k+1},
		\label{equ:lemma2.1.1}
		\\
		\operatorname{div}\left(P_k X^{\top}\right)+\left\langle X, \nu\right\rangle \operatorname{tr}\left(P_k \circ \mathrm{d} \nu\right)=(n-k) \sigma_k,
		\nonumber
	\end{align} 	
	where $X^{\top}=X- \left\langle X, \nu\right\rangle \nu$, and 
	$\operatorname{div}$ is divergence operator respect to metric $g$.
	
\end{lemma}

\subsection{New metric and anisotropic formulas}\label{subsec 2.3}
There are new metric on $\Sigma$ and $\mathbb{R}^{n+1}$, which were introduced by Andrews in \cite{And01} and reformulated by Xia in \cite{Xia13}.  This new Riemannian metric $G$ with respect to $F^0$ on $\mathbb{R}^{n+1}$ is defined as
\begin{align}
	\label{equ:G}
	G(\xi)(V,W):=\sum_{\alpha,\beta=1}^{n+1}\frac{\partial^2 \frac12(F^0)^2(\xi)}{\partial \xi^\alpha\partial \xi^\beta} V^\alpha W^\beta, \quad\hbox{ for } \xi\in \mathbb{R}^{n+1}\setminus \{0\}, V,W\in T_\xi{\mathbb{R}^{n+1}}.
\end{align}
The third order derivative of $F^0$ gives a $(0,3)$-tensor
\begin{equation}
	\label{equ:Q}
	Q(\xi)(U,V,W):=\sum_{\alpha,\beta,\gamma=1}^{n+1} Q_{\alpha\beta\gamma}(\xi)U^\alpha V^\beta W^\gamma:=\sum_{\alpha,\beta,\gamma=1}^{n+1} \frac{\partial^3(\frac12(F^0)^2(\xi)}{\partial \xi^\alpha \partial \xi^\beta \partial \xi^\gamma}U^\alpha V^\beta W^\gamma,
\end{equation}
for $\xi\in \mathbb{R}^{n+1}\setminus \{0\},$ $U,V,W\in T_\xi{\mathbb{R}^{n+1}}.$

When we restrict the metric $G$ to $\mathcal{W}$,  the $1$-homogeneity of $F^0$ implies that
\begin{eqnarray*}
	&G(\xi)(\xi,\xi)=1,  \quad G(\xi)(\xi, V)=0, \quad \hbox{ for } \xi\in \W,\  V\in T_\xi \W.
	\\&Q(\xi)(\xi, V, W)=0, \quad \hbox{ for } \xi\in\W,\  V, W\in \mathbb{R}^{n+1}.
\end{eqnarray*}
For a smooth hypersurface $\Sigma$ in $\overline{\mathbb{R}_+^{n+1}}$, since $\nu_F(X)\in \W$ for $X\in \Sigma$, we have
\begin{align}
	&G(\nu_F)(\nu_F,\nu_F)=1, \quad G(\nu_F)(\nu_F, V)=0, \quad \hbox{ for } V\in T_X \Sigma,\nonumber
	\\
	&Q(\nu_F)(\nu_F, V, W)=0, \quad \hbox{ for } V, W\in \mathbb{R}^{n+1}.\label{equ:Q_0}
\end{align}	
This means  $\nu_F(X)$ is perpendicular to $T_X \Sigma$ with respect to the metric $G(\nu_F)$.
Then the Riemannian metric $\hat{g}$ on $\Sigma$  can be defined as
\begin{eqnarray*}
	\hat{g}(X):=G(\nu_F(X))|_{T_X \Sigma}, \quad X\in \Sigma.
\end{eqnarray*}

We denote by $\hat{D}$ and $\hat{\nabla}$ the Levi-Civita connections of $G$ on $\mathbb{R}^{n+1}$ and $\hat{g}$ on $\Sigma$ respectively,
denote by $\hat{g}_{ij}$ and $\hat{h}_{ij}$ the first and second fundamental form of $(\Sigma, \hat{g})\subset (\mathbb{R}^{n+1}, G)$, that is
$$\hat{g}_{ij}=G(\nu_F(X))(\p_i X, \p_j X),\quad \hat{h}_{ij}=G(\nu_F(X))(\hat{D}_{\p_i }\nu_F, \p_j X).$$
Denote by $\{\hat{g}^{ij}\}$ the inverse matrix of $\{\hat{g}_{ij}\}$, we can reformulate $\k^F$ as  the  eigenvalues of $\{\hat{h}_j^k\}=\{\hat{g}^{ik}\hat{h}_{kj}\}$.
It is direct to see that for $\Sigma=\W$, we have $\nu_F(\W)=X(\W)$, $\hat{h}_{ij}=\hat{g}_{ij}$ and $\kappa^F=(1,\cdots,1)$.
The anisotropic Gauss-Weingarten type formula and the anisotropic Gauss-Codazzi type equation are as follows.
\begin{lemma}[\cite{Xia13}*{Lemma 2.5}] \label{lem2-1}
	\begin{eqnarray}
		\partial_i\partial_j X=-\hat{h}_{ij}\nu_F+\hat{\nabla}_{\partial_i } \partial_j+\hat{g}^{kl}A_{ijl}\partial_kX; \;\;\;\hbox{ (Gauss formula)}\label{equ:Gauss-formula}
		\end{eqnarray}
		\begin{eqnarray}\label{Weingarten}
		\partial_i \nu_F=\hat{g}^{jk}\hat{h}_{ij}\partial_k X;\;\; \hbox{ (Weingarten formula) }
	\end{eqnarray}
		 $$\hat{R}_{i j k \ell}=\hat{h}_{i k} \hat{h}_{j \ell}-\hat{h}_{i \ell} \hat{h}_{j k}+\hat{\nabla}_{\partial_{\ell}} A_{j k i}-\hat{\nabla}_{\partial_k} A_{j \ell i}+A_{j k}^m A_{m \ell i}-A_{j \ell}^m A_{m k i} ;\quad	\text{(Gauss equation)}$$
	\begin{eqnarray}\label{Codazzi}
		\hat{\nabla}_k\hat{h}_{ij}+\hat{h}_j^lA_{lki}=\hat{\nabla}_j\hat{h}_{ik}+\hat{h}_k^lA_{lji}. \;\;\hbox{ (Codazzi equation) }
	\end{eqnarray}
	Here, $\hat{R}$ is the Riemannian curvature tensor of $\hat{g}$ and  $A$ is a $3$-tensor
	\begin{eqnarray}\label{AA}
	A_{ijk}=-\frac12\left(\hat{h}_i^l Q_{jkl}+\hat{h}_j^l Q_{ilk}-\hat{h}_k^l Q_{ijl}\right),
	\end{eqnarray} where $Q_{ijk}=Q(\nu_F)(\partial_i X, \partial_j X, \partial_k X)$. Note that the $3$-tensor $A$ on $(\Sigma,\hat{g})\to (\mathbb{R}^{n+1},G)$ depends on $\hat{h}_i^j$. It is direct to see that $Q$ is totally symmetric in all three indices, while $A$ is only symmetric for the first two indices.
\end{lemma}
We remark 
that,  $X$ and $\nu_F$ can be regarded as vector-valued functions in $\mathbb{R}^{n+1}$ with a fixed Cartesian coordinate. Hence, the terms $\partial_i\partial_j X$ and $\partial_i \nu_F$ are understood as the usual partial derivative on vector-valued functions. For convenience, we also denote $X_i=\partial_iX$ and $X^i=\hat{g}^{ij}X_j$.

The case $\Sigma=\mathcal{W}$ is the most important case which will be used in Section \ref{sec 5}. So we rewrite Lemma \ref{lem2-1} in this case.

\begin{lemma}[\cite{Xia-2017-convex}*{Proposition 2.2}]
	 Let $X \in \mathcal{W}$. Then $\nu_F(X)=X$ and $\hat{h}_{i j}=\hat{g}_{i j}$. Namely, the anisotropic principal curvatures of $\mathcal{W}$ are all equal to 1. We also have
$$
A_{i j k}=-\frac{1}{2} Q_{i j k}=-\frac{1}{2} Q(X)\left(\partial_i X, \partial_j X, \partial_k X\right),
$$
and
\begin{align}\label{equ:Qjikl-syms}
	\hat{\nabla}_i Q_{j k l}=\hat{\nabla}_j Q_{i k l} .
\end{align}
The Gauss-Weingarten formula and the Gauss equation are as follows:
\begin{align}
	\partial_i \partial_j X & =-\hat{g}_{i j} X+\hat{\nabla}_{\partial_i} \partial_j-\frac{1}{2} \hat{g}^{k l} Q_{i j l} \partial_k X ; \text { (Gauss formula) } \nonumber\\
	\partial_i \nu_F & =\partial_i X ; \text { (Weingarten formula) }\nonumber \\
	\hat{R}_{i j k l} & =\hat{g}_{i k} \hat{g}_{j l}-\hat{g}_{i l} \hat{g}_{j k}+\frac{1}{4} \hat{g}^{p m} Q_{j k p} Q_{m l i}-\frac{1}{4} \hat{g}^{p m} Q_{j l p} Q_{m k i}. \text { (Gauss equation) }\label{equ:R=}
\end{align}
\end{lemma}
\subsection{Anisotropic support function and  area element}
Let $u=\< X,\nu \> $ be the support function of $X$, and $\hat{u}:=G(\nu_F)(\nu_F,X)$ be the anisotropic support function of $X$ with respect to $F$(or $\W$). Then we have (see e.g. \cite{Xia2017})
\begin{equation}
	\label{u-hatu}
	\hat{u}=\frac{u}{F(\nu)}.
\end{equation}
And for anisotropic $\omega_0$-capillary hypersurface $\Sigma$, we call the function defined by \begin{align}\label{equ:u}
	\bar{u}(X)=\frac{\<X,\nu(X)\>}{F(\nu(X))+\omega_0\<\nu(X),E_{n+1}^F\>},  \quad X\in\Sigma,
\end{align}
the anisotropic capillary support function of anisotropic capillary hypersurface $\Sigma$ in $\overline{\mathbb{R}_{+}^{n+1}}$. It follows that $\bar{ u}$  is well-defined by
\cite{Jia-Wang-Xia-Zhang2023}*{Proposition 3.2}. Similar to \eqref{u-hatu}, since  $DF^0(DF(x))=\frac{x}{F(x)}$ (see e.g. \cite{Xia-phd}), we obtain
\begin{align}\label{equ:G(vF,Y)=<v,Y>/F}
	G(\nu_F)(\nu_F,Y)=\<DF^0(DF(\nu)),Y\>=\<\frac{\nu}{F(\nu)},Y\>=\frac{\<Y,\nu\>}{F(\nu)},
\end{align}
for any $Y\in \mathbb{R}^{n+1}$.
Then \eqref{equ:u} can also be written as
\begin{align*}
	\bar{u}(X)=\frac{\hat{ u}}{1+\omega_0 G(\nu_F)(\nu_F(X),E_{n+1}^F)},  \quad X\in\Sigma.
\end{align*}

Let us define the anisotropic area element as ${\rm d}\mu_F:=F(\nu){\rm d}\mu_g$, where ${\rm d}\mu_g$ is the induced volume form of the metric $g$. The anisotropic area of $\Sigma$ is given by
\begin{align}
	\label{equ:F-area}
	|\Sigma |_F=\int_{\Sigma} {\rm d}\mu_F.
\end{align}
 For any anisotropic $\omega_0$-capillary hypersurface $\Sigma$, we write the anisotropic capillary area of ${\Sigma}$ as
\begin{align}\label{equ:v1}
	\mathcal{V}_{1,\omega_0}({\Sigma})=\frac{1}{n+1}\left(
	|\Sigma|_F+\omega_0|\widehat{\partial\Sigma}|
	\right).
\end{align}
There holds the following anisotropic Minkowski integral formula for $\omega_0$-capillary hypersurface.
\begin{lemma}[Jia-Wang-Xia-Zhang, 2023, \cite{Jia-Wang-Xia-Zhang2023}*{Theorem 1.3}] \label{Thm1.2}
	Assume  $(\Sigma,g) \subset \overline{\mathbb{R}_{+}^{n+1}}$ is a $C^2$ compact anisotropic $\omega_0$-capillary hypersurface, where $\omega_0 \in\left(-F\left(E_{n+1}\right)\right.,$ $\left.F\left(-E_{n+1}\right)\right)$.
	 Then for  $0 \leq k \leq n-1$, it holds
	\begin{align}\label{equ:Minkow}
		\int_{\Sigma} H_{k}^F\left(1+\omega_0G(\nu_F)( \nu_F, E_{n+1}^F)\right)-H_{k+1}^F\hat{ u} \mathrm{~d}\mu_F=0.
	\end{align}
	
\end{lemma}


\section{Proof of Theorem \ref{thm:flow} }\label{sec 3}

In this section, we first derive some evolution equations and then obtain a priori estimates for flow \eqref{equ:flow}. Next, we prove the long-time existence and smooth convergence of flow \eqref{equ:flow}.

\subsection{Evolution equations}
\begin{lemma}\label{lemma-evolution-vf}
	Along the general flow $\partial_tX=f\nu_F+T$ with $T\in T_X\Sigma_t$, it holds that
\begin{align}\label{equ:p_tVF}
	\partial_t \nu_F=-\hat{\nabla}f+\hat{h}^{kj}G(\nu_F)(\partial_kX,T)\partial_jX,
\end{align}
\begin{align}\label{equ:ptu^}
	\partial_t\hat{ u}=f-G(\nu_F)(X,\hat{\nabla}f)
		+G(\nu_F)(T,\hat{\nabla}\hat{ u}),
\end{align}
\begin{align}\label{equ:ptG(v,E)}
	\partial_tG(\nu_F)(\nu_F,E^F_{n+1})=-G(\nu_F)(\hat{\nabla}f,E^F_{n+1})+\hat{h}^{kj}G(\nu_F)(T,\partial_kX)G(\nu_F)(E_{n+1}^F,\partial_jX).
\end{align}
\end{lemma}
\begin{proof}
	 Taking derivative of $G(\nu_F)(\nu_F,\nu_F)=1 $  and $G(\nu_F)(\nu_F,\partial_i X)=0 $ with respect to $t$, using  \eqref{equ:Q_0} and the Weingarten formula \eqref{Weingarten}, we have
	 \begin{align}
	 	0&=\partial_tG(\nu_F)(\nu_F,\nu_F)\nonumber
	 	\\
	 	&=2G(\nu_F)(\partial_t\nu_F,\nu_F)+Q(\nu_F)(\partial_t\nu_F,\nu_F,\nu_F)\nonumber
	 	\\
	 	&=2G(\nu_F)(\partial_t\nu_F,\nu_F),\label{equ:pf-Lem3.1-1}
	 \end{align}
	 and
	 \begin{align}
	 	0=&\partial_tG(\nu_F)(\nu_F,\partial_iX)\nonumber
	 	\\
	 	=&G(\nu_F)(\partial_t\nu_F,\partial_iX)+G(\nu_F)(\nu_F,\partial_t(\partial_iX))+Q(\nu_F)(\partial_t\nu_F,\nu_F,\partial_iX)\nonumber
	 	\\
	 	=&G(\nu_F)(\partial_t\nu_F,\partial_iX)+G(\nu_F)(\nu_F,\partial_i(f\nu_F+T))\nonumber
	 	\\
	 	=&G(\nu_F)(\partial_t\nu_F,\partial_iX)+\partial_i f+fG(\nu_F)(\nu_F,\partial_i\nu_F)+G(\nu_F)(\nu_F,\partial_iT).\label{equ:pf-Lem3.1-2}
	 \end{align}
	Similarly, since  $G(\nu_F)(\nu_F,T)=0$ and $G(\nu_F)(\nu_F,\nu_F)=1 $, we have
	  \begin{align}
	 	0&=\partial_iG(\nu_F)(\nu_F,T)\nonumber
	 	\\
	 	&=G(\nu_F)(\partial_i\nu_F,T)+G(\nu_F)(\nu_F,\partial_iT)+Q(\nu_F)(\partial_i\nu_F,\nu_F,T)\nonumber
	 	\\
	 	&=\hat{h}^k_i G(\nu_F)(\partial_kX,T)+G(\nu_F)(\nu_F,\partial_iT),\label{equ:pf-Lem3.1-3}
	 \end{align}
	 and
	 	 \begin{align}
	 	0&=\partial_iG(\nu_F)(\nu_F,\nu_F)\nonumber
	 	\\
	 	&=2G(\nu_F)(\partial_i\nu_F,\nu_F)+Q(\nu_F)(\partial_i\nu_F,\nu_F,\nu_F)\nonumber
	 	\\
	 	&=2G(\nu_F)(\partial_i\nu_F,\nu_F).\label{equ:pf-Lem3.1-4}
	 \end{align}
	 Putting \eqref{equ:pf-Lem3.1-3} and  \eqref{equ:pf-Lem3.1-4} into  \eqref{equ:pf-Lem3.1-2}, combining with  \eqref{equ:pf-Lem3.1-1}, we obtain
	 \begin{align*}	 	\partial_t\nu_F=&G(\nu_F)(\partial_t\nu_F,\nu_F)\nu_F+\hat{g}^{ij}G(\nu_F)(\partial_t\nu_F,\partial_iX)\partial_jX\\ =&-\hat{\nabla}f+\hat{h}^{kj}G(\nu_F)(\partial_kX,T)\partial_jX,
	 \end{align*}
	which proves \eqref{equ:p_tVF}.\\

	  By the evolution of $\partial_t\nu_F$, we can derive \eqref{equ:ptu^} by the following calculation
	 \begin{align*}
	 	\partial_t\hat{ u}=&\partial_tG(\nu_F)(X,\nu_F)
	 	\\
	 	=&G(\nu_F)(\partial_tX,\nu_F)+G(\nu_F)(X,\partial_t\nu_F)+Q(\nu_F)(\partial_t\nu_F,X,\nu_F)
	 	\\
	 	=&G(\nu_F)(f\nu_F+T,\nu_F)
	 	+G(\nu_F)(X,-\hat{\nabla}f+\hat{h}^{kj}G(\nu_F)(\partial_kX,T)\partial_jX)
	 	\\
	 	=&f-
	 	\hat{\nabla}^kfG(\nu_F)(X,\partial_kX)+\hat{h}^{kj}G(\nu_F)(T,\partial_kX)G(\nu_F)(X,\partial_jX)
	 	\\
	 	\overset{\eqref{equ:pfL3.2-1}}{=}&f-
	 	\hat{\nabla}^kfG(\nu_F)(X,\partial_kX)
	 	+G(\nu_F)(T,\hat{\nabla}\hat{ u}),
	 \end{align*}
	 where in the last equality we use the equation (see e.g. \cite{Xia2017})
	 \begin{align}
	 	\label{equ:pfL3.2-1}
	 	\hat{\nabla}_i\hat{ u}=\hat{h}_i^pG(\nu_F)(X,\partial_pX).
	 \end{align}
	 \\
	 For the last evolution \eqref{equ:ptG(v,E)}, we can also derive straightly
	 \begin{align*}
	 &\partial_tG(\nu_F)(E^F_{n+1},\nu_F)
	 	\\
	 	=&G(\nu_F)(\partial_tE^F_{n+1},\nu_F)+G(\nu_F)(E^F_{n+1},\partial_t\nu_F)+Q(\nu_F)(\partial_t\nu_F,E^F_{n+1},\nu_F)
	 	\\
	 	=&G(\nu_F)(E^F_{n+1},-\hat{\nabla}f+\hat{h}^{kj}G(\nu_F)(\partial_kX,T)\partial_jX)
	 	\\
	 	=&-
	 	\hat{\nabla}^kfG(\nu_F)(E^F_{n+1},\partial_kX)+\hat{h}^{kj}G(\nu_F)(T,\partial_kX)G(\nu_F)(E^F_{n+1},\partial_jX).
	 \end{align*}
	
\end{proof}

For simplicity, we introduce the linearized operator with respect to \eqref{equ:flow} as
\begin{align}\label{equ:L-operator}
	\mathcal{L}:=\partial_t-\hat{ u}\left(
	\hat{\Delta}+\hat{g}^{ik}A_{pik}\hat{\nabla}^p
	\right)
	-G(\nu_F)\(T+H_FX-n\omega_0E^F_{n+1},\hat{\nabla}\),
\end{align}
where $\hat{\Delta}$ is the Laplace operator with respect to $\hat{g}$.

\begin{lemma}\label{lemma-Lu^}
	Along flow \eqref{equ:flow}, the anisotropic support function $\hat{ u}$ satisfies
	\begin{align}\label{equ:Lu^}
		\mathcal{L}\hat{ u}=n+n\omega_0G(\nu_F)(\nu_F,E^F_{n+1})-2\hat{ u}H_F+\hat{ u}^2|\hat{h}|_{\hat{g}}^2,
	\end{align}
	where $|\hat{h}|_{\hat{g}}^2=\sum_{i,j}\hat{g}^{ij}\hat{h}_{ij}$.
\end{lemma}
\begin{proof}
It is proved in \cite{Xia2017} that
	\begin{align}
		\label{equ:pfL3.2-2}
		\hat{\Delta}\hat{ u}+\hat{g}^{ik}A_{pik}\hat{\nabla}^p\hat{ u}=\hat{\nabla}^pH_F G(\nu_F)(X,\partial_pX)
		-|\hat{h}|^2_{\hat{g}}\hat{ u}+H_F.
	\end{align}
By \eqref{Weingarten} and \eqref{equ:pfL3.2-4}, direct computation yields
\begin{align}
	\hat{\nabla}_kf&=\partial_kf=n\omega_0G(\nu_F)(\partial_k\nu_F,E^F_{n+1})-\partial_k(H_F\hat{ u})\nonumber
	\\
	&=n\omega_0\hat{h}_k^lG(\nu_F)(\partial_lX,E^F_{n+1})-\partial_kH_F\hat{ u}-\partial_k\hat{ u}H_F,\label{equ:pfL3.2-5}
\end{align}
which implies
\begin{align}
	&G(\nu_F)(\hat{\nabla}f,X)=\hat{\nabla}_kf G(\nu_F)(X^k,X)\nonumber
	\\
	=&n\omega_0\hat{h}_k^lG(\nu_F)(\partial_lX,E^F_{n+1}) G(\nu_F)(X^k,X)
	- G(\nu_F)(\partial_kH_FX^k,X)\hat{ u}\nonumber
	\\
	&- G(\nu_F)(\partial_k\hat{ u}X^k,X)H_F
	\nonumber
	\\
	\overset{\eqref{equ:pfL3.2-1}}{=}
	&n\omega_0G(\nu_F)(E^F_{n+1},\hat{\nabla}\hat{ u})-\hat{\nabla}^pH_F G(\nu_F)(X,X_p)\hat{ u}
	-G(\nu_F)(\hat{\nabla}\hat{ u},X)H_F\nonumber
	\\
	\overset{\eqref{equ:pfL3.2-2}}{=}&n\omega_0G(\nu_F)(E^F_{n+1},\hat{\nabla}\hat{ u})
	-G(\nu_F)(\hat{\nabla}\hat{ u},X)H_F
	\nonumber
	\\
	&
	-\left(
	\hat{\Delta}\hat{ u}+\hat{g}^{ik}A_{pik}\hat{\nabla}^p\hat{ u}
	+|\hat{h}|^2_{\hat{g}}\hat{ u}
	-H_F
	\right)\hat{ u}.
	\label{equ:pfL3.2-3}
\end{align}
Putting \eqref{equ:pfL3.2-4} and \eqref{equ:pfL3.2-3} into \eqref{equ:ptu^} yields
\begin{align*}
	\partial_t\hat{ u}=&
	n
	+n\omega_0G(\nu_F)(\nu_F,E^F_{n+1})
	-\hat{u}H_F
	+G(\nu_F)(T,\hat{\nabla}\hat{ u})
	\\
	&-n\omega_0G(\nu_F)(E^F_{n+1},\hat{\nabla}\hat{ u})
	+G(\nu_F)(\hat{\nabla}\hat{ u},X)H_F
	\\
	&
	+\left(
	\hat{\Delta}\hat{ u}+\hat{g}^{ik}A_{pik}\hat{\nabla}^p\hat{ u}
	+|\hat{h}|^2_{\hat{g}}\hat{ u}
	-H_F
	\right)\hat{ u}
	\\
	=&\hat{ u}\left(
	\hat{\Delta}\hat{ u}+\hat{g}^{ik}A_{pik}\hat{\nabla}^p\hat{ u}
	\right)
	+G(\nu_F)\(T+H_FX-n\omega_0E^F_{n+1},\hat{\nabla}\hat{ u}\)
	\\
	&+n+n\omega_0G(\nu_F)(\nu_F,E^F_{n+1})
	-2\hat{u}H_F
	+|\hat{h}|^2_{\hat{g}}\hat{ u}^2,
\end{align*}
which means
\begin{align*}
	\mathcal{L}\hat{ u}=n+n\omega_0G(\nu_F)(\nu_F,E^F_{n+1})-2\hat{ u}H_F+\hat{ u}^2|\hat{h}|_{\hat{g}}^2.
\end{align*}
\end{proof}

\begin{lemma}\label{lemma-Lu-}
	Along flow \eqref{equ:flow}, the anisotropic capillary support function $\bar{ u}$ satisfies
	\begin{align}\label{equ:Lu-}
		\mathcal{L}\bar{ u}=n-2\bar{ u}H_F+\bar{ u}^2
		|\hat{h}|_{\hat{g}}^2
		+
		2\bar{ u}G(\nu_F)\(\hat{\nabla}\bar{ u},\hat{\nabla}(1+\omega_0G(\nu_F)(\nu_F,E^F_{n+1}))\).
	\end{align}
\end{lemma}
\begin{proof}
	For convenience, we use  $R_{,i}=\hat{\nabla}_{\partial_i}R$ to represent the covariant derivative of tensor $R$. In normal coordinates (i.e. $\hat{\nabla}_{\partial_i}\partial_j=0$) of a given point, by Lemma \ref{lem2-1} and \eqref{equ:Q_0}, direct computation yields
\begin{align}\label{equ:pf-lemma3.3-3}
	G(\nu_F)(\nu_F,E^F_{n+1})_{,k}&=G(\nu_F)(\partial_k\nu_F,E^F_{n+1})+Q(\nu_F)(\partial_k\nu_F,\nu_F,E^F_{n+1})\nonumber
	\\
	&=\hat{h}_k^sG(\nu_F)(X_s,E^F_{n+1}),
\end{align}
and
\begin{align}
	&G(\nu_F)(\nu_F,E^F_{n+1})_{,kl}
	\nonumber
	\\
	=&\hat{h}_{k,l}^sG(\nu_F)(X_s,E^F_{n+1})
	+\hat{h}_k^sG(\nu_F)(\partial_l\partial_sX,E^F_{n+1})
	+\hat{h}_k^sQ(\nu_F)(\partial_l\nu_F,X_s,E^F_{n+1})
	\nonumber
	\\
	=&\hat{h}_{ks,l}G(\nu_F)(X^s,E^F_{n+1})
	+
	\hat{h}_k^sG(\nu_F)\(-\hat{h}_{ls}\nu_F+A_{slp}X^p,E^F_{n+1}\)
	+
	\hat{h}_k^s\hat{h}_l^qQ(\nu_F)(X_q,X_s,E^F_{n+1})
	\nonumber
	\\
	=&\left(
	\hat{h}_{kl,s}
	+\hat{h}_l^pA_{psk}
	-\hat{h}_s^pA_{plk}
	\right)G(\nu_F)(X^s,E^F_{n+1})
	-
	\hat{h}_k^s\hat{h}_{ls}G(\nu_F)(\nu_F,E^F_{n+1})
	\nonumber
	\\
	&+
	\hat{h}_k^sA_{slp}G(\nu_F)(X^p,E^F_{n+1})
	+
	\hat{h}_k^s\hat{h}_l^qQ_{qsp}G(\nu_F)(X^p,E^F_{n+1})
	\nonumber
	\\
	=
	&\hat{h}_{kl,s}G(\nu_F)(X^s,E^F_{n+1})	
	-
	\hat{h}_k^s\hat{h}_{ls}G(\nu_F)(\nu_F,E^F_{n+1})
	\nonumber
	\\
	&+\left(
	\hat{h}_l^pA_{psk}
	-\hat{h}_s^pA_{plk}
	+\hat{h}_k^pA_{pls}
	+\hat{h}_k^p\hat{h}_l^qQ_{qps}
	\right)G(\nu_F)(X^s,E^F_{n+1})
    .\label{equ:pf-lemma3.3-1}
\end{align}
Due to \eqref{AA}, we check that
\begin{align} \label{equ:pf-lemma3.3-2}
	A_{psl}+A_{pls}+\hat{h}_p^qQ_{qls}=0.
\end{align}
By combination of \eqref{equ:pf-lemma3.3-3}$\thicksim$\eqref{equ:pf-lemma3.3-2}, and the symmetry of $\hat{g}^{lk}$ and $ \hat{h}^{pl}$, we obtain
 \begin{align}
 	&\hat{\Delta}G(\nu_F)(\nu_F,E^F_{n+1})=\hat{g}^{kl}G(\nu_F)(\nu_F,E^F_{n+1})_{,kl}\nonumber
 	\\
 	=
 	&G(\nu_F)(\hat{\nabla}H_F,E^F_{n+1})	
 	-
 	|\hat{h}|^2_{\hat{g}}G(\nu_F)(\nu_F,E^F_{n+1})
 	-\hat{g}^{kl}\hat{h}_s^pA_{plk}G(\nu_F)(X^s,E^F_{n+1})
 	\nonumber
 	\\
 	&+\left(
 	\hat{h}^{pl}A_{psl}
 	+\hat{h}^{pl}A_{pls}
 	+\hat{h}^{pl}\hat{h}_l^qQ_{qps}
 	\right)G(\nu_F)(X^s,E^F_{n+1})
 	\nonumber
 	\\
 	=
 	&G(\nu_F)(\hat{\nabla}H_F,E^F_{n+1})	
 	-
 	|\hat{h}|^2_{\hat{g}}G(\nu_F)(\nu_F,E^F_{n+1})
 	-\hat{g}^{kl}A_{plk}\hat{\nabla}^pG(\nu_F)(\nu_F,E^F_{n+1}).
 	\label{equ:pf-lemma3.3-4}
 \end{align}
By \eqref{equ:pfL3.2-5}, \eqref{equ:pfL3.2-1}, and \eqref{equ:pf-lemma3.3-3}, we have
\begin{align}
	&G(\nu_F)(\hat{\nabla}f,E^F_{n+1})
		\nonumber
		\\
	=&n\omega_0\hat{h}_k^lG(\nu_F)(\partial_lX,E^F_{n+1})G(\nu_F)(X^k,E^F_{n+1})
	-G(\nu_F)(\hat{\nabla} H_F,E^F_{n+1})
	\hat{ u}\nonumber\\
	&
	-H_F\hat{h}_i^pG(\nu_F)(X,X_p)G(\nu_F)(X_i,E^F_{n+1})
	\nonumber
	\\
	=&n\omega_0G(\nu_F)\(\hat{\nabla} (G(\nu_F)(\nu_F,E^F_{n+1})),E^F_{n+1}\)-G(\nu_F)(\hat{\nabla} H_F,E^F_{n+1})
	\hat{ u}
	\nonumber
	\\
	&-H_FG(\nu_F)\(\hat{\nabla} (G(\nu_F)(\nu_F,E^F_{n+1})),X\). \label{equ:pf-lemma3.3-5}
\end{align}
Putting \eqref{equ:pf-lemma3.3-3}, \eqref{equ:pf-lemma3.3-4}, \eqref{equ:pf-lemma3.3-5} into \eqref{equ:ptG(v,E)} yields
\begin{align*}
	&\partial_tG(\nu_F)(\nu_F,E^F_{n+1})
	\\
	=&-n\omega_0G(\nu_F)\(\hat{\nabla} (G(\nu_F)(\nu_F,E^F_{n+1})),E^F_{n+1}\)+G(\nu_F)(\hat{\nabla} H_F,E^F_{n+1})
	\hat{ u}
	\\
	&+H_FG(\nu_F)\(\hat{\nabla} (G(\nu_F)(\nu_F,E^F_{n+1})),X\)
	+\hat{h}^{kj}G(\nu_F)(T,\partial_kX)G(\nu_F)(E_{n+1}^F,\partial_jX)
	\\
	=&-n\omega_0G(\nu_F)\(\hat{\nabla} (G(\nu_F)(\nu_F,E^F_{n+1})),E^F_{n+1}\)
	+H_FG(\nu_F)\(\hat{\nabla} (G(\nu_F)(\nu_F,E^F_{n+1})),X\)
		\\
	&+
	\hat{ u}	
	\left(
	\hat{\Delta}G(\nu_F)(\nu_F,E^F_{n+1})
	+\hat{g}^{kl}A_{plk}\hat{\nabla}^pG(\nu_F)(\nu_F,E^F_{n+1})
	+|\hat{h}|^2_{\hat{g}}G(\nu_F)(\nu_F,E^F_{n+1})
	\right)
	\\
	&+G(\nu_F)\(\hat{\nabla} (G(\nu_F)(\nu_F,E^F_{n+1})),T\).
\end{align*}
Thus we conclude that
\begin{align}
	\label{equ:L<v,E>}
	\mathcal{L}\(G(\nu_F)(\nu_F,E^F_{n+1})\)=\hat{ u}|\hat{h}|^2_{\hat{g}}G(\nu_F)(\nu_F,E^F_{n+1}).
\end{align}
Form	\eqref{equ:Lu^} and \eqref{equ:L<v,E>}, we obtain
\begin{align*}
	\mathcal{L}\bar{ u}=&\mathcal{L}\frac{\hat{ u}}{1+\omega_0 G(\nu_F)(\nu_F(X),E_{n+1}^F)}
	\\
	=&\frac{\mathcal{L}\hat{ u}}{1+\omega_0 G(\nu_F)(\nu_F(X),E_{n+1}^F)}
	-\frac{\omega_0\hat{ u}\mathcal{L}\( G(\nu_F)(\nu_F(X),E_{n+1}^F)\)}{\(1+\omega_0 G(\nu_F)(\nu_F(X),E_{n+1}^F)\)^2}
	\\
	&-\hat{ u}\left(
	\frac{2\hat{ u}\hat{\nabla}|1+\omega_0 G(\nu_F)(\nu_F(X),E_{n+1}^F)|^2_{\hat{g}}}{\(1+\omega_0 G(\nu_F)(\nu_F(X),E_{n+1}^F)\)^3}
	\right.
	\\
	&\left.
	-
	\frac{2G(\nu_F) \left(\hat{\nabla}\hat{ u},\hat{\nabla}\left(1+\omega_0 G(\nu_F)(\nu_F(X),E_{n+1}^F)\right)\right)}{\(1+\omega_0 G(\nu_F)(\nu_F(X),E_{n+1}^F)\)^2}
	\right)
	\\
	=&n-2\bar{ u}H_F+\bar{ u}^2
	|\hat{h}|_{\hat{g}}^2
	+
	2\bar{ u}G(\nu_F)\(\hat{\nabla}\bar{ u},\hat{\nabla}\left(1+\omega_0G(\nu_F)(\nu_F,E^F_{n+1})\right)\).
\end{align*}
\end{proof}

\begin{lemma}
	\label{lemma-DU-=0}
	For  $X\in \partial\Sigma_t$, we take $\mu_F\in T_X\Sigma$, which satisfies $G(\nu_F)(\xi,\mu_F)=0$ for any $\xi\in T_X(\partial\Sigma)$. Then
		along flow \eqref{equ:flow},
		\begin{align*}
			\hat{\nabla}_{\mu_F}\bar{ u}=0, \quad \text{on}\  \partial\Sigma_t.
		\end{align*}
\end{lemma}
\begin{proof}
	Without loss of generality, we assume $G(\nu_F)(\mu_F,\mu_F)=1$.
	We take the orthogonal frame $\{e_i:=\partial_iX\}_{i=1}^n$ on $(\Sigma,\hat{g})$, which satisfies $e_n=\mu_F$ on the boundary of $\Sigma$. Then $e_{\alpha}\in T(\partial\Sigma)\subset\mathbb{R}^n$, $\alpha=1,\cdots,n-1$, which means on $ \partial\Sigma_t$, we have
	\begin{align}\label{equ:<e_a,E>=0}
		\<e_{\alpha},E_{n+1}\>=0, \quad \alpha=1,\cdots,n-1.
		\end{align}
		And since $T(\partial\Sigma)\subset T\Sigma\cap  \mathbb{R}^n$, $T\Sigma\neq 
		 \mathbb{R}^n$, we have $T(\partial\Sigma)= T\Sigma\cap  \mathbb{R}^n$, then
		\begin{align}\label{equ:<e_n,E>=/0}
			\<e_{n},E_{n+1}\>\neq 0.
		\end{align}
Now by $\<\nu_F,E_{n+1}\>=-\omega_0$ on $ \partial\Sigma_t$, we have
\begin{align*}
	0=\<\partial_{\alpha}\nu_F,E_{n+1}\>\overset{\eqref{Weingarten}}{=}\sum_{i=1}^n\<\hat{h}_{\alpha}^ie_i,E_{n+1}\>\overset{\eqref{equ:<e_a,E>=0}}{=}\hat{h}_{\alpha}^n\<e_n,E_{n+1}\>,
\end{align*}
which implies $\hat{h}_{\alpha}^n=0$ by \eqref{equ:<e_n,E>=/0}.
Then we can derive that
\begin{align}
	\label{equ:Dnu^}
	\hat{\nabla}_{\mu_F}\hat{ u}=\partial_n\hat{ u}=G(\nu_F)(e_n,\nu_F)+G(\nu_F)(X,\partial_n\nu_F)=\hat{h}_n^nG(\nu_F)(X,e_n),
\end{align}
and
\begin{align}
	\label{equ:Dn<E,v>}
	\hat{\nabla}_{\mu_F}\left(
	1+\omega_0G(\nu_F)(\nu_F,E^F_{n+1})
	\right)
	=
	\omega_0G(\nu_F)(E^F_{n+1},\partial_n\nu_F)=\omega_0\hat{h}_n^nG(\nu_F)(E^F_{n+1},e_n).
\end{align}
By combination of \eqref{equ:Dnu^} and \eqref{equ:Dn<E,v>}, we obtain
\begin{align*}
	\hat{\nabla}_{\mu_F}\bar{ u}=&\frac{\partial_n\hat{ u}\(1+\omega_0 G(\nu_F)(\nu_F,E_{n+1}^F)\)
	-\hat{ u}\hat{\nabla}_n\(1+\omega_0 G(\nu_F)(\nu_F,E_{n+1}^F)\)}
	{\(1+\omega_0 G(\nu_F)(\nu_F,E_{n+1}^F)\)^2}
	\\
	=&\frac{\hat{h}_n^nG(\nu_F)(X,e_n)\(1+\omega_0 G(\nu_F)(\nu_F,E_{n+1}^F)\)
		-\omega_0\hat{h}_n^nG(\nu_F)(E^F_{n+1},e_n)\hat{ u}}
	{\(1+\omega_0 G(\nu_F)(\nu_F,E_{n+1}^F)\)^2}
	\\
	=& \frac{\hat{h}_n^nG(\nu_F)\(\left(1+\omega_0 G(\nu_F)(\nu_F,E_{n+1}^F)\right)X-\omega_0\hat{ u}E_{n+1}^F,e_n\)}
	{\(1+\omega_0 G(\nu_F)(\nu_F,E_{n+1}^F)\)^2}.
\end{align*}
Let $Y=\(1+\omega_0 G(\nu_F)(\nu_F,E_{n+1}^F)\)X-\omega_0\hat{ u}E_{n+1}^F$, we have $\<Y,\nu\>=\<X,\nu\>$ since \eqref{equ:G(vF,Y)=<v,Y>/F}. Let  $\widehat{Y^{\top}}=Y-G(\nu_F)(Y,\nu_F)\nu_F\in T\Sigma$, we know $G(\nu_F)(Y,e_n)=G(\nu_F)(\widehat{Y^{\top}},e_n)$, since $G(\nu_F)(\nu_F,e_n)=0$. Then we just need to prove the claim that $\widehat{Y^{\top}}\in T(\partial\Sigma)$.

Denote by $\mu$ the unit outward co-normal of $\partial \Sigma\subset (\Sigma,g)$ with respect to the isotropic metric $g$ induced  from Euclidean space $(\mathbb{R}^{n+1},\<\cdot,\cdot\>)$, it follows from
the proof of Theorem 1.3 in \cite{Jia-Wang-Xia-Zhang2023} that on $\partial\Sigma$:
\begin{align*}
	-\langle X, \mu\rangle\left\langle \nu, E_{n+1}^F\right\rangle \omega_0+\langle X, \nu\rangle\left\langle\mu, E_{n+1}^F\right\rangle \omega_0
	=
	F(\nu)\langle X, \mu\rangle-\langle X, \nu\rangle\left\langle \nu_F, \mu\right\rangle,
\end{align*}
which implies
\begin{align*}
F(\nu)\<	\widehat{Y^{\top}},\mu\>=\left\langle \(F(\nu)+\omega_0 \<\nu,E_{n+1}^F\>\)X-\omega_0\<X,\nu\>E_{n+1}^F,\mu\right\rangle -\<Y,\nu\>\<\nu_F,\mu\>=0.
\end{align*}
That means $\widehat{Y^{\top}}\in T(\partial\Sigma)$, thus $\hat{\nabla}_{\mu_F}\bar{ u}=0$.
\end{proof}

\begin{remark}\label{rk3.5}
	(1) $\mu_F$ defined by Lemma \ref{lemma-DU-=0} can be  specifically represented as ${A_F(\nu)\mu }$ on the  boundary of $\Sigma$. That means $\mu_F=A_F(\nu)\mu\in T_X\Sigma$ and satisfies  $G(\nu_F)(A_F(\nu)\mu,Y)=0$ for any $Y\in T(\partial\Sigma)$.
	
	(2) For anisotropic $\omega_0$-capillary hypersurface $\Sigma$, we take the standard orthogonal frame $\{e_i\}_{i=1}^n$ on $(\Sigma,\hat{g})$, which satisfies $e_n=\frac{\mu_F}{|\mu_{{F}}|_{\hat{g}}}=\frac{F(\nu)^{1/2}\cdot A_F(\nu)\mu}{\<A_F(\nu)\mu,\mu\>^{1/2}}$ on the boundary of $\Sigma$. Then on $\partial\Sigma$, we have
	\begin{align}
		\label{equ:han=0}
		\hat{h}_{\alpha n}=0.
	\end{align}
	By rotation of the local frame, we can assume that matrix $(\hat{h}_{ij})_{n\times n}$ is diagonal of a given point on $\partial\Sigma$,  at this point we have
	\begin{align}
		\label{equ:haan}
		\hat{h}_{\alpha\alpha,n}=\left(\frac{\omega_0}{\<E_{n+1},\mu\> F(\nu)|\mu_{{F}}|_{\hat{g}}}
		-
		Q(\nu_{{F}})(e_{\alpha},e_{\alpha},e_n)
		\right)
		\hat{h}_{\alpha\alpha}(\hat{h}_{nn}-\hat{h}_{\alpha\alpha}),
	\end{align}
	for $\alpha\in\{1,\cdots,n-1\}.$
\end{remark}
\begin{proof}
	(1) We just need to check that $G(\nu_F)(A_F(\nu)\mu,Y)=0$ for $Y\in T(\partial\Sigma)$.
	
	Take the orthogonal frame $\{\bar{e}_{\alpha}\}_{\alpha=1}^n$ on $(\Sigma,g)$, which satisfies $\bar{e}_n=\mu$ on the boundary of $\Sigma$,
	then by the proof of \cite{Xia-phd}*{Proposition 1.4} we check that
	\begin{align}
		D^2F({\nu})\mu=&\sum_{{\alpha}=1}^{n}\<D^2F({\nu})\mu,\bar{e}_{\alpha}\>\bar{e}_{\alpha}+\<D^2F({\nu})\mu,\nu\>\nu\nonumber
		\\
		=&\sum_{{\alpha}=1}^{n}\<A_F({\nu})\mu,\bar{e}_{\alpha}\>\bar{e}_{\alpha}\nonumber
		\\
		=&A_F({\nu})\mu\label{equ:remark0}, \quad \text{on the  boundary of $\Sigma$},
	\end{align}
	where we use the fact that $\<D^2F(\nu)\nu,\cdot\>\equiv0$, since the $1$-homogeneity of $F$.
	
	Besides, denote by $\{\epsilon_i\}_{i=1}^{n+1}$ the orthogonal frame on $\mathbb{R}^{n+1}$, and for any $\xi=\xi_i\epsilon_i\in\mathbb{R}^{n+1}$, we have $F(\xi)D_iF^0(DF(\xi))=\xi_i$ (see \cite{Xia-phd}*{Proposition 1.3(ii)}). Taking the partial derivative with respect to $\xi_j$ on both sides of the above equation yields
	\begin{align*}
		D_jF|_{\xi}\cdot D_iF^0|_{DF(\xi)}+F(\xi)D^2_{ik}F^0|_{DF(\xi)}\cdot D^2_{kj}F|_{\xi}=\delta_{ij}.
	\end{align*}
	Take $\xi=\nu$, since $\nu_F=DF(\nu)$, we obtain
	\begin{align}\label{equ:remark1}
		D^2_{ik}F^0|_{\nu_F}\cdot D^2_{kj}F|_{\nu}=-\frac{D_jF|_{\nu}\cdot D_iF^0|_{\nu_F}}{F(\nu)}+\frac{\delta_{ij}}{F(\nu)}.
	\end{align}
	Then for $\mu=\mu_i\epsilon_i\in\mathbb{R}^{n+1}$ and for any $Y=Y_i\epsilon_i\in T\Sigma$, we have
	\begin{align}
		G(\nu_F)(A_F({\nu})\mu,Y)=&\left.D^2\left(\frac{1}{2}(F^0)^2\right)\right|_{\nu_F}(A_F({\nu})\mu,Y)\nonumber
		\\
		=&F^0(\nu_F)D^2F^0|_{\nu_F}(A_F({\nu})\mu,Y)+\<A_F({\nu})\mu, DF^0(\nu_F)\>\cdot \< Y,DF^0(\nu_F)\>
	\nonumber	\\
		=&\<D^2F^0|_{\nu_F}A_F({\nu})\mu,Y\>
	\nonumber	\\
		\overset{\eqref{equ:remark0}}{=}&\<D^2F^0|_{\nu_F}D^2F|_{\nu}\mu,Y\>
		=\sum_{i,j,k=1}^n D^2_{ik}F^0|_{\nu_F}D^2_{kj}F|_{\nu}\mu_jY_i\nonumber
		\\
		\overset{\eqref{equ:remark1}}{=}&-\frac{1}{F(\nu)^2}\<\nu,Y\>\<\nu_{{F}},\mu\>+\frac{1}{F(\nu)}\<\mu,Y\>\nonumber
		\\
		=&\frac{1}{F(\nu)}\<\mu,Y\>, \quad\text{on the boundary of $\Sigma$}, \label{equ:G(muF,Y)=<mu,Y>}
	\end{align}
	where we use $F^0(\nu_{{F}})=F^0(DF(\nu))=1$ and $DF^0(\nu_{{F}})=DF^0(DF(\nu))=\frac{\nu}{F(\nu)}$.
	
	Thus, for $Y\in T(\partial\Sigma)$, we have $  \<\mu,Y\>=0$, which implies  $G(\nu_F)(A_F({\nu})\mu,Y)=0$ by \eqref{equ:G(muF,Y)=<mu,Y>}.
	
	(2) Equation \eqref{equ:han=0} has been derived from the proof of Lemma \ref{lemma-DU-=0}.
	
	 We take $\alpha,\beta,\gamma,\eta\in\{1,\cdots,n-1\}.$ Since $\hat{h}_{\alpha n}=0$ on $\partial \Sigma$,   we have
	\begin{align}
		0=&e_{\beta}\left(\hat{h}(e_n,e_{\alpha})\right)=\hat{h}_{ n\alpha,\beta }
		+\hat{h}(\hat{\nabla}_{e_{\beta}}e_n,e_{\alpha})
		+\hat{h}(\hat{\nabla}_{e_{\beta}}e_{\alpha},e_n)\nonumber
		\\
		=&\hat{h}_{\alpha\beta,n}+\hat{h}_{\beta\gamma}A_{\gamma n \alpha}-\hat{h}_{nn}A_{n \beta \alpha}+\hat{h}\left(D_{e_{\beta}}e_{\alpha}+\hat{h}_{\alpha\beta}\nu_{{F}}-\sum_{k=1}^{n}A_{\alpha\beta k}e_k , e_n\right)\nonumber
		\\
		&+\hat{h}\left(D_{e_{\beta}}e_{n}+\hat{h}_{n\beta}\nu_{{F}}-\sum_{k=1}^{n}A_{n\beta k}e_k , e_{\alpha}\right)\nonumber
		\\
		=&\hat{h}_{\alpha\beta , n}+G(\nu_{{F}})(D_{e_{\beta}}e_{\alpha},e_n)\hat{h}_{nn}+G(\nu_{{F}})(D_{e_{\beta}}e_n,e_{\gamma})\hat{h}_{\alpha\gamma}\nonumber
		\\
		&+A_{n\gamma\alpha}\hat{h}_{\beta\gamma}-A_{n\beta\alpha}\hat{h}_{nn}-A_{\alpha\beta n}\hat{h}_{nn}-A_{n\beta\gamma}\hat{h}_{\gamma\alpha}\nonumber
		\\
		=&\hat{h}_{\alpha\beta,n}+G(\nu_{{F}})(D_{e_{\beta}}e_{\gamma},e_n)\left(\hat{h}_{nn}\delta_{\gamma\alpha}-\hat{h}_{\gamma\alpha}\right)-\hat{h}_{\alpha\gamma}\hat{h}_{\beta\eta}Q_{\eta\gamma n}\nonumber
		\\
		&+A_{n\gamma\alpha}\hat{h}_{\beta\gamma}-A_{n\beta\alpha}\hat{h}_{nn}-A_{\alpha\beta n}\hat{h}_{nn}-A_{n\beta\gamma}\hat{h}_{\gamma\alpha},\label{equ:pf-rk-1}
	\end{align}
	where we use Lemma \ref{lem2-1}  and
	\begin{align*}
		0=&G(\nu_F)(D_{e_{\beta}}e_{\alpha},e_n)+G(\nu_F)(D_{e_{\beta}}e_n,e_{\alpha})+Q(\nu_{{F}})(D_{e_{\beta}}\nu_{{F}},e_{\alpha},e_n)
		\\
		=&G(\nu_F)(D_{e_{\beta}}e_{\alpha},e_n)+G(\nu_F)(D_{e_{\beta}}e_n,e_{\alpha})+\hat{h}_{\eta\beta}Q(\nu_{{F}})(e_{\eta},e_{\alpha},e_n),
	\end{align*}
	since $G(\nu_{{F}})(e_{\alpha},e_n)=0.$
	
	By rotation of local frame, we assume that $\hat{h}_{\alpha\gamma}=\hat{h}_{\alpha\alpha}\delta_{\alpha\gamma}$.
	By  \eqref{equ:pf-lemma3.3-2},  for  $\alpha=\beta$, we can simplify  \eqref{equ:pf-rk-1} as follows,
	\begin{align}\label{equ:pf-rk-2}
		\hat{h}_{\alpha\alpha,n}=-\left(G(\nu_{{F}})(D_{e_{\alpha}}e_{\alpha},e_n)+Q_{\alpha\alpha n}\hat{h}_{\alpha\alpha}\right)\left(\hat{h}_{nn}-\hat{h}_{\alpha\alpha}\right).
	\end{align}
By $e_n=\frac{\mu_{{F}}}{|\mu_{{F}}|_{\hat{g}}}$, 	
$\mu=\frac{E_{n+1}-\<E_{n+1},\nu\>\nu}{\<E_{n+1},\mu\>},
	-\omega_0=\<\nu_{{F}},E_{n+1}\>=F(\nu)\<\nu,E_{n+1}\>+\<\mu,E_{n+1}\>\<\nu_{{F}},\mu\>,$ \eqref{equ:G(vF,Y)=<v,Y>/F}, \eqref{equ:G(muF,Y)=<mu,Y>}, and \eqref{equ:Gauss-formula}, we have
\begin{align}
	&G(\nu_{{F}})(D_{e_{\alpha}}e_{\alpha},e_n)=|\mu_{{F}}|_{\hat{g}}^{-1}G(\nu_{{F}})(D_{e_{\alpha}}e_{\alpha}-G(\nu_{{F}})(D_{e_{\alpha}}e_{\alpha},\nu_{{F}})\nu_{{F}},\mu_{{F}})\nonumber
	\\
	=&\frac{\left\langle D_{e_{\alpha}}e_{\alpha}-G(\nu_{{F}})(D_{e_{\alpha}}e_{\alpha},\nu_{{F}})\nu_{{F}},\mu\right\rangle }{F(\nu)|\mu_{{F}}|_{\hat{g}}}\nonumber
	\\
	=&\left\langle \frac{D_{e_{\alpha}}e_{\alpha}}{F(\nu)|\mu_{{F}}|_{\hat{g}}},\frac{E_{n+1}-\<E_{n+1},\nu\>\nu}{\<E_{n+1},\mu\>}\right\rangle -\frac{\<D_{e_{\alpha}}e_{\alpha},\nu\>\<\nu_{{F}},\mu\>}{F(\nu)^2|\mu_{{F}}|_{\hat{g}}}\nonumber
	\\
	=&-\frac{\omega_0\hat{h}_{\alpha\alpha}}{\<\mu,E_{n+1}\>F(\nu)|\mu_{{F}}|_{\hat{g}}}.\label{equ:pf-rk-3}
\end{align}
	From \eqref{equ:pf-rk-2} and  \eqref{equ:pf-rk-3}, we obtain \eqref{equ:haan}.
\end{proof}

\subsection{Radial graph parametrisation and the short time existence}

We assume that a capillary hypersurface $\Sigma$ is strictly star-shaped with respect to the  origin. Then it can be written as a graph over $\overline{\mathbb{S}_{+}^n}$:
 $$\Sigma=\{ \rho(x)x|x\in \overline{\mathbb{S}_{+}^n} \},$$
where $x:=\left(x_1, \cdots, x_n\right)$ is a local coordinate of $\overline{\mathbb{S}_{+}^n}$.
We denote $\nabla^0$ as the Levi-Civita connection on $\mathbb{S}_{+}^n$ with respect to the standard round metric $\sigma:=g^{\mathbb{S}_{+}^n}, \partial_i:=\partial_{x_i}, \sigma_{i j}:=\sigma\left(\partial_i, \partial_j\right), \rho_i:=\nabla_i^0 \rho$, and $\rho_{i j}:=\nabla_i^0 \nabla_j^0 \rho$,  $\varphi(x):=\log \rho(x)$,  $\rho^i:=\sigma^{i j} \rho_j, \varphi^i:=\sigma^{i j} \varphi_j$ and $v:=\sqrt{1+\left|\nabla^0 \varphi\right|^2}$.
Then  the unit outward normal vector field and support function on $\Sigma$ are respectively given by
\begin{align}\label{equ:v-graph}
	\nu=\frac{1}{v}\left(\partial_\rho-\rho^{-2} \nabla^0 \rho\right)=\frac{1}{v}\left(\partial_\rho-e^{-\varphi} \nabla^0 \varphi\right),
\end{align}
\begin{align}\label{equ:u^-grahp}
	\<X,\nu\>=\<\rho\partial_{\rho},\nu\>=\frac{e^{\varphi}}{v}:=\psi_1(x,\varphi,\nabla^0 \varphi)>0.
\end{align}
The induced metric $g$ and its inverse $g^{-1}$ on $\Sigma$ can be written respectively as
$$
g_{i j}=\rho^2 \sigma_{i j}+\rho_i \rho_j=e^{2 \varphi}\left(\sigma_{i j}+\varphi_i \varphi_j\right),
$$
$$
g^{i j}=\frac{1}{\rho^2}\left(\sigma^{i j}-\frac{\rho^i \rho^j}{\rho^2+\left|\nabla^0 \rho\right|^2}\right)=e^{-2 \varphi}\left(\sigma^{i j}-\frac{\varphi^i \varphi^j}{v^2}\right).
$$
The second fundamental form $h$ on $\Sigma$ is
$$
h_{i j}=\frac{e^{\varphi}}{v}\left(\sigma_{i j}+\varphi_i \varphi_j-\varphi_{i j}\right),
$$
and its Weingarten matrix $\left(h_j^i\right)$ is
$$
h_j^i=g^{i k} h_{k j}=\frac{1}{e^{\varphi} v}\left(\delta_j^i-\left(\sigma^{i k}-\frac{\varphi^i \varphi^k}{v^2}\right) \varphi_{k j}\right) .
$$
We also have  the anisotropic  mean curvature
\begin{align}\label{eqref:HF-graph}
	H_F=&\operatorname{tr}(A_F\mathrm{d}\nu)=\frac{1}{e^{\varphi} v}\left(A_F(\nu)\right)_i^j\left(\delta_j^i-\left(\sigma^{i k}-\frac{\varphi^i \varphi^k}{v^2}\right) \varphi_{k j}\right)\nonumber
	\\
	:=&-a^{ij}(x,\varphi,\nabla^0 \varphi) \varphi_{ij}+\psi_2(x,\varphi,\nabla^0 \varphi),
\end{align}
where $a^{ij}=v^{-1}{A_F(\nu)}_k^je^{\varphi}g^{ik}$ is a positive define matrix.

Moreover, we use the polar coordinate in the half-space.  For $z:=\left(z^{\prime}, z_{n+1}\right) \in$ $\mathbb{R}^n \times[0,+\infty)$ and $x:=(\beta, \xi) \in\left[0, \frac{\pi}{2}\right] \times \mathbb{S}^{n-1}$, we have that
$$
z_{n+1}=\rho \cos \beta, \quad\left|z^{\prime}\right|=\rho \sin \beta .
$$
Then
$$
E_{n+1}=\partial_{z_{n+1}}=\cos \beta \partial_\rho-\frac{\sin \beta}{\rho} \partial_\beta .
$$
In these coordinates the standard Euclidean metric is given by
$$
|d z|^2=d \rho ^2+\rho ^2\left(d \beta^2+\sin ^2 \beta g^{\mathbb{S}^{n-1}}\right) .
$$
It follows that $F(\nu)$ and $\<\nu, E^F_{n+1}\>$ are independent of the second-order derivative of $\varphi$, we can denote that
\begin{align}\label{equ:<v,E>-graph}
nF(\nu)+n\omega_0\left\langle\nu, E^F_{n+1}\right\rangle:=\psi_3(x,\varphi,\nabla^0 \varphi,\omega_0).
\end{align}
Along $\partial \mathbb{S}_{+}^n$, $\partial_\beta$ denotes the unit outward co-normal of $\partial \mathbb{S}_{+}^n$ in $\overline{\mathbb{S}_{+}^n}$, it holds
$$
E_{n+1}=-\frac{1}{\rho} \partial_\beta.
$$
Hence the anisotropic capillary boundary condition can be expressed by
\begin{align}
	\label{equ:w0=<>-graph}
	\omega_0=\langle\nu_F, -E_{n+1}\rangle
	=
\<DF(\nu),	\frac{\partial_{\beta}}{\rho}\>,
\end{align}
where $\nu$ is expressed by \eqref{equ:v-graph}.

Now under the assumption of star-shapedness, by combination of \eqref{equ:u^-grahp}$\thicksim$\eqref{equ:w0=<>-graph} the flow \eqref{equ:flow} can be reduced to a quasilinear parabolic equation with oblique boundary condition for $\varphi(x, t)$,
\begin{equation}
	\left\{
	\begin{aligned}\label{equ:flow-graph}
	\partial_t \varphi  =\frac{vF(\nu)}{e^{\varphi}}f&=a^{ij}(t,x,\varphi,\nabla^0 \varphi
	)\varphi_{i j}+b(\omega_0,t,x,\varphi,\nabla^0 \varphi
	), & & \text { in } \mathbb{S}_{+}^n \times\left[0, T\right),
	\\
	\psi(t,x,\varphi,\nabla^0\varphi) & =
	0, & & \text { on } \partial \mathbb{S}_{+}^n \times\left[0, T\right),
	\\
	\varphi(\cdot, 0) & =\varphi_0(\cdot), & & \text { on } \mathbb{S}_{+}^n,
\end{aligned}\right.
\end{equation}
where 
$ b:=\frac{\psi_3}{\psi_1}-\psi_2,$ $\psi_1, \psi_2,\psi_3 $ are defined by \eqref{equ:u^-grahp}$\thicksim$\eqref{equ:<v,E>-graph} respectively,  $\psi(t,x,\varphi,\nabla^0\varphi):=\<DF(\nu),\frac{\partial_{\beta}}{e^{\varphi}}\>-\omega_0 $,  and $\varphi_0$ is the parameterization radial function of $\Sigma_0$ over $\overline{\mathbb{S}_{+}^n}$.

Next, we check that the oblique boundary condition \eqref{equ:w0=<>-graph} satisfies the non-degeneracy condition (ref. \cite{Nina1992}) $\<\frac{\partial\psi(t,x,\varphi,p)}{\partial p},\partial_{\beta}\> >0 \ (\text{or }<0)$ on $\partial \mathbb{S}_{+}^n \times\left[0, T\right)$, where we denote $p=\nabla^0\varphi$.

We take orthogonal frame $\{e_{i}\}_{i=0}^{n}$ on $\mathbb{R}^{n+1}$ such that $e_0=\partial_{\rho},e_{n}=-\frac{\partial_{\beta}}{\rho}$.
 Since on $\partial\mathbb{S}_+^n$, $e_{n}=E_{n+1}$, denote by $p=\sum_{i=1}^np_ie_i$, we have by \eqref{equ:v-graph}, $\nu=\frac{e_0-\sum_{i=1}^n{\rho^{-1}p_ie_i}}{\left(1+\sum_{i=1}^{n}p_i^2\right)^{\frac{1}{2}}}:=\nu_0e_0+\cdots+\nu_ne_n$. Then
\begin{align*}
	\frac{\partial \nu_0}{\partial p_n}&=-v^{-3}p_n,
	\\
 \frac{\partial \nu_i}{\partial p_n}&=v^{-3}\rho^{-1}p_ip_n, \quad i=1,\cdots,n-1,
 \\
 \frac{\partial \nu_n}{\partial p_n}&=v^{-3}\rho^{-1}p_n^2-v^{-1}\rho^{-1},
\end{align*}
which implies
\begin{align*}
	\<\frac{\partial\psi(t,x,\varphi,p)}{\partial p},e_n\> =&\frac{\partial\psi(t,x,\varphi,p)}{\partial p_n}
	=-\frac{\partial F_n(\nu)}{\partial p_n}
	=-\sum_{i=0}^{n} F_{ni}\frac{\partial \nu_i}{\partial p_n}	
	\\
	=&v^{-2}p_n\sum_{i=0}^n F_{ni}(\nu)\nu_i+v^{-1}\rho^{-1}F_{nn}(\nu)
	\\
	=&v^{-1}\rho^{-1}F_{nn}(\nu),
\end{align*}
where we denote $F_i=\<DF(\nu),e_i\>, F_{ij}=\<D^2F(\nu)e_i,e_j\>$
and in the last equality we use the $1$-homogeneity of $F$.
$\omega_0 \in\left(-F\left(E_{n+1}\right), F\left(-E_{n+1}\right)\right)$ implies that $\nu \neq E_{n+1}$ and $\nu \neq -E_{n+1}$ for $x\in \partial\mathbb{S}_+^n$.
We can decompose vector $E_{n+1}$ into $E_{n+1}=E'+\<E_{n+1},\nu\>\nu$ with $E'\in T_{\nu}\mathbb{S}^n$, then
\begin{align*}
	F_{nn}(\nu)=&\<D^2F(\nu)E_{n+1},E_{n+1}\>=\left\langle D^2F(\nu)(E'+\<E_{n+1},\nu\>\nu),(E'+\<E_{n+1},\nu\>\nu)\right\rangle
	\\
	=&\<D^2F(\nu)E',E'\>+2\<E_{n+1},\nu\>\<D^2F(\nu)\nu,E'\>+\<E_{n+1},\nu\>^2\<D^2F(\nu)\nu,\nu\>
	\\
	=&\<D^2F(\nu)E',E'\>,
\end{align*}
where we used the $1$-homogeneity of $F$ again. By \cite{Xia-phd}*{Proposition 1.4(ii)}, we have $\<D^2F(\nu)E',E'\>>0$, thus $F_{nn}>0$. Then
\begin{align*}
	\left\langle \frac{\partial\psi(t,x,\varphi,p)}{\partial p},-\partial_{\beta}\right\rangle  >0\quad\text{  on }\partial \mathbb{S}_{+}^n \times\left[0, T\right).
\end{align*}
Hence, the short time existence follows.


\subsection{$C^0$ estimate}
We consider a general  parabolic equation with a strictly oblique boundary condition:
\begin{equation}
	\left\{
	\begin{aligned}\label{equ:general-pde}
		\partial_t \varphi  & =\mathcal{F}(t,x,\varphi,\nabla^0 \varphi,
		\nabla^0 \nabla^0 \varphi)
		, & & \text { in } \mathbb{S}_{+}^n \times\left[0, T\right),
		\\
		\psi(t,x,\varphi,\nabla^0\varphi) & =
		0, & & \text { on } \partial \mathbb{S}_{+}^n \times\left[0, T\right),
		\\
		\varphi(\cdot, 0) & =\varphi_0(\cdot), & & \text { on } \mathbb{S}_{+}^n,
	\end{aligned}\right.
\end{equation}
with	$\<\frac{\partial\psi(t,x,\varphi,p)}{\partial p},e_n\> >0$  on $\partial \mathbb{S}_{+}^n \times\left[0, T\right)$ and $\mathcal{F}^{ij}=\frac{\partial^2\mathcal{F}}{\partial \varphi_{i j}}>0$, where we denote $p=\nabla^0 \varphi$ and $e_n$ is the unit outward co-normal of $\partial \mathbb{S}_+^n \subset (\overline{\mathbb{S}_+^n},\sigma,\nabla^0)$.

Similar to \cite{Wang-Weng-2020}*{Proposition 4.2}, we  show the following   avoidable principle:
 \begin{lemma}\label{lem:avoidable}
 	Assume $\varphi(x,t)$ is a solution of equation \eqref{equ:general-pde} and $\tilde{\varphi}(x)$ is a static solution of equation \eqref{equ:general-pde}. If $\varphi(x,0)\leq (\text{resp.}\geq) \tilde{\varphi}(x)$, then $\varphi(x,t)\leq(\text{resp.}\geq) \tilde{\varphi}(x)$ in $\mathbb{S}_{+}^n \times\left[0, T\right)$.
 \end{lemma}
 \begin{proof}
 	We just check the case $\varphi_0(x)\leq \tilde{\varphi}(x)$.
 Since both $\varphi(x,t)$ and $\tilde{\varphi}(x)$ are solutions of equation \eqref{equ:general-pde}, we know that
 	$$
 	\begin{aligned}
 		\partial_t(\varphi-\tilde{\varphi}) & =\mathcal{F}\((\nabla^0)^2 \varphi, \nabla^0 \varphi, \varphi, x,t\)-\mathcal{F}\((\nabla^0)^2 \tilde{\varphi}, \nabla^0 \tilde{\varphi}, \tilde{\varphi}, x,t\) \\
 		& =a^{i j} \nabla^0_{i j}(\varphi-\tilde{\varphi})+b^j \cdot(\varphi-\tilde{\varphi})_j+c \cdot(\varphi-\tilde{\varphi}),
 	\end{aligned}
 	$$
 	where $a^{i j}:=\int_0^1 \mathcal{F}^{i j}\((\nabla^0)^2(s \varphi+(1-s) \tilde{\varphi}), \nabla^0(s \varphi+(1-s) \tilde{\varphi}), s \varphi+(1-s) \tilde{\varphi}, x,t\) d s$,  	$ b^j:=\int_0^1 \mathcal{F}_{p_j}\((\nabla^0)^2(s \varphi+(1-s) \tilde{\varphi}), \nabla^0(s \varphi+(1-s) \tilde{\varphi}), s \varphi+(1-s) \tilde{\varphi}, x,t\) d s$, and
 	$$
 	c:=\int_0^1 \mathcal{F}_{\varphi}\((\nabla^0)^2(s \varphi+(1-s) \tilde{\varphi}), \nabla^0(s \varphi+(1-s) \tilde{\varphi}), s \varphi+(1-s) \tilde{\varphi}, x,t\)  d s.
 	$$
 	Denote $\lambda:=-\sup _{\mathbb{S}_{+}^n \times[0, T]}|c|$. Applying the maximum principle, we know that $e^{\lambda t}(\varphi-\tilde{\varphi})$ attains its nonnegative maximum value at the parabolic boundary, says $\left(x_0, t_0\right)$. That is,
 	$$
 	e^{\lambda t}(\varphi(x, t)-\tilde{\varphi}(x)) \leq \sup _{\partial \mathbb{S}_{+}^n \times[0, T) \cup \mathbb{S}_{+}^n \times\{0\}}\left\{0, e^{\lambda t}(\varphi(x, t)-\tilde{\varphi}(x))\right\} ,
 	$$
 	with either $x_0 \in \partial \mathbb{S}_{+}^n$ or $t_0=0$. If $x_0 \in \partial \mathbb{S}_{+}^n$, from the Hopf lemma, we have
 	$$
 	(\nabla^0)^{\prime}(\varphi-\tilde{\varphi})\left(x_0, t_0\right)=0,
 	\quad
 	 \nabla^0_n \varphi\left(x_0, t_0\right)>\nabla^0_n \tilde{\varphi}\left(x_0, t_0\right),
 	$$
 	that is, $\left|(\nabla^0)^{\prime} \varphi\right|=\left|(\nabla^0)^{\prime} \tilde{\varphi}\right|:=s$ and $\nabla^0_n \varphi>\nabla^0_n \tilde{\varphi}$ at $\left(x_0, t_0\right)$. Here we denote $(\nabla^0)^{\prime}$ and $\nabla^0_n$ as the tangential and normal part of $\nabla^0$ on $\partial \mathbb{S}_{+}^n, e_n$ is the outward co-normal vector field on $\partial \mathbb{S}_{+}^n\subset \overline{\mathbb{S}_n^+}$. From the boundary condition in \eqref{equ:general-pde} we have
 	$$
 	\psi(t,x,\varphi,\nabla^0\varphi)  =
 	0=\psi(t,x,\tilde{\varphi},\nabla^0\tilde{\varphi}) ,
 	$$
 	a contradiction to the fact that 	$\left\langle \frac{\partial\psi(t,x,\varphi,p)}{\partial p},e_n\right\rangle  =\frac{\partial\psi(t,x,\varphi,p)}{\partial p_n}>0$ (i.e.  function $\psi(t,x,\varphi,p)$ is strictly increasing with respect to $p_n \in \mathbb{R}$) and $\nabla^0_n \varphi>\nabla^0_n \tilde{\varphi}$ at $\left(x_0, t_0\right)$. Hence we have $t_0=0$, which implies that
 	$$
 	e^{\lambda t}\(\varphi(x, t)-\tilde{\varphi}(x)\) \leq \varphi_0\left(x_0\right)-\tilde{\varphi}\left(x_0\right) \leq 0, \quad \text { in } \quad(x, t) \in \mathbb{S}_{+}^n \times[0, T],
 	$$
 	that is
 	$$
 	\varphi(x, t) \leq \tilde{\varphi}(x), \quad \text { in } \quad(x, t) \in \mathbb{S}_{+}^n \times[0, T] .
 	$$
 	The case $\varphi_0(x)\geq \tilde{\varphi}(x)$ is similar.
 	Hence, we  finish the proof of the lemma.
 \end{proof}

Now we assume $T^*>0$ is the maximal time of existence of a solution to \eqref{equ:flow}, in the class of star-shaped hypersurfaces.
We want to get the $C^0$ estimate of $\varphi$ in $\mathbb{S}_{+}^n \times[0, T^*)$.

\begin{proposition}\label{prop:C0}
	For any $t\in [0,T^*)$, along flow \eqref{equ:flow}, it holds
	$$
	\Sigma_t \subset \widehat{\W_{ r_2,\omega_0}} \backslash \widehat{\W_{r_1,\omega_0}} .
	$$
	where $\W_{ r,\omega_0}$ is defined by  \eqref{equ:crwo} and $r_1,r_2$ only depend on $\Sigma_0$ and Wulff shape $\W$.
\end{proposition}
\begin{proof}
	The star-shapedness and compactness of $\Sigma_0$ imply that  there exist some constants  $0<C_1<C_2<\infty$, such that
	$C_1\leq \<X,\nu(X)\>\leq |X|\cdot |\nu(X)|=|X|\leq C_2$. Since $F^0(\cdot)$ is a norm on $\mathbb{R}^{n+1}$,  there exist some constants $0<C_3<C_4<\infty$, such that $C_3\leq F^0(X)\leq C_4$, for $X\in\Sigma_0$.
	
	Let $r_1=\frac{C_3}{2\(1+F^0(\omega_0E_{n+1}^F)\)}$ and $r_2=\frac{2C_4}{1-F^0(-\omega_0E_{n+1}^F)}>0$. For  $X\in\Sigma_0$, we can check that
	$$
	F^0\(X-r_1\omega_0E_{n+1}^F\)
	\geq
	F^0(X)  -r_1F^0(\omega_0E_{n+1}^F)
	\geq
	C_3-r_1F^0(\omega_0E_{n+1}^F)\geq r_1,
	$$
	and
	$$
	F^0\(X-r_2\omega_0E_{n+1}^F\)
	\leq
	F^0(X)  +r_2F^0(-\omega_0E_{n+1}^F)
	\leq
	C_4+r_2F^0(-\omega_0E_{n+1}^F)\leq r_2,
	$$
	which implies
	$$
	\Sigma_0 \subset \widehat{\W_{ r_2,\omega_0}} \backslash \widehat{\W_{r_1,\omega_0}} .
	$$
	On the other hand, by $\nu_F(X)=\frac{X-r\omega_0E_{n+1}^F}{r}$ and  $1=G(\nu_F)(\nu_F,\frac{X}{r}-\omega_0E_{n+1}^F)$ on $ \W_{r,\omega_0} $, one can easily check that
	\begin{align}\label{equ:Wulff-1+w<v,e>-}
		1+\omega_0G(\nu_F)(\nu_F,E^F_{n+1})-\frac{1}{r}G(\nu_F)(X,\nu_F)=0,
		\quad\text{for}\ X\in \W_{r,\omega_0}.
	\end{align}
	Then $ \W_{r_1,\omega_0}$ and $ \W_{r_2,\omega_0}$ are static solutions to flow \eqref{equ:flow}. We denote $\rho_1$, $\rho_2$ and $\rho(\cdot,t)$ as the radial  functions of  $ \W_{r_1,\omega_0}$,  $ \W_{r_2,\omega_0}$, and $\Sigma_t$ respectively. Then $\log \rho_1,\log \rho_2,\log \rho$ satisfy equation \eqref{equ:flow-graph}.   
	By Lemma \ref{lem:avoidable}, we have
	$\log \rho_1\leq \log \rho\leq \log \rho_2$, thus $
	\Sigma_t \subset \widehat{\W_{ r_2,\omega_0}} \backslash \widehat{\W_{r_1,,\omega_0}}
	$.
\end{proof}

\subsection{$C^1$ estimates}
Next we show the star-shapedness is preserved along flow \eqref{equ:flow}, which would imply the uniform gradient estimate of $\varphi$ in \eqref{equ:flow-graph}.

\begin{proposition}\label{prop:c1}
	 Let $\Sigma_0$ be a star-shaped   anisotropic $\omega_0$-capillary hypersurface in $\overline{\mathbb{R}_{+}^{n+1}}$, and $\omega_0 \in\left(-F\left(E_{n+1}\right), F\left(-E_{n+1}\right)\right)$, then there exists $c_0>0$ depending only on $\Sigma_0$ and Wulff shape $\W$, such that
$$
\langle X, \nu(X)\rangle \geq c_0 ,
$$
for all $X \in \Sigma_t$, $t\in[0,T^*)$.
\end{proposition}

\begin{proof}
	 From \eqref{equ:Lu-} and $|\hat{h}|_{\hat{g}}^2 \geq \frac{H_F^2}{n}$, we see
$$
\begin{aligned}
	\mathcal{L} \bar{u} &
	=n-2\bar{ u}H_F+\bar{ u}^2
	|\hat{h}|_{\hat{g}}^2\quad
	 \bmod \hat{\nabla} \bar{u} \\
	& \geq\left(\frac{\bar{u}H_F}{\sqrt{n}}-\sqrt{n}\right)^2 \geq 0,
\end{aligned}
$$
together with $\hat{\nabla}_{\mu_F} \bar{u}=0$ on $\partial \Sigma_t$, which implies
$$
\bar{u} \geq \min  \bar{u}(\cdot, 0) .
$$
By \cite{Jia-Wang-Xia-Zhang2023}*{Proposition 2.1}, if $\omega_0>0$, we have
$$\<\nu, E^F_{n+1}\>\geq
 -F(\nu)F^0(-E^F_{n+1})
 =
 -F(\nu)F^0\left(\frac{\Psi\left(-E_{n+1}\right)}{F\left(-E_{n+1}\right)}\right)=-\frac{F(\nu)}{F(-E_{n+1})},$$
thus $F(\nu)+\omega_0\langle\nu, E^F_{n+1}\rangle \geq F(\nu)\(
1-\frac{\omega_0}{F(-E_{n+1})}
\)>0$.

If $\omega_0<0$, we have
$$\<\nu, E^F_{n+1}\>\leq
F(\nu)F^0(E^F_{n+1})
=
F(\nu)F^0\(\frac{\Psi\left(E_{n+1}\right)}{F\left(E_{n+1}\right)}\)=\frac{F(\nu)}{F(E_{n+1})},$$
thus $F(\nu)+\omega_0\langle\nu, E^F_{n+1}\rangle \geq F(\nu)\(
1+\frac{\omega_0}{F(E_{n+1})}
\)>0$.

Since $F(\nu)\geq c>0$ is independent of $t$,
we conclude
$$
\langle X, \nu\rangle=\bar{u} \cdot\(F(\nu)+\omega_0\langle\nu, E^F_{n+1}\rangle \) \geq c_0,
$$
for some positive constant $c_0>0$ which is independent of $t$.
\end{proof}

\subsection{Long time existence and convergence}
\begin{proposition}\label{prop:long-time-C^infity}
	If the initial hypersurface $\Sigma_0 \subset \overline{\mathbb{R}_{+}^{n+1}}$ is star-shaped, then flow \eqref{equ:flow} exists for all time. Moreover, it smoothly converges to a uniquely determined $\omega_0$-capillary Wulff shape  $\W_{ r,\omega_0}(E_{n+1}^F)$ given by \eqref{equ:crwo} as $t \rightarrow+\infty$.
\end{proposition}

\begin{proof}
	By Proposition \ref{prop:C0}, we have a uniform bound for $\varphi$. From Proposition \ref{prop:c1}, we have a uniform bound for $v$, and hence a bound for $\nabla^0 \varphi$. Therefore, $\varphi$ is uniformly bounded in $C^1\left(\overline{\mathbb{S}_{+}^n }\times\left[0, T^*\right)\right)$ and the scalar equation in \eqref{equ:flow-graph} is uniformly parabolic. From the standard quasi-linear parabolic theory with a strictly oblique boundary condition theory (cf. \cite{2-pde}*{Chapter 13}, \cite{3-pde}), we conclude the uniform $C^{\infty}$ estimates and the long-time existence.

Next, we show the hypersurfaces  converge to a capillary Wulff shape.

From Lemma \ref{lemma-evolution} in next section,
we know that the volume of the enclosed domain $\widehat{\Sigma}$ is a fixed constant along the flow \eqref{equ:flow}. On the other hand, by \eqref{equ:lemma4.3-2},
\begin{align*}
	\partial_t \mathcal{V}_{1,\omega_0}(\Sigma_t)
	=-\frac{n}{(n+1)(n-1)}\int_{\Sigma}
	\left|\hat{h}-\frac{1}{n} H_F \hat{g}\right|_{\hat{g}}^2
	\hat{u}
	\mathrm{~d}\mu_F.
\end{align*}
Integrating both sides over $[0, \infty)$, we have
$$
\int_0^{\infty} \int_{\Sigma}
\left|\hat{h}-\frac{1}{n} H_F \hat{g}\right|_{\hat{g}}^2
\hat{u}
\mathrm{~d}\mu_F
 \leq \frac{(n-1)(n+1)}{n}\mathcal{V}_{1,\omega_0}(\Sigma_0)<\infty .
$$
As in \cite{Xia-phd}*{page 123}, applying the regularity estimate, the interpolation inequality, and the anisotropic Codazzi equation \eqref{Codazzi}, we have that
\begin{align}
	& \kappa_i^F-\kappa_j^F \rightarrow 0, \quad \text { uniformly as } t \rightarrow \infty, i \neq j \text {, } \nonumber\\
	& H_F-n \kappa_0 \rightarrow 0, \quad\text { uniformly as } t \rightarrow \infty \text {, }
	&\label{equ:pf-converge-3}
\end{align}
for some constant $\kappa_0>0$. Therefore, the regularity estimate implies that there exists a subsequence of times $t_k \rightarrow \infty$ such that $\Sigma_{t_k}$ converges in $C^{\infty}$ to a limit hypersurface $\Sigma_{\infty}$, which is a capillary Wulff shape ${\W_{r_0,\omega_0}}(E_{\infty})$ with $r_0=1 / \kappa_0$.
The radius $r_0$ can be determined using the preserving of the enclosed volume along the flow \eqref{equ:flow}, thus $r_0$ is independent of the choice of subsequent of $t$. 

 At last, we show that any limit of a convergent subsequence is uniquely determined, which implies the flow smoothly converges to a unique capillary Wulff shape. We shall use the similar argument as in \cite{Wang-Weng-Xia,Scheuer-Wang-Xia-2022}
to show  that $E_{\infty}=E_{n+1}^F$.

Denote by $r(\cdot, t)$ the radius of the unique $\omega_0$-capillary Wulff shape $\W_{r(\cdot, t),\omega_0}(E_{n+1}^F)$  passing through the point $X(\cdot, t)\in \Sigma_t$. Due to the $C^0$ estimate, i.e. Proposition \ref{prop:C0}, we know
$$
C_5\leq r_{\max }(t):=\max r(\cdot, t)=r\left(\xi_t, t\right)\leq C_6,
$$
where $C_5,C_6$ are independent of $t$, and $r_{\max}(t)$ is non-increasing with respect to $t$. Hence the limit $\lim\limits_{t \rightarrow+\infty} r_{\max }(t)$ exists. Besides, from $\Sigma_{t_1}\subset \W_{r_{\max}( t_1),\omega_0}(E_{n+1}^F)\subset \W_{r_{\max}( t),\omega_0}(E_{n+1}^F) $ for $\forall t_1\geq t$, we have  ${\W_{r_0,\omega_0}}(E_{\infty})\subset \W_{r_{\max}( t),\omega_0}(E_{n+1}^F)$, which implies $r_{\max}(t)\geq r_0$.
Next we claim that
\begin{align}\label{equ:pf-converge-1}
	\lim\limits_{t \rightarrow+\infty} r_{\max }(t)=r_0 .
\end{align}
We prove this claim by contradiction. Suppose \eqref{equ:pf-converge-1} is not true, then there exists $\varepsilon>0$ such that
\begin{align}\label{equ:pf-converge-5}
	r_{\max }(t)>r_0+\varepsilon, \text { for } t \text { large enough. }
\end{align}
By definition, $r(\cdot, t)$ satisfies
$$F^0\(X-r\omega_0E_{n+1}^F\)=r,$$
that is
\begin{align*}
	1&=(F^0(z))^2=G(z)(z,z),
\end{align*}
where we denote $z=\frac{X}{r}-\omega_0E^F_{n+1}$.
Hence
\begin{align*}
	0=&\partial_tG(z)(z,z)=2G(z)(\partial_t z,z)+Q(z)(\partial_t z,z,z)\nonumber
	\\
	=&2G(z)\(-r^{-2}\partial_trX+r^{-1}\partial_tX,z\),
\end{align*}
then
\begin{align}\label{equ:pf-converge-2}
	G(z)(X,z)\partial_tr=r G(z)(\partial_tX,z).
\end{align}
At $\left(\xi_t, t\right)$, since $\Sigma_t$ is tangential to $\W_{r, \omega_0}(E_{n+1}^F)$ at $X\left(\xi_t, t\right)$, we have
$$
{\nu_F}^{\Sigma_t}\left(\xi_t, t\right)={\nu_F}^{\W_{r, \omega_0}(E_{n+1}^F)}\left(\xi_t, t\right)=\frac{X-r \omega_0 E_{n+1}^F}{r} =z(\xi_t, t),
$$
where ${\nu_F}^{\Sigma_t}$ denotes the anisotropic Gauss map of $\Sigma_t$.  Thus we have the conclusion that $G(z)(X,z)|_{(\xi_t, t)}=\hat{ u}|_{(\xi_t, t)}$.

Since  $\W_{r_{\max }, \omega_0}(E_{n+1}^F)$ is the static solution to \eqref{equ:flow} and $\Sigma_t=X(\cdot, t)$ is tangential to $\W_{r_{\max }, \omega_0}(E_{n+1}^F)$ at $X\left(\xi_t, t\right)$, we see from \eqref{equ:Wulff-1+w<v,e>-}
\begin{align*}
\left.\frac{n
		+n\omega_0G(\nu_F)(\nu_F,E^F_{n+1})}{\hat{u}}
\right|_{X\left(\xi_t, t\right)}
=\left.\frac{n
		+n\omega_0G(\nu_F)(\nu_F,E^F_{n+1})}{\hat{u}}
	\right|_{\W_{r_{\max }, \omega_0}(E_{n+1}^F)}
	=\frac{n}{r_{\max }(t)} .
\end{align*}
Combining with \eqref{equ:pfL3.2-4},  \eqref{equ:pf-converge-2} becomes
\begin{align}\label{equ:pf-converge-4}
	\left. \partial_t r\right|_{\left(\xi_t, t\right)}=\left.\left(\frac{n
	+n\omega_0G(\nu_F)(\nu_F,E^F_{n+1})}{\hat{u}}
-H_F\right)\right|_{(\xi_t,t)}r_{\max }
=\left(
n-H_Fr_{\max}(t)
\right).
\end{align}
 By \eqref{equ:pf-converge-3}, there exist $T_0>0$ large enough, and a constant $C>0$ such that
$$
0<C<H_F
\quad\text{and}\quad
\frac{n}{H_F}-r_0<\frac{\varepsilon}{2},\quad \text{for}\ t>T_0,
$$
and hence
$$
H_F\left(\frac{n}{H_F}-r_{\max }(t)\right)=H_F\left(\frac{n}{H_F}-r_0+r_0-r_{\max }(t)\right)\overset{\eqref{equ:pf-converge-5}}{<}-\frac{\varepsilon}{2}C,
$$
for all $t>T_0$. Putting it into \eqref{equ:pf-converge-4}, we see
$$
\left. \partial_t r\right|_{\left(\xi_t, t\right)}<-\frac{\varepsilon}{2}C,
$$
for all $t>T_0$. By adopting Hamilton's trick, we conclude  that there exists some $C>0$ such that for almost every $t$,
$$
\frac{d}{d t} r_{\max } \leq-\frac{C}{2} \varepsilon.
$$
This is a contradiction to the fact that $ r_{\max }\geq C_5$, and hence claim \eqref{equ:pf-converge-1} is true. Similarly, we can obtain that
$$
\lim\limits_{t \rightarrow+\infty} r_{\min }(t)=r_0.
$$
 This implies that any limit of a convergent subsequence is the  $\omega_0$-capillary Wulff shape $\W_{r_0,\omega_0}(E_{n+1}^F)$. This completes the proof of Theorem \ref{thm:flow}.
\end{proof}

\section{Monotonic quantities and capillary isoperimetric inequality}\label{sec 4}
The aim of this section is to prove Theorem \ref{thm:iso-neq}.
First, we introduce the  following first variation formulas for capillary hypersurface.
\begin{lemma}[\cite{Jost-1988,Koiso2023,Guo-Xia}]
	\label{lemma-evolution-1}
	Let $\Sigma_t \left(t \in J:=\left(-t_0, t_0\right)\right)$ be a smooth variation of $\Sigma_0$ with  ${\partial\Sigma_t}\subset \mathbb{R}^n= \partial\overline{\mathbb{R}_{+}^{n+1}} $. 
	Set
	$$
	\delta X:=\left.\frac{\partial X_t}{\partial t}\right|_{t=0}.
	$$
	Then
	\begin{align}\label{equ:lemma4.1-1}
		\left.\frac{\partial}{\partial_t}\right|_{t=0} \mathcal{V}_{0,\omega_0}(\Sigma_t)
		=\left.\frac{\partial}{\partial_t}\right|_{t=0}
		|\widehat{\Sigma_t}|
		=\int_{\Sigma_0}\<\delta X,\nu\>  \mathrm{~d}\mu_g,
	\end{align}
	\begin{align}\label{equ:lemma4.1-2}
		\left.\frac{\partial}{\partial_t}\right|_{t=0} \mathcal{V}_{1,\omega_0}(\Sigma_t)=\frac{1}{n+1}\left\{\int_{\Sigma_0}H_F
		\<\delta X,\nu\>  \mathrm{~d}\mu_g
		+\int_{\partial\Sigma_0}\<\delta X,\mathcal{R}(\mathcal{P}(\Psi(\nu)))+\omega_0\bar{\nu}\>\mathrm{~d}s\right\},
	\end{align}
	where $\bar{\nu}$  is the unit outward co-normal of $\partial \Sigma\subset \mathbb{R}^n=\partial\overline{\mathbb{R}^{n+1}_+}$, $\mathrm{~d}s$ is the $(n-1)$-dimensional volume form of $\partial \Sigma$ induced by $X$, $\mathcal{R}$ is the $\pi/2$-rotation on the $(E_{n+1}, \nu)$-plane, and $\mathcal{P}$ is the projection from $\mathbb{R}^{n+1}$ into the $(E_{n+1}, \nu)$-plane.
\end{lemma}
	Along flow, we  take the first variation point by point for $t$, then derive the following evolution equations:
\begin{lemma}
	\label{lemma-evolution-2}
	Along flow $\partial_tX=f\nu_F+T$ with $T\in T_X\Sigma_t$ and  ${\partial\Sigma_t}\subset\mathbb{R}^n= \partial\overline{\mathbb{R}_{+}^{n+1}} $, we have
	\begin{align}\label{equ:lemma4.2-1}
		\partial_t \mathcal{V}_{0,\omega_0}(\Sigma_t)=\int_{\Sigma}f \mathrm{~d}\mu_F.
	\end{align}
	If we assume   that   $\<\nu_F(X),E_{n+1}\>=-\omega_0$ for $X\in\partial\Sigma_t$, then we have
	\begin{align}\label{equ:lemma4.2-2}
		\partial_t \mathcal{V}_{1,\omega_0}(\Sigma_t)=\frac{1}{n+1}\int_{\Sigma}H_Ff \mathrm{~d}\mu_F.
	\end{align}
\end{lemma}
\begin{proof}
	(1) Since $\<f\nu_F+T,\nu\>=fF(\nu)$ and $\mathrm{~d}\mu_F=F(\nu)\mathrm{~d}\mu_g $,
	 \eqref{equ:lemma4.1-1} implies \eqref{equ:lemma4.2-1} directly.
	
	(2) On $\partial\Sigma_t$,
	we know $\bar{\nu}\in (E_{n+1},\nu )$-plane. Since $\<\bar{\nu},E_{n+1}\>=0$, $\mathcal{P}(\Psi(\nu))=\mathcal{P}(\nu_F)$ can  be expressed by $\mathcal{P}(\nu_F)=\<\nu_F,E_{n+1}\>E_{n+1}+\<\nu_F,\bar{\nu}\>\bar{\nu}$. By the definition of $\mathcal{R}$, $\mathcal{R}(E_{n+1})=\bar{\nu}$, $\mathcal{R}(\bar{\nu})=-E_{n+1}$, thus
	$$\mathcal{R}(\mathcal{P}(\nu_F))=\<\nu_F,E_{n+1}\>\bar{\nu}-\<\nu_F,\bar{\nu}\>E_{n+1},$$
	which implies
	$$\<\partial_tX,\mathcal{R}(\mathcal{P}(\Psi(\nu)))\>= \<\partial_tX, -\omega_0\bar{\nu}-\<\nu_F,\bar{\nu}\>E_{n+1}\>
	=-\omega_0\<\partial_tX, \bar{\nu}\>,\quad \text{on}\ \partial\Sigma_t.$$
	Then \eqref{equ:lemma4.2-2} follows form \eqref{equ:lemma4.1-2}.
\end{proof}

\begin{lemma}\label{lemma-evolution}
Along flow \eqref{equ:flow}, the enclosed volume $\mathcal{V}_{0,\omega_0}({\Sigma_t})$ is preserved, the anisotropic capillary area  $\mathcal{V}_{1,\omega_0}({\Sigma_t})$ is non-increasing, and $\mathcal{V}_{1,\omega_0}({\Sigma_t})$ is a constant function if and only if  $\Sigma_t$ is an $\omega_0$-capillary Wulff shape.
\end{lemma}
\begin{proof}
 In terms of anisotropic capillary Minkowski Formula \eqref{equ:Minkow}, we obtain by \eqref{equ:lemma4.2-1}, \eqref{equ:lemma4.2-2} and \eqref{equ:pfL3.2-4} that
 \begin{align*}
 	\partial_t \mathcal{V}_{0,\omega_0}(\Sigma_t)=\frac{1}{n+1}\int_{\Sigma}
 	\(
 	n
 	+n\omega_0G(\nu_F)(\nu_F,E^F_{n+1})
 	-\hat{u}H_F
 	\)
 	 \mathrm{~d}\mu_F=0, 
 \end{align*}
and
\begin{align}
	\partial_t \mathcal{V}_{1,\omega_0}(\Sigma_t)
	=&\frac{1}{n+1}\int_{\Sigma}H_F
	\(
	n
	+n\omega_0G(\nu_F)(\nu_F,E^F_{n+1})
	-\hat{u}H_F
	\)
	\mathrm{~d}\mu_F\nonumber
	\\
	=&\frac{1}{n+1}\int_{\Sigma}
	n^2\(
	H_2^F-(H_1^F)^2
	\)\hat{u}
	\mathrm{~d}\mu_F\nonumber
	\\
	=&\frac{1}{n+1}\int_{\Sigma}
	\(
	\frac{2 n}{n-1} \sigma_2\left(\kappa^F\right)-H_F^2
	\)\hat{u}
	\mathrm{~d}\mu_F\nonumber
	\\
	=&-\frac{n}{(n+1)(n-1)}\int_{\Sigma}
	\left|\hat{h}-\frac{1}{n} H_F \hat{g}\right|_{\hat{g}}^2
	\hat{u}
	\mathrm{~d}\mu_F
	\leq 0,\label{equ:lemma4.3-2}
\end{align}
and the equality holds above if and only if  $\Sigma$ is an $\omega_0$-capillary Wulff shape.
Then we complete the proof.
\end{proof}

\begin{proof}[Proof of  Theorem \ref{thm:iso-neq}]
	Taking flow \eqref{equ:flow} with the initial hypersurface $\Sigma$,  from Theorem \ref{thm:flow}, we know it converges to an $\omega_0$-capillary Wulff shape ${\W_{r_0,\omega_0}}$. And by Lemma \ref{lemma-evolution} and $\W_{ r_0,\omega_0}=\{X\in\overline{\mathbb{R}_+^{n+1}}:F^0(X-r_0\omega_0E_{n+1}^F)=r_0\}=\{r_0X\in\overline{\mathbb{R}_+^{n+1}}:F^0(r_0X-r_0\omega_0E_{n+1}^F)=r_0\}=\{r_0X\in\overline{\mathbb{R}_+^{n+1}}:F^0(X-\omega_0E_{n+1}^F)=1\}=r_0\W_{1,\omega_0}$,
	we derive
	\begin{align*}
		 \mathcal{V}_{0,\omega_0}(\Sigma)= |\widehat{\W_{ r_0,\omega_0}}|=\int_{\W_{ r_0,\omega_0}} \<X,\nu\> d\mu_g=r_0^{n+1}|\widehat{\W_{1,\omega_0}}|,
	\end{align*}
	and
	\begin{align*}
		 \mathcal{V}_{1,\omega_0}(\Sigma)&\geq \mathcal{V}_{1,\omega_0}(\W_{ r_0,\omega_0})
		 =\frac{1}{n+1}\left(
		 \int_{\W_{ r_0,\omega_0}}
		 F(\nu)
		 \mathrm{~d}\mu_g
		 +\omega_0|\widehat{\partial \W_{ r_0,\omega_0}}|
		 \right)
		 \\
		 &=\frac{1}{n+1}\left(
		 r_0^n\int_{\W_{1,\omega_0}}
		 F(\nu)
		 \mathrm{~d}\mu_g
		 +r_0^n\omega_0|\widehat{\partial \W_{1,\omega_0}}|
		 \right)
		 \\
		 &=r_0^n\mathcal{V}_{1,\omega_0}(\W_{1,\omega_0}).
	\end{align*}
	Thus
	\begin{align*}
		\mathcal{V}_{1,\omega_0}(\Sigma)\geq
		r_0^n\mathcal{V}_{1,\omega_0}(\W_{1,\omega_0})
		=\left(\frac{\mathcal{V}_{0,\omega_0}(\Sigma)}{\mathcal{V}_{0,\omega_0}(\W_{1,\omega_0})}
		\right)
		^{\frac{n}{n+1}}\mathcal{V}_{1,\omega_0}(\W_{1,\omega_0}),
	\end{align*}
	where equality holds above if and only if equality holds in \eqref{equ:lemma4.3-2}, that means $\Sigma$ is an $\omega_0$-capillary Wulff shape.
	 	Then we complete the proof.
\end{proof}

\section{Proof of convexity preservation}\label{sec 5}In this section, we shall prove Theorem \ref{thm:convex preservation}.
Firstly, we derive some evolution equations to show the upper and lower bounds of anisotropic mean curvature $H_F$.
\subsection{The uniform  bound of $H_F$ along flow}

\begin{lemma}\label{lemma-evolution-HF}
	Along the general flow $\partial_tX=f\nu_F+T$ with $T\in T_X\Sigma_t$, it holds that
	\begin{align}
		\partial_t\hat{g}_{ij}=&2f\hat{h}_{ij}+G(\nu_{{F}})(\partial_iT,X_j)+G(\nu_{{F}})(\partial_jT,X_i)-\hat{\nabla}^pfQ_{ijp}+\hat{h}^{kp}\hat{g}(T,X_k)Q_{ijp},\label{equ:ptg}
		\\
		\partial_t\hat{h}_{ij}=&f\hat{h_i^k}\hat{h}_{jk}-\hat{\nabla}_i\hat{\nabla}_jf+A_{jip}\hat{\nabla}^pf
		+\hat{h}_j^kG(\nu_{{F}})(\partial_iT,X_k)+\hat{h}_i^kG(\nu_{{F}})(\partial_jT,X_k)\nonumber
		\\
		&+
		\hat{g}(T,X^k)\left(\hat{h}_{ij,k}+\hat{h}_j^pA_{pki}+\hat{h}_i^pA_{pjk}+\hat{h}_i^q\hat{h}_j^pQ_{kpq}+\hat{h}_j^q\hat{h}_k^pQ_{ipq}
		\right),\label{equ:pth1}
		\\
		\partial_tH_F=&-f|\hat{h}|^2_{\hat{g}}-\hat{\Delta}f-\hat{g}^{ik}A_{pik}\hat{\nabla}^pf+\hat{\nabla}_TH_F\label{equ:ptHF}.
	\end{align}
\end{lemma}
\begin{proof}
	 In
	normal coordinates (i.e. $\hat{\nabla}_{\partial_i}\partial_j=0$) at a given point, by Lemma \ref{lem2-1} and \eqref{equ:p_tVF}, we compute
	\begin{align*}
		\partial_t\hat{g}_{ij}=&\partial_tG(\nu_{{F}})(\partial_iX,\partial_jX)
		\\
		=
		&G(\nu_{{F}})(\partial_i(\partial_tX),X_j)+G(\nu_{{F}})(X_i,\partial_j(\partial_tX))+Q(\nu_{{F}})(\partial_t\nu_{{F}},X_i,X_j)
		\\
		=
		&G(\nu_{{F}})\(\partial_if\nu_{{F}}+f\partial_i\nu_{{F}}+\partial_iT,X_j\)+G(\nu_{{F}})\(X_i,\partial_jf\nu_{{F}}+f\partial_j\nu_{{F}}+\partial_jT\)
		\\
		&+Q(\nu_{{F}})\(-\hat{\nabla}^pfX_p+\hat{h}^{kp}\hat{g}(T,X_k)X_p,X_i,X_j\)
		\\
		=&f\hat{h}_i^k\hat{g}(X_k,X_j)+G(\nu_{{F}})(\partial_iT,X_j)+	f\hat{h}_j^k\hat{g}(X_k,X_i)+G(\nu_{{F}})(\partial_jT,X_i)
		\\
		&+\left(
		-\hat{\nabla}^pf+\hat{h}^{kp}\hat{g}(T,X_k)
		\right)Q_{ijp}
		\\
		=&2f\hat{h}_{ij}+G(\nu_{{F}})(\partial_iT,X_j)+G(\nu_{{F}})(\partial_jT,X_i)-\hat{\nabla}^pfQ_{ijp}+\hat{h}^{kp}\hat{g}(T,X_k)Q_{ijp}.
	\end{align*}
	This proves \eqref{equ:ptg}.

 To prove \eqref{equ:pth1}, similarly in
normal coordinates at a given point, by Lemma \ref{lem2-1} and \eqref{equ:p_tVF} again, we have
\begin{align}
		\partial_t\hat{h}_{ij}=&\partial_tG(\nu_{{F}})(\partial_iX,\partial_j\nu_{{F}})\nonumber
		\\
		=&G(\nu_{{F}})\(\partial_i(\partial_tX),\partial_j\nu_{{F}}\)+G(\nu_{{F}})\(X_i,\partial_j(\partial_t\nu_{{F}})\)+Q(\nu_{{F}})\(X_i,\partial_j\nu_{{F}},\partial_t\nu_{{F}}\)\nonumber
		\\
		=&G(\nu_{{F}})\(\partial_if \nu_{{F}}+f\partial_i\nu_{{F}}+\partial_iT,\hat{h}_j^kX_k\)\nonumber
		\\
		&+G(\nu_{{F}})\(X_i,\partial_j(-\hat{\nabla}^pfX_p+\hat{h}^{kp}\hat{g}(X_k,T)X_p)\)\nonumber
		\\
		&+Q(\nu_{{F}})\(X_i,\hat{h}_j^qX_q,-\hat{\nabla}^pfX_p+\hat{h}^{kp}\hat{g}(X_k,T)X_p\)\nonumber
		\\
		=&f\hat{h_i^k}\hat{h}_{jk}+\hat{h}_j^kG(\nu_{{F}})(\partial_iT,X_k)
		+\hat{h}_j^q\(-\hat{\nabla}^pf+\hat{h}^{kp}\hat{g}(T,X_k)\)Q_{ipq}+(II)_{ij},
		\label{equ:pf-pth1}
\end{align}
where we denote
\begin{align*}
	(II)_{ij}=G(\nu_{{F}})(X_i,(I)_j)=G(\nu_{{F}})\(X_i,\partial_j(-\hat{\nabla}^pfX_p+\hat{h}^{kp}\hat{g}(X_k,T)X_p)\).
\end{align*}
Since $G(\nu_{{F}})(X_i,\nu_{{F}})=0$, we can eliminate the anisotropic normal  part of $(I)_j$, and calculate $(I)_j$ as
\begin{align*}
	(I)_j=&-\partial_j\(\hat{\nabla}^pf\)X_p+\partial_j\(\hat{h}^{kp}\hat{g}(X_k,T)\)X_p+\(-\hat{\nabla}^pf+\hat{h}^{kp}\hat{g}(T,X_k)\)A_{jpl}X^l
	\\
	=&-\hat{\nabla}_j\hat{\nabla}^pfX_p+\hat{\nabla}_j\hat{h}^{kp}\hat{g}(X_k,T)X_p
	+\left(-\hat{\nabla}^lf+\hat{h}^{kl}\hat{g}(T,X_k) \right)A_{jlp}X^p
	\\
	&+\hat{h}^{kp}\(A_{kjl}\hat{g}(X^l,T)+Q(\nu_{{F}})(X_k,T,\hat{h}_j^lX_l)+G(\nu_{{F}})(X_k,\partial_jT)\)X_p,
\end{align*}
then
\begin{align}
	(II)_{ij}=&-\hat{\nabla}_j\hat{\nabla}_if+\hat{\nabla}_j\hat{h}_i^k\hat{g}(X_k,T)-\hat{\nabla}^pfA_{jpi}+\hat{h}^{kp}\hat{g}(X_k,T)A_{jpi}\nonumber
	\\
	&+\hat{h}_i^k\left(
	A_{kjl}\hat{g}(X^l,T)+G(\nu_{{F}})(X_k,\partial_jT)+\hat{h}_j^lQ_{kql}\hat{g}(T,X^q)
	\right).\label{equ:pf-pth2}
\end{align}
By Codazzi equation \eqref{Codazzi} and Lemma \ref{lem2-1}, we have
\begin{align}\label{equ:pf-pth3}
	&\left(	\hat{h}^{kp}A_{jpi}+\hat{h}_j^q\hat{h}^{kp}Q_{ipq}+\hat{\nabla}_j\hat{h}_i^k\right)\hat{g}(T,X_k)+\hat{h}_i^kA_{kjl}\hat{g}(T,X^l)+\hat{h}_i^k\hat{h}_j^lQ_{kql}\hat{g}(T,X^q)
	\nonumber
	\\
	=&\hat{g}(T,X^k)\left(
	\hat{h}_{ij,k}+\hat{h}_j^pA_{pki}+\hat{h}_i^pA_{pjk}+\hat{h}_i^q\hat{h}_j^pQ_{qkp}+\hat{h}_j^q\hat{h}_k^pQ_{ipq}
	\right).
\end{align}
Combining with \eqref{equ:pf-pth1}$\thicksim$\eqref{equ:pf-pth3}, we obtain \eqref{equ:pth1} by \eqref{equ:pf-lemma3.3-2}.

For the last evolution \eqref{equ:ptHF},  we see by Lemma \ref{lem2-1}, \eqref{equ:ptg} and \eqref{equ:pth1} that
\begin{align}
	\partial_t\hat{h}_i^j =&\partial_t\hat{g}^{jk}\hat{h}_{ki}+\hat{g}^{jk}\partial_t\hat{h}_{ik}
	=-\hat{g}^{jp}\hat{g}^{kq}\hat{h}_{ki}\partial	_t\hat{g}_{pq}+\hat{g}^{jk}\partial_t\hat{h}_{ik}\nonumber
	\\
	=&-\hat{g}^{jp}\hat{g}^{kq}\hat{h}_{ki}\left(
	2f\hat{h}_{pq}+G(\nu_{{F}})(\partial_pT,X_q)+G(\nu_{{F}})(\partial_qT,X_p)+\hat{h}^{ml}\hat{g}(T,X_m)Q_{pql}\right.
	\nonumber
	\\
	&\left.-\hat{\nabla}^lfQ_{pql}	
	\right)
	+\hat{g}^{jk}\left(
	f\hat{h_i^l}\hat{h}_{kl}-\hat{\nabla}_i\hat{\nabla}_kf+\hat{h}_i^lG(\nu_{{F}})(\partial_kT,X_l)
	+\hat{h}_k^lG(\nu_{{F}})(\partial_iT,X_l)
	\right.\nonumber
	\\
	&
	\left.+A_{kip}\hat{\nabla}^pf
	+\hat{g}(T,X^l)(\hat{h}_{ik,l}+\hat{h}_k^pA_{pli}+\hat{h}_i^pA_{pkl}+\hat{h}_i^q\hat{h}_k^pQ_{lpq}+\hat{h}_k^q\hat{h}_l^pQ_{ipq}
	)
	\right)\nonumber
	\\
	=&-f\hat{h}_i^k\hat{h}_k^j-\hat{\nabla}_i\hat{\nabla}^jf-\hat{g}^{jk}A_{pik}\hat{\nabla}^pf+\hat{g}(T,X^l)\hat{h}^j_{i,l}+(III)_i^j\nonumber
	\\
	&
	+\hat{g}^{jp}\left(\hat{h}_p^lG(\nu_{{F}})(\partial_iT,X_l)+\hat{h}_i^lG(\nu_{{F}})(\partial_pT,X_l)\right)
	\nonumber
	\\
	&-\hat{g}^{jp}\hat{g}^{kq}\hat{h}_{ki}\(
	G(\nu_{{F}})(\partial_pT,X_q)+G(\nu_{{F}})(\partial_qT,X_p)\), \label{equ:pf-ptHF1}
\end{align}
where
\begin{align*}
	(III)_i^j=\hat{g}(T,X^l)\hat{g}^{jk}\left(
	-\hat{h}_l^pQ_{kpq}\hat{h}_{i}^q
	+\hat{h}_k^pA_{pli}+\hat{h}_i^pA_{pkl}+\hat{h}_i^q\hat{h}_k^pQ_{lpq}+\hat{h}_k^q\hat{h}_l^pQ_{ipq}
	\right).
\end{align*}
By \eqref{equ:pf-lemma3.3-2}, 
we have $\sum_{i=1}^{n}(III)_i^i=0$. Combining with \eqref{equ:pf-ptHF1} and $H_F=\sum_{i=1}^{n}\hat{h}_i^i$, we can obtain \eqref{equ:ptHF}.
\end{proof}
We still use the operator $\mathcal{L}$ defined by  \eqref{equ:L-operator}, to calculate the further evolution equation.
\begin{lemma}
	Along flow \eqref{equ:flow}, the anisotropic mean curvature $H_F$ satisfies
	\begin{align}\label{equ:LHF}
		\mathcal{L}H_F=2G(\nu_{{F}})(\hat{\nabla
		}H_F,\hat{\nabla}\hat{ u})+H_F^2-n|\hat{h}|^2_{\hat{g}}.
	\end{align}
\end{lemma}
\begin{proof}
	Putting \eqref{equ:pfL3.2-4}, \eqref{equ:pf-lemma3.3-4}, and \eqref{equ:pfL3.2-5} into \eqref{equ:ptHF} yields
	\begin{align*}
		\partial_tH_F=&-\(n+n\omega_0G(\nu_{{F}})(\nu_{{F}},E_{n+1}^F)-\hat{u}H_F\) |\hat{h}|^2_{\hat{g}}+\hat{u}\hat{\Delta}H_F+H_F\hat{\Delta}\hat{u}
		\\
		&+2G(\nu_{{F}})( \hat{\nabla}H_F,\hat{\nabla}\hat{u})
		-n\omega_0\left(G(\nu_F)(\hat{\nabla}H_F,E^F_{n+1})	
		-
		|\hat{h}|^2_{\hat{g}}G(\nu_F)(\nu_F,E^F_{n+1})\right.
		\\
		&\left.-\hat{g}^{kl}A_{plk}\hat{\nabla}^pG(\nu_F)(\nu_F,E^F_{n+1})\right)
		-\hat{g}^{ik}A_{pik}\left(
		n\omega_0\hat{h}^{pl}G(\nu_F)(\partial_lX,E^F_{n+1})\right.
		\\
		&\left.-\hat{\nabla}^pH_F\hat{ u}-\hat{\nabla}^p\hat{ u}H_F
		\right)+\hat{\nabla}_TH_F
		\\
		\overset{\eqref{equ:pfL3.2-2}}{=}&
		-n|\hat{h}|^2_{\hat{g}}
		+G(\nu_{{F}})\(-n\omega_0E_{n+1}^F+T+H_F X,\hat{\nabla}H_F\)-\hat{g}^{ik}A_{pik}H_F\hat{\nabla}^p\hat{u}
		\\
		&+H_F^2+\hat{u}\hat{\Delta}H_F
		+2G(\nu_{{F}})(\hat{\nabla}H_F,\hat{\nabla}\hat{u})+\hat{g}^{ik}A_{pik}\hat{\nabla}^p(\hat{u}H_F)
		\\
		=&\hat{u}\left(\hat{\Delta}H_F+\hat{g}^{ik}A_{pik}\hat{\nabla}^pH_F\right)+G(\nu_{{F}})(T+H_FX-n\omega_0E_{n+1}^F ,\hat{\nabla}H_F)
		\\
		&+2G(\nu_{{F}})(\hat{\nabla}H_F,\hat{\nabla}\hat{u})+H_F^2-n|\hat{h}|^2_{\hat{g}}.
	\end{align*}
	This is equation \eqref{equ:LHF}.
\end{proof}
\begin{lemma}\label{lem:DuH_F}
		For  $X\in \partial\Sigma_t$, we take $\mu_F\in T_X\Sigma$, which satisfies $G(\nu_F)(\xi,\mu_F)=0$ for any $\xi\in T_X(\partial\Sigma)$. Then
	along flow \eqref{equ:flow}, we have
	\begin{align*}
		\hat{\nabla}_{\mu_F}H_F=0, \quad \text{on}\  \partial\Sigma_t.
	\end{align*}
\end{lemma}
\begin{proof}
	Since $f=nF(\nu)^{-1}\left(
	F(\nu)+\omega_0\<\nu,E^F_{n+1}\>-\<X,\nu\>H^F_1
	\right)$, we know
	\begin{align}
		H_F=\frac{n}{\bar{ u}}-\frac{\hat{f}}{\<X,\nu\>},
	\end{align}
	where $\hat{f}=fF(\nu)=\<\partial_tX,\nu\>$. On the boundary of $\Sigma_t$, from Lemma \ref{lemma-DU-=0}, we have $\hat{\nabla}_{\mu_F}\bar{u}
	=0$, then we only need to check $\hat{\nabla}_{\mu_F}\left(\frac{\hat{f}}{\<x,\nu\>}\right)
	=0$.
	
	From Remark \ref{rk3.5} (1), we can take $$\mu_F=A_F(\nu)\mu.$$
	From \cite{Guo-Xia}*{Proposition 4.4}, we know that along anisotropic capillary hypersurface $\Sigma_t$,
	\begin{align*}
		& \left\langle A_F(\nu) \nabla\langle X, \nu\rangle, \mu\right\rangle=q_F\langle X, \nu\rangle,\quad\text{for\ } X\in\partial\Sigma_t, 
	\end{align*}
	where $\nabla$ is  the Levi-Civita connection  on $(\Sigma,g)$, and $q_F$ is defined by \cite{Guo-Xia}*{Eq. (4.1)} as
	$$
	q_F=-\frac{\<\nu,E_{n+1}\>}{\<\mu,E_{n+1}\>}\<S_F\mu,\mu\>.
	$$
	On the other hand, we have the formula \eqref{equ:DT} in Section \ref{sec 6}, that is
	\begin{align}
		\< A_F(\nu)\nabla \hat{f},\mu\>	=\<S_F\mu,\mu\> \<\partial_tX,\mu\>.
	\end{align}
	Then, we can calculate that
	\begin{align*}
		\hat{\nabla}_{\mu_F}\left(\frac{\hat{f}}{\<X,\nu\>}\right)=&{\nabla}_{\mu_F}\left(\frac{\hat{f}}{\<X,\nu\>}\right)=\left\langle {\mu_F},\nabla\left(\frac{\hat{f}}{\<X,\nu\>}\right)\right\rangle
		\\
		=&\<X,\nu\>^{-2}\<A_F(\nu)\mu,\<X,\nu\>\nabla\hat{f}-\hat{f}\nabla\<X,\nu\>\>
		\\
		=&\<X,\nu\>^{-2}\left(
		\<X,\nu\>\<S_F\mu,\mu\>\<\partial_tX,\mu\>-q_F\<X,\nu\>\hat{f}
		\right)
		\\
		=&\<X,\nu\>^{-1}\left\langle \partial_tX,\mu+\frac{\<\nu,E_{n+1}\>}{\<\mu,E_{n+1}\>}\nu\right\rangle \<S_F\mu,\mu\>
		\\
		=&\<X,\nu\>^{-1}{\<\mu,E_{n+1}\>}^{-1}\<\partial_tX,E_{n+1}\>\<S_F\mu,\mu\>
		\\
		=&0,
	\end{align*}
	where we use  the symmetry of $A_F(\nu)$, $\hat{f}=\<\partial_tX,\nu\>$, and  $\partial_tX\in\partial\mathbb{R}^{n+1}_+$ (since the restriction of $X(\cdot, t)$ on $\partial M$ is contained in $\partial\mathbb{R}^{n+1}_+$).
	Then we complete the proof.
\end{proof}

\begin{proposition}\label{prop:H_F<c}
	If $\Sigma_t$ solves flow \eqref{equ:flow}, then the anisotropic mean curvature is uniformly bounded from above, that is,
	$$
	H_F(p, t) \leq \max H_F(\cdot, 0), \quad \forall(p, t) \in M \times\left[0, \infty\right) .
	$$
\end{proposition}
\begin{proof}
	 By \eqref{equ:LHF} and $n|\hat{h}|_{\hat{g}}^2 \geq H_F^2$, we have
$$
\mathcal{L} H_F \leq 0, \quad \bmod\  \hat{\nabla} H_F,
$$
combining with $\hat{\nabla}_{\mu_F} H_F=0$ on $\partial \Sigma$,  the conclusion follows directly from the maximum principle.
\end{proof}
In the rest of this section, we assume the  initial hypersurface is  strictly convex, then the solution
$\Sigma_t$ are  strictly convex hypersurfaces for $t\in[0,T_0)$ for a constant $T_0>0$.
The lower bound of $H_F$ will be derived with strictly convex solution
$\Sigma_t$ in $t\in[0,T_0)$.
\begin{proposition}\label{prop:H>c}
	 If $\Sigma_t$ solves flow \eqref{equ:flow} with initial hypersurface $\Sigma_0$ being a strictly  convex 
	  anisotropic  $\omega_0$-capillary hypersurface in $\overline{\mathbb{R}_{+}^{n+1}}$, then
$$
H_F(p, t) \geq C, \quad \forall(p, t) \in M \times\left[0, T_0\right],
$$
where the positive constant $C$ depends on the initial hypersurface $\Sigma_0$ and Wulff shape $\W$.
\end{proposition}

 \begin{proof}
 Define the function
$$
P:=H_F \bar{u} .
$$
 By Lemma \ref{lem:DuH_F} and \ref{lemma-DU-=0}, we see that
\begin{align}
	\hat{\nabla}_{\mu_F} P=0, \quad \text { on } \partial M .\label{equ:w-x-3.20}
\end{align}
Using \eqref{equ:LHF} and \eqref{equ:Lu-}, we have
\begin{align}
	\mathcal{L} P & =\bar{u} \mathcal{L} H_F+H_F \mathcal{L} \bar{u}-2 \hat{u}G(\nu_{{F}})(\hat{\nabla} \bar{u}, \hat{\nabla} H_F)\nonumber \\
	& =\bar{u}^2 H_F|\hat{h}|_{\hat{g}}^2-\bar{u} H_F^2-n \bar{u}|\hat{h}|_{\hat{g}}^2+n H_F+2 \bar{u}\hat{g}\left(\hat{\nabla}(1+\omega_0 G(\nu_{{F}})(\nu_F, E_{n+1}^F)), \hat{\nabla} P \right).\label{equ:w-x-3.21}
\end{align}
From \eqref{equ:w-x-3.20} and the Hopf Lemma, if $P$ attains the minimum value at $t=0$, then the conclusion follows directly by combination with the uniform bound of $\bar{u}$ (from $\bar{ u}\leq C|X|$ and $C^0$ estimate). Therefore, we assume that $P$ attains the minimum value at some interior point $p_0 \in \operatorname{int}(M)$ for some $t_0>0$. At $(p_0, t_0)$, we have
$$
\hat{\nabla} P=0, \quad \mathcal{L} P \leq 0 .
$$
Putting it into \eqref{equ:w-x-3.21} yields
\begin{align}\label{equ:w-x3.22}
	\left(\frac{n}{H_F}-\bar{u}\right) \bar{u}|\hat{h}|_{\hat{g}}^2+\bar{u} H_F \geq n .
\end{align}
If $\bar{u} H_F \geq n$ at $p_0$, then we are done. Assume now that $\bar{u} H_F<n$ at $p_0$,  since $\Sigma_t$ is convex for $t\in [0,T_0]$, we have $|\hat{h}|_{\hat{g}}^2 \leq n H_F^2$, putting it into  \eqref{equ:w-x3.22}, we can see
$$
\bar{u} H_F\left(p_0, t\right) \geq c,\quad t\in [0,T_0],
$$
for some positive constant $c$, which only depends on $n$. This also yields the desired estimate.
Then we complete the proof.
 \end{proof}

 \subsection{Reparameterize the flow by the translated Wulff shape $\widetilde{\W}$}\label{subsec 5.2}
 For a given Wulff shape $\W$ (with respect to $F$) and   a constant  $\omega_0 \in(-F(E_{n+1}), F(-E_{n+1}))$, we can define a new Wulff shape $\widetilde{\W}:=\rm \mathcal{T}(\W)\in\mathbb{R}^{n+1}$, where $\mathcal{T}$ is a translation  transformation on $\mathbb{R}^{n+1}$ defined by
 $$\mathcal{T}(x)=x+\omega_0E^F_{n+1},\quad \forall x\in\mathbb{R}^{n+1}.$$
 Since $\omega_0 \in(-F(E_{n+1}), F(-E_{n+1}))$, we have $F^0(O-\omega_0E^{F}_{n+1})=F^0(-\omega_0E^{F}_{n+1})<1$, then the translated Wulff shape $\widetilde{\W}=\{x\in\mathbb{R}^{n+1}:F^0(x-\omega_0E^{F}_{n+1})=1\}$ still encloses the origin $O$ in its interior.
 \begin{figure}[h]
 	\centering
 	\label{fig:2}
 	\begin{tikzpicture}[x=0.75pt,y=0.75pt,yscale=-1,xscale=1]
 		
 		\draw   (96.2,78.6) .. controls (125.2,67.6) and (229,37) .. (260.2,61.6) .. controls (291.4,86.2) and (326.4,138.2) .. (284.2,158.6) .. controls (242,179) and (178,158) .. (135.2,141.6) .. controls (92.4,125.2) and (67.2,89.6) .. (96.2,78.6) -- cycle ;
 		\draw  [dash pattern={on 4.5pt off 4.5pt}]  (220,21) -- (355.2,152.6) ;
 		\draw  [dash pattern={on 4.5pt off 4.5pt}]  (174,45) -- (310.2,180.6) ;
 		\draw    (224,95) -- (241.59,77.41) ;
 		\draw [shift={(243,76)}, rotate = 135] [fill={rgb, 255:red, 0; green, 0; blue, 0 }  ][line width=0.08]  [draw opacity=0] (12,-3) -- (0,0) -- (12,3) -- cycle    ;
 		\draw    (270,71) -- (287.59,53.41) ;
 		\draw [shift={(289,52)}, rotate = 135] [fill={rgb, 255:red, 0; green, 0; blue, 0 }  ][line width=0.08]  [draw opacity=0] (12,-3) -- (0,0) -- (12,3) -- cycle    ;
 		\draw [color={rgb, 255:red, 208; green, 2; blue, 27 }  ,draw opacity=1 ]   (224,95) -- (268.23,71.93) ;
 		\draw [shift={(270,71)}, rotate = 152.45] [fill={rgb, 255:red, 208; green, 2; blue, 27 }  ,fill opacity=1 ][line width=0.08]  [draw opacity=0] (12,-3) -- (0,0) -- (12,3) -- cycle    ;
 		\draw    (232.2,145.6) -- (224.32,96.97) ;
 		\draw [shift={(224,95)}, rotate = 80.79] [fill={rgb, 255:red, 0; green, 0; blue, 0 }  ][line width=0.08]  [draw opacity=0] (12,-3) -- (0,0) -- (12,3) -- cycle    ;
 		\draw    (232.2,145.6) -- (269.1,72.78) ;
 		\draw [shift={(270,71)}, rotate = 116.87] [fill={rgb, 255:red, 0; green, 0; blue, 0 }  ][line width=0.08]  [draw opacity=0] (12,-3) -- (0,0) -- (12,3) -- cycle    ;
 		\draw   (52.6,105.3) .. controls (81.6,94.3) and (185.4,63.7) .. (216.6,88.3) .. controls (247.8,112.9) and (282.8,164.9) .. (240.6,185.3) .. controls (198.4,205.7) and (134.4,184.7) .. (91.6,168.3) .. controls (48.8,151.9) and (23.6,116.3) .. (52.6,105.3) -- cycle ;
 		\draw  [dash pattern={on 4.5pt off 4.5pt}]  (232.2,145.6) -- (290.2,91.6) ;
 		\draw   (232.2,145.6) .. controls (235.47,149.93) and (238.77,149.96) .. (242.1,146.69) -- (269.99,119.31) .. controls (274.74,114.64) and (278.75,113.96) .. (282.02,117.29) .. controls (278.75,113.96) and (279.5,109.96) .. (284.26,105.29)(282.11,107.39) -- (288.11,101.5) .. controls (293.44,98.23) and (293.47,94.93) .. (291.2,91.6) ;
 		
 		\draw (215,110) node [anchor=north west][inner sep=0.75pt]   [align=left] {$z$};
 		\draw (244,93) node [anchor=north west][inner sep=0.75pt]   [align=left] {$\tilde{z}$};
 	\draw (32,91) node [anchor=north west][inner sep=0.75pt]   [align=left] {$\W$};
 	\draw (60,35) node [anchor=north west][inner sep=0.75pt]   [align=left] {$\widetilde{\W}=\W+\omega_0E_{n+1}^F$};
 	\draw (226,70) node [anchor=north west][inner sep=0.75pt]   [align=left] {$x$};
 	\draw (282,40) node [anchor=north west][inner sep=0.75pt]   [align=left] {$x$};
 	\draw (280,65) node [anchor=north west][inner sep=0.75pt]   [align=left] {\color{red}$\omega_0E_{n+1}^F$};
 	\draw (221,146) node [anchor=north west][inner sep=0.75pt]   [align=left] {$O$};
 	\draw (270,120) node [anchor=north west][inner sep=0.75pt]   [align=left] {$\widetilde{F}(x)$};
 	\draw (295,178) node [anchor=north west][inner sep=0.75pt]   [align=left] {$T_z\W=T_x\mathbb{S}^n$};
 	\draw (343,154) node [anchor=north west][inner sep=0.75pt]   [align=left] {$T_{\tilde{z}}\widetilde{\W}=T_x\mathbb{S}^n$};

 \end{tikzpicture}
 \caption{$\W$ and $\widetilde{\W}$} \label{fig:0}
 \end{figure}
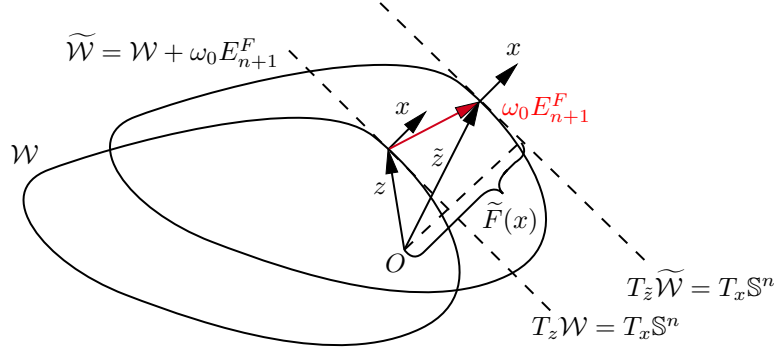 \\

 For any $z\in\W$, we have $\tilde{z}:=z+\omega_0E_{n+1}^F\in \widetilde{\W}$. If $x$ is the unit outward normal vector of $\W\subset\mathbb{R}^{n+1}$ at point $z$, then $x$ is also a unit outward normal vector of $\widetilde{\W}\subset\mathbb{R}^{n+1}$ at point $\tilde{z}$,  since $Tz\W=T_{\tilde{z}}\widetilde{\W}$ (see Fig. \ref{fig:0}). Then
 the support function $\widetilde{F}$ of $\widetilde{\W}$ satisfies
 \begin{align*}
    \widetilde{F}(x)&=\<\tilde{z},x\>=\<z,x\>+\<x,\omega_0E_{n+1}^F\>
    \\
    &=F(x)+\omega_0\<E_{n+1}^F,x\>, \quad\forall x\in \mathbb{S}^n.
 \end{align*}
  We extend it to a $1$-homogeneous function on $\mathbb{R}^{n+1}$ as
 $$\widetilde{F}(X)=F(X)+\omega_0\<E_{n+1}^F,X\>, \quad\forall X\in \mathbb{R}^{n+1}.$$
 Then $D\widetilde{F}=DF+\omega_0E_{n+1}^F, D^2\widetilde{F}=D^2F$.
 Obviously the translation map does not change the principal curvature radii of $\W$ and $\widetilde{\W}$, i.e., ${\widetilde{A_F}}(x):=D^2\widetilde{F}|_{T_x\mathbb{S}^n}(x)=D^2F|_{T_x\mathbb{S}^n}(x)=A_F(x)$, for $x\in\mathbb{S}^n$.
 We denote $\widetilde{F}^{0}$ as the  dual Minkowski norm of $\widetilde{F}$. The new metric $\widetilde{G}$ and $(0,3)$-tensor $\widetilde{Q}$ with respect to $\widetilde{F}^0$ are respectively constructed by \eqref{equ:G} and \eqref{equ:Q} in Section \ref{subsec 2.3}. The relationships between $G$ and $\widetilde{G}$, $Q$ and $\widetilde{Q}$  can be found in Appendix \ref{Appendix} (Proposition \ref{prop:Q-Q}).

We use a tilde $(\ \widetilde{}\ )$ to indicate the related anisotropic geometric quantities with respect to the new Wulff shape $\widetilde{\W}$ in this subsection.  For  $X\in\Sigma$, denote  anisotropic Gauss map with respect to $\widetilde{\W}$ as: $$\widetilde{\nu_{{F}}}(X)=D\widetilde{F}(\nu(X))=DF(\nu)+\omega_0E_{n+1}^F=\mathcal{T}(\nu_{{F}}(X)).$$
Anisotropic support function with respect to $\widetilde{\W}$ is:
\begin{align}\label{equ:SF-Translation invariant}
	\widetilde{u}(X)=\frac{\<X,\nu\>}{\widetilde{F}(\nu)}=\frac{\<X,\nu\>}{F(\nu)+\omega_0\<E_{n+1}^F,\nu\>}=\bar{ u}(X).
\end{align}
Anisotropic Weingarten matrix with respect to $\widetilde{\W}$ is:
$$\widetilde{S_F}(X)=\widetilde{A_F}\operatorname{d}\nu={A_F}\operatorname{d}\nu=S_F(X),$$
that means the anisotropic principal curvatures (or principal curvature radii) of $\Sigma$ with respect to $\W$  are the same as anisotropic principal curvatures (or principal curvature radii) of $\Sigma$ with respect to $\widetilde{\W}$, i.e.,  $\widetilde{\kappa^F}(X)=\kappa^F(X), \tilde{\tau}(X)=\tau(X)$.

On the boundary of  $\omega_0$-capillary hypersurface, since $$\<\widetilde{\nu_F},-E_{n+1}\>=\<{\nu_F}+\omega_0E_{n+1}^F,-E_{n+1}\>=\<{\nu_F},-E_{n+1}\>-\omega_0=0,$$ we have:
\begin{proposition}
	The anisotropic $\omega_0$-capillary boundary problem \eqref{equ:flow} for Wulff shape $\W$ can be naturally solved based on the results of the corresponding free boundary problem for the translated Wulff shape $\widetilde{\W}$.
\end{proposition}

  Let $\Sigma$ be a strictly convex anisotropic  $\omega_0$-capillary hypersurface, we know  $\widetilde{\W}$ is also a strictly convex hypersurface, then $\widetilde{\nu_{{F}}}$ is an everywhere nondegenerate diffeomorphism. We use it to reparameterize $\Sigma$:
 $$X:\widetilde{\W}\rightarrow\Sigma,\quad X(\tilde{z})=X(\widetilde{\nu_F}^{-1}(\tilde{z})),\  \tilde{z}\in \widetilde{\W}.$$
For convenience, denote by $s=\hat{ u}\circ\nu_{{F}}^{-1},\tilde{s}=\tilde{u}\circ\widetilde{\nu_F}^{-1}$, then from \eqref{equ:SF-Translation invariant} we have
  \begin{align}\label{equ:s-trandlation}
  \tilde{s}\circ \mathcal{T}(z)=\frac{s(z)}{1+G(z)(z,E_{n+1}^F)}
  , \quad z\in\W.
  \end{align}
 If we still write $\kappa^F=\kappa^F\circ \nu_{{F}}^{-1},\widetilde{\kappa^F}=\widetilde{\kappa^F}\circ \widetilde{\nu_{{F}}}^{-1},\tau=\tau\circ\nu_{{F}}^{-1}$, and $\tilde{\tau}=\tilde{\tau}\circ\widetilde{\nu_{{F}}}^{-1}$, then
  $$\widetilde{\kappa^F}\circ \mathcal{T}=\kappa^F, \quad\tilde{\tau}\circ \mathcal{T}=\tau.$$
 We denote $\tilde{s}_i=\tilde{\nabla}_{\tilde{e}_i}s,\, \tilde{s}_{ij}=\tilde{\nabla}_{\tilde{e}_i}\tilde{\nabla}_{\tilde{e}_j}s,\, \tilde{Q}_{ijk}=\tilde{Q}(\tilde{z})(\tilde{e}_i,\tilde{e}_j,\tilde{e}_k)$, where $\tilde{\nabla}$ is the covariant derivative on $\widetilde{\W}$ with respect to metric $\tilde{g}=\widetilde{G}(\tilde{z})|_{T_{\tilde{z}}\widetilde{\W}}$, and $\{\tilde{e}_i\}_{i=1}^n$ are standard orthogonal frame of $(\widetilde{\W},\tilde{g},\tilde{\nabla})$.
 From \cite{Xia-2017-convex,Xia13}, the anisotropic principal curvatures at $X(\tilde{z})\in\Sigma$ are the reciprocals of  the eigenvalues of
 \begin{align}\label{equ:tau_ij}
 	\tilde{\tau}_{ij}(\tilde{z}):=\tilde{s}_{ij}-\frac{1}{2}\tilde{Q}_{ijk}\tilde{s}_k+\delta_{ij}\tilde{s}>0, \quad \tilde{z}\in\widetilde{\W}.
 \end{align}
 Let  \begin{align}\label{equ:psi-T-def}
 \Phi([\tilde{\tau}_{ij}])=\phi(\tilde{\tau})=n\frac{\sigma_{n}(\tilde{\tau})}{\sigma_{n-1}(\tilde{\tau})}.
 \end{align}
 Since $\tau(z)=\tilde{\tau}\circ\mathcal{T}(z)$, we have
 \begin{align}\label{equ:psi-T}
 	\Phi([\tilde{\tau}_{ij}(\tilde{z})])=\phi({\tilde{\tau}}(\tilde{z}))=\phi({\tau}({z}))=\frac{1}{H_1^F}(z),
 \end{align} where ${z}=\mathcal{T}^{-1}(\tilde{z})\in{\W}$.

 Denote by $\dot{\Phi}^{k \ell}=\frac{\partial\Phi}{\partial\tilde{\tau}_{k \ell}},\ \ddot{\Phi}^{k \ell, p q}=\frac{\partial^2\Phi}{\partial\tilde{\tau}_{k \ell}\partial\tilde{\tau}_{pq}},\dot{\phi}^k=\frac{\partial \phi}{\partial \tilde{\tau}_k}$, and $\ddot{\phi}^{k\ell}=\frac{\partial^2 \phi}{\partial \tilde{\tau}_{k}\partial\tilde{\tau}_{\ell}}.$
We review the following basic properties \cite{Andrews,Wei-Xiong2022} on the symmetric functions $\Phi$ and $\phi$, which will be used  later.
  \begin{lemma} \label{lem:w-x-lem-5.3}
 	The function
 	$
 	\phi(\tilde{\tau})
 	$ defined by \eqref{equ:psi-T-def}
 	is concave with respect to $\tilde{\tau}$, and satisfies
 	\begin{align}\label{equ:pf-5.6}
 		\dot{\Phi}^{k \ell} \tilde{g}_{k \ell}=\sum_k \dot{\phi}^k \geq 1,
 	\end{align}
 	and
 	$$
 	\left(\ddot{\phi}^{i j}\right) \leq 0, \quad\left(\dot{\phi}^k(\tilde{\tau})-\dot{\phi}^{\ell}(\tilde{\tau})\right)\left(\tilde{\tau}_k-\tilde{\tau}_{\ell}\right) \leq 0, \quad \forall k \neq \ell.
 	$$
 \end{lemma}
\begin{lemma}\label{lem:w-x-lem-5.2}
	 Let $\operatorname{Sym}(n)$ be the set of $n \times n$ symmetric matrices, and let $\Phi(A)=\phi(\tau)$, where $\phi$ is a smooth symmetric function and $A$ is in $\operatorname{Sym}(n)$ with eigenvalues $\tau=\left(\tau_1, \cdots, \tau_n\right)$. If $A$ is diagonal, the first derivatives of $\Phi(A)$ and $\phi(\tau)$ with respect to their arguments are related by the following equation
 $$
 \dot{\Phi}^{i j}(A)=\dot{\phi}^i(\tau) \delta_i^j.
 $$
 The second derivative of $\Phi$ in the direction $B \in \operatorname{Sym}(n)$ satisfies
 \begin{align}\label{equ:pf-5.5}
 	\ddot{\Phi}^{i j, k \ell}(A) B_{i j} B_{k \ell}=\ddot{\phi}^{k \ell}(\tau) B_{k k} B_{\ell \ell}+2 \sum_{k<\ell} \frac{\dot{\phi}^k(\tau)-\dot{\phi}^{\ell}(\tau)}{\tau_k-\tau_{\ell}}\left(B_{k \ell}\right)^2 .
 \end{align}
 This formula makes sense as a limit in the case of any repeated values of $\tau_k$.
\end{lemma}
 For anisotropic $\omega_0$-capillary hypersurface $\Sigma$,  $\nu_{{F}}(\partial\Sigma)=\{\nu_{{F}}\in\W|\<\nu_{{F}},-E_{n+1}\>=\omega_0\}=\W\cap\{x_{n+1}=-\omega_0\}$ and $\nu_{{F}}(\Sigma)=\W\cap\{x_{n+1}\geq-\omega_0\}$.
 Since  $\<X,E_{n+1}\>=0$ for $X\in\partial\Sigma\subset\mathbb{R}^n$, and   $X(\nu_{{F}}^{-1}(z))=s(z)z+\hat{\nabla}^0 s(z)$ from \cite{Xia13}*{(3.3)},  here $\hat{\nabla}^0$ is the  Levi-Civita connection of $G(z)|_{T_z\W}$ on $\W$, we have $0=\<X(\nu_{{F}}^{-1}(z)),E_{n+1}\>=\<s(z)z+\hat{\nabla}^0 s(z),E_{n+1}\>=-\omega_0 s(z)+\<\hat{\nabla}^0 s(z),E_{n+1}\>$ for $z\in\W\cap\{x_{n+1}=-\omega_0\}$.

  Similar to \cite{Xia-2017-convex,Wei-Xiong2022}, by \eqref{equ:psi-T}, we have $\partial_t s=G(z)(z,\partial_t X)=f(z,t)=n+n\omega_0G(z)(E_{n+1}^F,z)-\frac{ns(z,t)}{\phi(\tau[s(z,t)])}$ for $z\in\W\cap\{x_{n+1}\geq-\omega_0\}$. Thus, when $t\in [0,T_0)$, $\Sigma_t$ are strictly convex, the  flow equation \eqref{equ:flow} can be written as
 \begin{equation}\label{equ:flow-inv-Gauss}
 	\renewcommand{\arraystretch}{1.5}
 	\left\{
 	\begin{array}{ll}
 		\partial_ts(z,t)=n+n\omega_0G(z)(E_{n+1}^F,z)-\frac{ns(z,t)}{\phi(\tau[s(z,t)])}, & \text{in}\ (\W\cap\{x_{n+1}\geq-\omega_0\})\times[0,T_0),
 		\\
 		\<\hat{\nabla}^0 s(z,t),E_{n+1}\>=\omega_0\cdot s(z,t), & \text{on}\ (\W\cap\{x_{n+1}=-\omega_0\})\times [0,T_0),
 		\\
 		s(\cdot,0)=s_0(\cdot) &\text{in}\ \W\cap\{x_{n+1}\geq-\omega_0\},
 	\end{array}	
 	\right.
 \end{equation}
 where
  $s_0$ is the anisotropic support function of initial hypersurface with respect to $\W$.

   Similarly, we have $\widetilde{\nu_{{F}}}(\Sigma)=\widetilde{\W}\cap\{x_{n+1}\geq0\}$ and  $\widetilde{\nu_{{F}}}(\partial\Sigma)=\W\cap\{x_{n+1}=0\}$ by $\<\widetilde{\nu_{{F}}}(X),E_{n+1}\>=0 $ on $\partial\Sigma$.
  Since  $\<X,E_{n+1}\>=0$ for $X\in\partial\Sigma\subset\mathbb{R}^n$, and   $X(\widetilde{\nu_{{F}}}^{-1}(\tilde{z}))=\tilde{s}(\tilde{z})\tilde{z}+\tilde{\nabla} \tilde{s}(\tilde{z})$ from \cite{Xia13}*{(3.3)}, we have $0=\<X(\widetilde{\nu_{{F}}}^{-1}(\tilde{z})),E_{n+1}\>=\<\tilde{s}(\tilde{z}) \tilde{z}+\tilde{\nabla} \tilde{s},E_{n+1}\>=\<\tilde{\nabla} s,E_{n+1}\>$ for $\tilde{z}\in\widetilde{\W}\cap\{x_{n+1}=0\}$.

 By  \eqref{equ:s-trandlation}, \eqref{equ:psi-T}, and the fact that $1+\omega_0G(z)(E_{n+1}^F,z)>0$  is independent of time $t$, we can rewrite flow equation \eqref{equ:flow-inv-Gauss} as
 \begin{equation}\label{equ:flow-inv-Trans-Gauss}
 	\renewcommand{\arraystretch}{1.5}
	\left\{
	\begin{array}{ll}
		\partial_t\tilde{s}(\tilde{z},t)=n-\frac{n\tilde{s}(\tilde{z},t)}{\Phi(\tilde{\tau}_{ij}[\tilde{s}(\tilde{z},t)])}, & \text{in}\ (\widetilde{\W}\cap\overline{\mathbb{R}_+^{n+1}})\times[0,T_0),
		\\
		\<\tilde{\nabla} \tilde{s}(\tilde{z},t),E_{n+1}\>=0, & \text{on}\ (\widetilde{\W}\cap\{x_{n+1}=0\})\times [0,T_0),
		\\
		\tilde{s}(\cdot,0)=\tilde{s}_0(\cdot) &\text{in}\  \widetilde{\W}\cap\overline{\mathbb{R}_+^{n+1}},
	\end{array}	
	\right.
\end{equation}
 where $\tilde{s}_0$ is the anisotropic support function of initial hypersurface with respect to $\widetilde{\W}$. 

In the above, we have used the translated Wulff shape $\widetilde{\W}$ to rewrite flow \eqref{equ:flow}  as scalar parabolic equation \eqref{equ:flow-inv-Trans-Gauss} for $t\in[0,T_0)$, which is similar to the equation (5.8) in \cite{Wei-Xiong2022} without considering the boundary.
This is the key point  of this section.
For equation \eqref{equ:flow-inv-Trans-Gauss}, we can follow the method in \cite{Wei-Xiong2022}*{Proposition 6.2} to estimate  the positive lower bound of anisotropic principal curvature internally.

 Next we need some evolution equation along the flow \eqref{equ:flow-inv-Trans-Gauss} in $\widetilde{\W}\cap\overline{\mathbb{R}_+^{n+1}}$.
 \begin{lemma}\label{lem:evolution-3}
 	For $\tilde{z}\in\widetilde{\W}\cap\overline{\mathbb{R}_+^{n+1}}$, we have the following  evolution equations  of the anisotropic support function $\tilde{s}$ and the tensor $\tilde{\tau}_{i j}[\tilde{s}]$  along the flow \eqref{equ:flow-inv-Trans-Gauss},
\begin{align}\label{equ:evo-s}
	 &\frac{1}{n}\frac{\partial}{\partial t} \tilde{s}-\tilde{s} \Phi^{-2} \dot{\Phi}^{k \ell}\left(\tilde{\nabla}_k \tilde{\nabla}_{\ell} \tilde{s}-\frac{1}{2} \tilde{Q}_{k \ell p} \tilde{\nabla}_p \tilde{s}\right)
 =\left(1-\tilde{s} \Phi^{-1}\right)^2
 +\tilde{s}^2 \Phi^{-2}\left(\sum_k \dot{\phi}^k-1\right) ,
\end{align}
and
\begin{align}
	& \quad \frac{1}{n}\frac{\partial}{\partial t} \tilde{\tau}_{i j}-\tilde{s} \Phi^{-2} \dot{\Phi}^{k \ell}\left(\tilde{\nabla}_k \tilde{\nabla}_{\ell} \tilde{\tau}_{i j}-\frac{1}{2} \tilde{Q}_{k \ell p} \tilde{\nabla}_p \tilde{\tau}_{i j}\right)\nonumber \\
	& =\tilde{s} \Phi^{-2} \ddot{\Phi}^{k \ell, p q} \tilde{\nabla}_i \tilde{\tau}_{k \ell} \tilde{\nabla}_j \tilde{\tau}_{p q}+\tilde{s} \Phi^{-2} \dot{\Phi} * \tilde{Q} * \tilde{Q} * \tilde{\tau}+\tilde{s} \Phi^{-2} \dot{\Phi} * \tilde{\nabla} \tilde{Q} * \tilde{\tau}\nonumber \\
	& \quad-\left(\Phi^{-1}+\tilde{s} \Phi^{-2} \dot{\Phi}^{k \ell} \tilde{g}_{k \ell}\right) \tilde{\tau}_{i j}+\tilde{g}_{i j}\left(1+\tilde{s} \Phi^{-1}\right)\nonumber \\
	& \quad+2 \Phi^{-2} \tilde{\nabla}_i \tilde{s} \tilde{\nabla}_j \Phi-2 \tilde{s} \Phi^{-3} \tilde{\nabla}_i \Phi \tilde{\nabla}_j \Phi ,\label{equ:evo-tau_ij}
\end{align}
where $* $  denotes contraction of tensors using the metric $\tilde{g}$ on $\widetilde{\W}$.
 \end{lemma}
 \begin{proof}
 The proof of these equations are same as \cite{Wei-Xiong2022}, more details can be found in \cite{Wei-Xiong2022}*{Lemma 5.4 and 5.6}, respectively.
\end{proof}

 \begin{lemma}
 	\label{lemma-Ds=0}
 	For  $\tilde{z}\in\partial (\widetilde{\W}\cap\overline{\mathbb{R}_+^{n+1}}) $, we take standard orthogonal frame $\{\tilde{e}_i\}_{i=1}^n\in T_{\tilde{z}}\widetilde{\W}$ with respect to $\tilde{g}$, which satisfies $\tilde{e}_{\alpha}\in T_{\tilde{z}}(\partial (\widetilde{\W}\cap\overline{\mathbb{R}_+^{n+1}})) $ for $\alpha=1,\cdots,n-1$, and $\tilde{g}(\tilde{e}_{\alpha},\tilde{e}_n)=0$ for any $\tilde{e}_{\alpha}$. Then
 	along flow \eqref{equ:flow-inv-Trans-Gauss}, we have
 	\begin{align}\label{equ:Dns}
 		\tilde{\nabla}_{\tilde{e}_n}\tilde{s}=0, \quad\text{on\ } \partial (\widetilde{\W}\cap\overline{\mathbb{R}_+^{n+1}}),
 	\end{align}
 	\begin{align}\label{equ:Dnpsi}
 		\tilde{\nabla}_{\tilde{e}_n}{\Phi}=0, \quad\text{on\ } \partial (\widetilde{\W}\cap\overline{\mathbb{R}_+^{n+1}}),
 	\end{align}
 	and
 		\begin{align}
 		\tilde{\nabla}_{\tilde{e}_n}\tilde{\tau}_{\alpha\alpha}=\sum_{\gamma=1}^{n-1}\tilde{Q}(\tilde{z})(\tilde{e}_{\alpha},\tilde{e}_{\gamma},\tilde{e}_{n})(\tilde{\tau}_{\alpha\gamma}-\tilde{\tau}_{nn}\delta_{\alpha\gamma}), \quad\text{on\ } \partial (\widetilde{\W}\cap\overline{\mathbb{R}_+^{n+1}}).\label{equ:taan}
 	\end{align}
 \end{lemma}
 \begin{proof}  
 	By the proof of \cite{Xia13}*{Proposition 3.2}, we have \begin{align}\label{equ:X=siei+sz}
 		X({\tilde{z}})=\tilde{\nabla}_{\tilde{e}_i}\tilde{s}\cdot \tilde{e}_i({\tilde{z}})+\tilde{s}(\tilde{z})\cdot \tilde{z},
 	\end{align} and \begin{align}\label{equ:e_ie_j(z)}
 	\tilde{e}_i(\tilde{z})=\tilde{e}_i,\ \tilde{e}_i(\tilde{e}_j)=\tilde{\nabla}_{\tilde{e}_i}\tilde{e}_j-\frac{1}{2}{\tilde{Q}}_{ijk}\tilde{e}_k-\tilde{g}_{ij}{\tilde{z}}.
 	\end{align}
From \eqref{equ:A.G} (see  Proposition \ref{prop:Q-Q} in Appendix  \ref{Appendix}) and  Remark \ref{rk3.5} (1), 
we can take $\tilde{e}_n(\tilde{z})=\lambda^{-1}\cdot(A_F(\nu)\mu)(\tilde{z})$ with $\lambda=|(A_F(\nu)\mu)(\tilde{z})|_{\tilde{g}}$. 
  Here we still denote $\nu(\tilde{z})=\nu(X(\tilde{z}))$ and $\mu(\tilde{z})=\mu(X(\tilde{z}))$, which are also the normal vector and co-normal of $\widetilde{\W}\subset\mathbb{R}^{n+1}$ and $(\widetilde{\W}\cap\{ x_{n+1}=0\})\subset\widetilde{\W}$, respectively.

 	On $\partial (\widetilde{\W}\cap\mathbb{R}_+^{n+1})$, for any $t$, we have
 	\begin{align}\label{equ:pfLem-Dns-1}
 		\tilde{\nabla}_{\tilde{e}_n}\tilde{s}=\tilde{e}_n(\tilde{G}(\tilde{z})(X(\tilde{z}),\tilde{z}))=G(\tilde{z})(X(\tilde{z}),\tilde{e}_n).
 	\end{align}
 	Similar to the proof of Remark \ref{rk3.5} (1), since $\lambda\cdot \tilde{e}_{n}(\tilde{z})=(\widetilde{A_F}(\nu)\mu)(\tilde{z})=D^2\widetilde{F}|_{\nu(\tilde{z})}\mu(\tilde{z})$ and $D^2_{ik}\widetilde{F}^0|_{\tilde{z}}\cdot D^2_{kj}\widetilde{F}|_{\nu}=-\frac{D_j\widetilde{F}|_{\nu}\cdot D_i\widetilde{F}^0|_{\tilde{z}}}{\widetilde{F}(\nu)}+\frac{\delta_{ij}}{\widetilde{F}(\nu)},$ we have
 \begin{align}
 &	\tilde{G}(\tilde{z})(X(\tilde{z}),\lambda\cdot\tilde{e}_n)=\tilde{G}(\tilde{z})(X(\tilde{z})-\tilde{s}\tilde{z},\lambda\cdot\tilde{e}_n)
 \nonumber\\
 	=&\<D^2\widetilde{F}^0|_{\tilde{z}}D^2\widetilde{F}|_{\nu(\tilde{z})}\mu(\tilde{z}),X(\tilde{z})-\tilde{s}\tilde{z}\>\nonumber
 	\\
 	=&\frac{\<X(\tilde{z})-\tilde{s}\tilde{z},\mu(\tilde{z})\>}{\widetilde{F}(\nu)}=\<\frac{X(\tilde{z})-\tilde{s}\tilde{z}}{\widetilde{F}(\nu)},\frac{E_{n+1}-\<E_{n+1},\nu\>\nu}{\<\mu(\tilde{z}),E_{n+1}\>}\>\nonumber\\
 	=&\<\frac{X(\tilde{z})-\tilde{s}\tilde{z}}{\widetilde{F}(\nu)},\frac{E_{n+1}}{\<\mu(\tilde{z}),E_{n+1}\>}\>=0,\label{equ:pfLem-Dns-2}
 \end{align}
 where we use $X(\tilde{z})\in\partial\mathbb{R}^{n+1}_+$ and $\<\tilde{z},E_{n+1}\>=0$ for $\tilde{z}\in \widetilde{\W}\cap(\partial\mathbb{R}^{n+1}_+)$.
 	Combining with \eqref{equ:pfLem-Dns-1} and \eqref{equ:pfLem-Dns-2}, we obtain \eqref{equ:Dns}.\\
 	
 Write $\bar{f}(\tilde{z},t)=\partial_t\tilde{s}(\tilde{z},t)=n-\frac{n\tilde{s}(\tilde{z},t)}{\Phi([\tilde{\tau}_{ij}](\tilde{z},t))}$. We have by \eqref{equ:Dns} that $\tilde{e}_n(\bar{f})=\tilde{e}_n(\partial_t\tilde{s})=\partial_t(\tilde{e}_n(\tilde{s}))=0$ on $\partial (\widetilde{\W}\cap\overline{\mathbb{R}_+^{n+1}})$, for any $t$. Then \eqref{equ:Dnpsi} holds since $\Phi=\frac{n\tilde{s}}{n-\bar{f}}$.\\

 Assume $\alpha,\beta,\gamma\in\{1,\cdots,n-1\}$ and $i,j,k,l\in\{1,\cdots,n\}$. On $\partial (\widetilde{\W}\cap\overline{\mathbb{R}_+^{n+1}})$, for any $t$, by $\tilde{e}_i(X)=(\tilde{e}_i\tilde{e}_j\tilde{s}-\frac{1}{2}\tilde{Q}_{ijk}\tilde{\nabla}_{\tilde{e}_k}\tilde{s}+\delta_{ij}\tilde{s})\tilde{e}_j+\tilde{\nabla}_{\tilde{e}_j}\tilde{s}\cdot\tilde{\nabla}_{\tilde{e}_i}\tilde{e}_j=(\tilde{\nabla}_{\tilde{e}_i}\tilde{\nabla}_{\tilde{e}_j}\tilde{s}-\frac{1}{2}\tilde{Q}_{ijk}\tilde{\nabla}_{\tilde{e}_k}\tilde{s}+{\delta}_{ij}\tilde{s})\tilde{e}_j=\tilde{\tau}_{ij}\tilde{e}_j$ and $\<X,E_{n+1}\>=0$, we have
 	$$0=\<\tilde{e}_{\alpha}X,E_{n+1}\>=\tilde{\tau}_{\alpha j}\<\tilde{e}_{j},E_{n+1}\>.$$
 	Since $\<\tilde{e}_{\alpha},E_{n+1}\>=0,\<\tilde{e}_{n},E_{n+1}\>\neq0$, we obtain   $\tilde{\tau}_{\alpha n}=0$.
 	It follows that
 	\begin{align}
 		0=\tilde{e}_{\beta}(\tilde{\tau}(\tilde{e}_{\alpha},\tilde{e}_{n}))=\tilde{\tau}_{\alpha n,\beta}+\tilde{\tau}(\tilde{\nabla}_{\tilde{e}_{\beta}}\tilde{e}_{\alpha},\tilde{e}_n)+\tilde{\tau}(\tilde{e}_{\alpha},\tilde{\nabla}_{\tilde{e}_{\beta}}\tilde{e}_n),\label{equ:pf-lemma-taan-1}
 	\end{align}
 	and
 	\begin{align}
 		\tilde{\tau}_{\alpha n,\beta}=&\tilde{\nabla}_{\tilde{e}_{\beta}}\tilde{\tau}_{\alpha n}=(\tilde{s}_{\alpha, n}-\frac{1}{2}\tilde{Q}_{\alpha n k}\tilde{s}_k+\delta_{\alpha n}\tilde{s})_{,\beta}
 		\nonumber
 		\\
 		=&\tilde{s}_{\alpha, n \beta}-\frac{1}{2}\tilde{Q}_{\alpha n k,\beta}\tilde{s}_k-\frac{1}{2}\tilde{Q}_{\alpha n k}\tilde{s}_{k,\beta}+\delta_{\alpha n}\tilde{s}_{\beta}
 		\nonumber
 		\\
 		=&\tilde{s}_{\alpha,\beta n}+\left(
 		\delta_{ln}\delta_{\alpha \beta}-\delta_{l \beta}\delta_{\alpha n}+\frac{1}{4}\tilde{Q}_{\alpha n k}\tilde{Q}_{l \beta k}-\frac{1}{4}\tilde{Q}_{lnk}\tilde{Q}_{\alpha\beta k}
 		\right)\tilde{s}_l\nonumber
 		\\
 		&-\frac{1}{2}\tilde{Q}_{\alpha n k,\beta}\tilde{s}_k-\frac{1}{2}\tilde{Q}_{\alpha n k}\tilde{s}_{k,\beta}+\delta_{\alpha n}\tilde{s}_{\beta}\nonumber
 		\\
 		=&\tilde{\tau}_{\alpha\beta,n}+\frac{1}{2}\tilde{Q}_{\alpha \beta l}\tilde{s}_{l,n}+\frac{1}{2}\tilde{Q}_{\alpha\beta l,n}\tilde{s}_l-\frac{1}{2}\tilde{Q}_{\alpha n k,\beta}\tilde{s}_k-\frac{1}{2}\tilde{Q}_{\alpha n k}\tilde{s}_{k,\beta}\nonumber
 		\\
 		&+\left(
 	\frac{1}{4}\tilde{Q}_{\alpha n k}\tilde{Q}_{l \beta k}-\frac{1}{4}\tilde{Q}_{lnk}\tilde{Q}_{\alpha\beta k}
 		\right)\tilde{s}_l\nonumber
 		\\
 		=&\tilde{\tau}_{\alpha \beta, n}+\frac{1}{2}\tilde{Q}_{\alpha \beta n}(\tilde{\tau}_{nn}-\tilde{s})-\frac{1}{2}\tilde{Q}_{\alpha\gamma n}(\tilde{\tau}_{\beta\gamma}-\delta_{\beta\gamma}\tilde{s})\nonumber
 		\\
 		=&\tilde{\tau}_{\alpha \beta, n}+\frac{1}{2}\tilde{Q}_{\alpha \beta n}\tilde{\tau}_{nn}-\frac{1}{2}\tilde{Q}_{\alpha\gamma n}\tilde{\tau}_{\beta\gamma},\label{equ:pf-lemma-taan-2}
 	\end{align}
 	where  we use the Ricci identity, \eqref{equ:Qjikl-syms}, and \eqref{equ:R=}.
 	
 	We calculate that
 	\begin{align}
 		\tilde{\tau}(\tilde{\nabla}_{\tilde{e}_{\beta}}\tilde{e}_{\alpha},\tilde{e}_n)=\tilde{g}(\tilde{\nabla}_{\tilde{e}_{\beta}}\tilde{e}_{\alpha},\tilde{e}_n)\tilde{\tau}_{nn}
 		=
 		\tilde{G}(\tilde{z})({\tilde{e}_{\beta}}(\tilde{e}_{\alpha}),\tilde{e}_n)\tilde{\tau}_{nn}+\frac{1}{2}\tilde{Q}_{\alpha\beta n}\tilde{\tau}_{nn},\label{equ:pf-lemma-taan-3}
 	\end{align}
 	and
 	\begin{align}
 		\tilde{\tau}(\tilde{e}_{\alpha},\tilde{\nabla}_{\tilde{e}_{\beta}}\tilde{e}_n)
 		=&\tilde{g}(\tilde{e}_{\gamma},\tilde{\nabla}_{\tilde{e}_{\beta}}\tilde{e}_n)\tilde{\tau}_{\gamma \alpha } 	
 		=-\tilde{g}(\tilde{\nabla}_{\tilde{e}_{\beta}}\tilde{e}_{\gamma},\tilde{e}_n)\tilde{\tau}_{\gamma \alpha }\nonumber
 		\\
 		=&-\tilde{G}(\tilde{z})(\tilde{e}_n,\tilde{e}_{\beta}(\tilde{e}_{\gamma}))\tilde{\tau}_{\gamma \alpha}-\frac{1}{2}\tilde{Q}_{\beta n \gamma}\tilde{\tau}_{\gamma \alpha},\label{equ:pf-lemma-taan-4}
 	\end{align}
 	where we use $\tilde{\tau}_{\alpha n}=0 		$ and \eqref{equ:e_ie_j(z)}.
 	
 	Combination of \eqref{equ:pf-lemma-taan-1}$\thicksim$\eqref{equ:pf-lemma-taan-4} gives
 	\begin{align}\label{equ:pf-lemma-taan-5}
 		\tilde{\tau}_{\alpha\alpha,n}=\sum_{\gamma=1}^{n-1}\left(\tilde{G}(\tilde{z})(\tilde{e}_{\alpha}\tilde{e}_{\gamma}(\tilde{z}),\tilde{e}_{n})+\tilde{Q}_{\alpha\gamma n}\right)(\tilde{\tau}_{\alpha\gamma}-\tilde{\tau}_{nn}\delta_{\alpha\gamma}).
 	\end{align}
 	Similar to \eqref{equ:pfLem-Dns-2} we can calculate that
 	\begin{align}
 		&\tilde{G}(\tilde{z})(\tilde{e}_{\alpha}\tilde{e}_{\gamma}(\tilde{z}),(A_F(\nu){\mu})(\tilde{z}))
 		=\frac{\<\tilde{e}_{\alpha}\tilde{e}_{\gamma}(\tilde{z})-\tilde{G}(\tilde{z})(\tilde{e}_{\alpha}\tilde{e}_{\gamma}(\tilde{z}),\tilde{z})\tilde{z},\mu(\tilde{z})\>}{\widetilde{F}(\nu)}\nonumber
 		\\
 			=&\left\langle \frac{\tilde{e}_{\alpha}\tilde{e}_{\gamma}(\tilde{z})-\tilde{G}(\tilde{z})(\tilde{e}_{\alpha}\tilde{e}_{\gamma}(\tilde{z}),\tilde{z})\tilde{z}
 			}{\widetilde{F}(\nu)},\frac{E_{n+1}}{\<\mu(\tilde{z}),E_{n+1}\>}\right\rangle
 			=\frac{\<\tilde{e}_{\alpha}(\tilde{e}_{\gamma}),E_{n+1}\>}{\<\mu(\tilde{z}),E_{n+1}\>\widetilde{F}(\nu)}\nonumber
 			\\
 			=&-\frac{\<\tilde{e}_{\alpha}(E_{n+1}),\tilde{e}_{\gamma}\>}{\<\mu(\tilde{z}),E_{n+1}\>\widetilde{F}(\nu)}=0.\label{equ:pf-lemma-taan-6}
 	\end{align}
 	Combining with \eqref{equ:pf-lemma-taan-5} and \eqref{equ:pf-lemma-taan-6}, we obtain \eqref{equ:taan}.
 \end{proof}
  We remark that, since the  direction $\tilde{e}_n$ has been fixed, we cannot directly calculate at normal coordinates in the above proof.

 Besides, we need the relationship between $\tilde{Q}$ and $Q$. More details can be seen in Appendix \ref{Appendix}. Now we list a  spacial case which will be used later.  
 \begin{lemma}\label{lem:Q-translate}
 	For any $\tilde{z}\in\partial(\widetilde{\W}\cap\overline{\mathbb{R}_+^{n+1}}) $, we take ${e}_n=A_F(\nu)\mu$, which is perpendicular to the tangent plane $T_{\tilde{z}}(\partial (\widetilde{\W}\cap\overline{\mathbb{R}_+^{n+1}}))$ with respect to $\tilde{g}$.
 	For any  $\tilde{e}_{\alpha}\in T_{\tilde{z}}(\partial (\widetilde{\W}\cap\overline{\mathbb{R}_+^{n+1}})) $,  we have
 	\begin{align}\label{Q-translate}
 		\tilde{Q}(\tilde{z})(\tilde{e}_{\alpha},\tilde{e}_{\alpha},e_n)=\frac{1}{1+\omega_0G(z)(z,E_{n+1}^F)}\left(
 		Q(z)(\tilde{e}_{\alpha},\tilde{e}_{\alpha},e_n)
 		-\frac{\omega_0G(z)(\tilde{e}_{\alpha},\tilde{e}_{\alpha})}{F(\nu)\<\mu,E_{n+1}\>}
 		\right),
 	\end{align}
 	where  $z=\mathcal{T}^{-1}(\tilde{z})$.
 \end{lemma}
 \begin{proof}
 	Take $\eta=\omega_0E_{n+1}^F,X=Y=\tilde{e}_{\alpha},Z=e_n=A_F(\nu)\mu$ in Proposition \ref{prop:Q-Q}, we have $G(z)(Z,X)=G(z)(Z,Y)=0$ by Remark $\ref{rk3.5}$.
 	We know that $1+G(z)(z,\eta)>0$, since $\omega_0\in (-F(E_{n+1}),F(-E_{n+1}))$ and \cite{Jia-Wang-Xia-Zhang2023}*{Proposition 3.2}. 	
 	Similar to \eqref{equ:pfLem-Dns-2} or \eqref{equ:pf-lemma-taan-6}, we can calculate that
 	\begin{align*}
 		&G(z)(E_{n+1}^F,A_F(\nu)\mu)=G(z)(E_{n+1}^F-G(z)(E_{n+1}^F,z)z,A_F(\nu)\mu)\nonumber
 		\\
 		=&\frac{\<E_{n+1}^F,\mu\>-G(z)(E_{n+1}^F,z)\<z,\mu\>}{F(\nu)}\nonumber
 		\\
 		=&\frac{\<E_{n+1}^F,E_{n+1}\>}{F(\nu)\<E_{n+1},\mu\>}-\frac{\<E_{n+1},\nu\>\<\nu,E_{n+1}^F\>}{F(\nu)\<E_{n+1},\mu\>}
 		-\frac{\<E_{n+1}^F,\nu\>\<z,\mu\>}{F(\nu)^2}\nonumber
 		\\
 		=&\frac{1}{F(\nu)\<E_{n+1},\mu\>}
 		+\frac{\omega_0}{F(\nu)^2\<E_{n+1},\mu\>}\<\nu,E_{n+1}^F\>
 		\\
 		=&\frac{1+\omega_0G(z)(z,E_{n+1}^F)}{F(\nu)\<E_{n+1},\mu\>},
 	\end{align*}
 	where we use
\begin{align*}
	\mu=\frac{E_{n+1}-\<E_{n+1},\nu\>\nu}{\<E_{n+1},\mu\>},
\end{align*}
and
\begin{align*}
	-\omega_0=\<z,E_{n+1}\>=F(\nu)\<\nu,E_{n+1}\>+\<\mu,E_{n+1}\>\<z,\mu\>.
\end{align*}
Then from \eqref{equ:A.Qaan}, we have
\begin{align*}
	\tilde{Q}(\tilde{z})(\tilde{e}_{\alpha},\tilde{e}_{\alpha},e_n)=&\frac{Q(z)(\tilde{e}_{\alpha},\tilde{e}_{\alpha},e_n)}{1+\omega_0G(z)(z,E_{n+1}^F)}-\frac{G(z)(\tilde{e}_{\alpha},\tilde{e}_{\alpha})G(z)(e_n,\omega_0E_{n+1}^F)+0+0}{\left(1+\omega_0G(z)(z,E_{n+1}^F)\right)^2}
	\\
	=&\frac{1}{1+\omega_0G(z)(z,E_{n+1}^F)}\left(Q(z)(\tilde{e}_{\alpha},\tilde{e}_{\alpha},e_n)-\frac{\omega_0G(z)(\tilde{e}_{\alpha},\tilde{e}_{\alpha})}{F(\nu)\<E_{n+1},\mu\>}\right).
\end{align*}
This completes the proof.
\end{proof}
At the end of this section, we shall prove that flow \eqref{equ:flow} preserves the convexity by contradiction if the initial hypersurface is strictly convex. 
\begin{proof}[Proof of Theorem \ref{thm:convex preservation}]
	
	We assume that the solution $\Sigma_t$ of flow \eqref{equ:flow} is only strictly convex in a maximum finite time interval $[0,T_0)$ for a constant $T_0>0$, i.e. $\Sigma_t$ is strictly convex for $t\in[0,T_0)$ and the minimum  principal curvature  of $\Sigma_{T_0}$ is $0$. Then for $t\in[0,T_0)$, $H_F$ is uniformly bounded from below by Proposition \ref{prop:H>c}, and flow \eqref{equ:flow} is equivalent to \eqref{equ:flow-inv-Trans-Gauss}.
	
   In order to estimate the lower bound of $\kappa^F$, it suffices to estimate the upper bound of  of $\tau$. Since $\tilde{\tau}\circ\mathcal{T}=\tau$, we first need to prove the uniform upper bound of $\tilde{\tau}$ along the flow \eqref{equ:flow-inv-Trans-Gauss} for $t\in[0,T_0)$.

   By the proof of Proposition \ref{prop:c1}, we know that $\bar{u}$ has the uniform positive lower bound,   then $\tilde{s}(\tilde{z},t)\geq\frac{1}{C}>0$, for  $\tilde{z}\in\widetilde{\W}\cap\overline{\mathbb{R}^{n+1}_+}$ and  positive constant $C$ only depends on initial hypersurface  $\Sigma_0$, Wulff shape $\W$ (with respect to $F$), and constant $\omega_0$.  By Proposition \ref{prop:C0} and Cauchy-Schwartz inequality, we have $\tilde{s}(\tilde{z})=\tilde{g}(X,\tilde{z})=\frac{\<X,\tilde{z}\>}{\widetilde{F}(\tilde{z})}\leq\frac{|X|\cdot|\tilde{z}|}{\widetilde{F}(\tilde{z})}\leq C.$ From Proposition \ref{prop:H_F<c},  Proposition \ref{prop:H>c} and \eqref{equ:psi-T}, we have $0<\frac{1}{C}\leq\Phi\leq C$. It is easy to check that $\sum_k\dot{\phi}^k=\sum_{k} n^{-1}\phi^{2}{\tilde{\tau}_k}^{-2}\leq n^{-1}\phi^{2}(\sum_k{\tilde{\tau}_k}^{-1})^{2}=n$, since $\phi=n(\sum_k\tilde{\tau}_k^{-1})^{-1}$. Combining with \eqref{equ:pf-5.6}, we have $1\leq\sum_k\dot{\phi}^k\leq n$. Since the $C^0$ estimate in Proposition \ref{prop:C0} and    $\tilde{G}(\tilde{z})(\cdot,\cdot)=\frac{{G}({z})(\cdot,\cdot)}{1+\omega_0G(z)(z,E_{n+1}^F)}$ from Appendix  \ref{Appendix} (Proposition \ref{prop:Q-Q}),  we have $G(\tilde{z})(X(\tilde{z}),X(\tilde{z}))\leq C$, 
   that means $\tilde{s}^2+|\tilde{\nabla}\tilde{s}|^2_{\tilde{g}}\leq C$ by \eqref{equ:X=siei+sz}, then
   we have $|\tilde{\nabla}\tilde{s}|_{\tilde{g}}\leq C$.

  We  choose the standard orthogonal frame $\{\tilde{e}_i\}_{i=1}^n$  such that $\left(\tilde{\tau}_{i j}\right)$ is diagonal at the point we are considering.  Suppose that $\tilde{e}_1$ is the direction where the largest eigenvalue of $\left(\tilde{\tau}_{i j}\right)$ occurs. 
   Then the  evolution equation  \eqref{equ:evo-tau_ij} gives
   \begin{align}
  	& \frac{1}{n}\frac{\partial}{\partial t} \tilde{\tau}_{11}-\tilde{s} \Phi^{-2}\left(\dot{\Phi}^{k \ell} \tilde{\nabla}_k \tilde{\nabla}_{\ell} \tilde{\tau}_{11}-\frac{1}{2} \dot{\Phi}^{k \ell} \tilde{Q}_{k \ell p} \tilde{\nabla}_p \tilde{\tau}_{11}\right) \nonumber
  	\\
  	\leq & \tilde{s} \Phi^{-2} \ddot{\Phi}^{k \ell, p q} \tilde{\nabla}_1 \tilde{\tau}_{k \ell} \tilde{\nabla}_1 \tilde{\tau}_{p q}+C_0 \tilde{s} \Phi^{-2} \dot{\Phi}^{k \ell} \tilde{g}_{k \ell} \tilde{\tau}_{11}\nonumber \\
  	& -\left(\Phi^{-1}+\tilde{s} \Phi^{-2} \dot{\Phi}^{k \ell} \tilde{g}_{k \ell}\right) \tilde{\tau}_{11}+\left(1+\tilde{s} \Phi^{-1}\right)\nonumber
  	 \\
  	& -2 \tilde{s} \Phi^{-3}\left(\tilde{\nabla}_1 \Phi-\frac{1}{2} \Phi \tilde{\nabla}_1 \log \tilde{s}\right)^2+\frac{1}{2 \tilde{s} \Phi}\left|\tilde{\nabla}_1 \tilde{s}\right|_{\tilde{g}}^2\nonumber
  	\\
  	&\leq   C \ddot{\Phi}^{k \ell, p q} \tilde{\nabla}_1 \tilde{\tau}_{k \ell} \tilde{\nabla}_1 \tilde{\tau}_{p q}+C \tilde{\tau}_{11}+C,\label{equ:pf-convex5.37}
  \end{align}
    where $C,C_0$ are positive constants,  $C$ only depends on initial hypersurface $\Sigma_0$, Wulff shape $\W$ (with respect to $F$), and constant $\omega_0$. $C_0$ depends on $\tilde{Q}$ and $\tilde{\nabla}\tilde{Q}$,   $\tilde{Q}$ and $\tilde{\nabla}\tilde{Q}$ only depend on Wulff shape $\W$ (with respect to $F$) and constant $\omega_0$.
   In  order to deal with the `bad' term $C\tilde{\tau}_{11}$, we consider the auxiliary function
  \begin{align}\label{equ:w-x-6.8}
  	 {W}(z, t)=\log \zeta(z, t)-a\cdot  \tilde{s}(z, t), \quad \text{on\ } \left(\widetilde{\W}\cap\overline{\mathbb{R}^{n+1}_+} \right)\times[0, T_0),
  \end{align}
   where $a>0$ is positive constant to be determined later, and $
   \zeta=\sup \left\{\tilde{\tau}_{i j} \xi^i \xi^j:|\xi|_{\tilde{g}}=1\right\}
   $.

  We consider a point $\left(z_1, t_1\right)$ where a new maximum of the function ${W}$ is achieved, i.e., ${W}\left(z_1, t_1\right)=\max _{\left(\widetilde{\W}\cap\overline{\mathbb{R}^{n+1}_+} \right)\times\left[0, t_1\right]} {W}$.

 (1) If  $z_1$ is on the boundary of $ \widetilde{\W}\cap\overline{\mathbb{R}^{n+1}_+} $, by rotating  the  orthogonal frame $\{\tilde{e}_i\}_{i=1}^n$, we assume the $n$-th direction $\tilde{e}_n$ is the same as in Lemma \ref{lemma-Ds=0} (pointing outward of $ (\widetilde{\W}\cap\overline{\mathbb{R}_+^{n+1}})\subset\widetilde{\W}$), 
 and $\tilde{\tau}_{ij}$ is still diagonal at $(z_1,t_1)$ since $\tilde{\tau}_{\alpha n}=0$ (for $\alpha\neq n$). Then we have either $\xi=\tilde{e}_{n}$ or $\xi=\tilde{e}_{\alpha}$, for some $\alpha\in\{1,\cdots,n-1\}$.
\begin{itemize}
	\item  Case 1: If $\xi=\tilde{e}_{\alpha}$, for some $\alpha\in\{1,\cdots,n-1\}.$ At $\left(z_1, t_1\right)$ we have $\zeta=\tilde{\tau}_{\alpha\alpha}$, and $\tilde{\tau}_{\alpha\alpha}\geq\tilde{\tau}_{nn}$. Then by Lemma \ref{lemma-Ds=0}, we have
	\begin{align*}
		\tilde{\nabla}_{\tilde{e}_n}W|_{(z_1,t_1)}=\tilde{\tau}_{\alpha\alpha}^{-1}\tilde{\nabla}_{\tilde{e}_n}\tilde{\tau}_{\alpha\alpha}=\tilde{\tau}_{\alpha\alpha}^{-1}\tilde{Q}_{\alpha\alpha n}(\tilde{\tau}_{\alpha\alpha}-\tilde{\tau}_{nn})\leq 0,
	\end{align*}
	  where we use Condition \ref{condition}, and Lemma \ref{lem:Q-translate}.
	  \item  Case 2: If $\xi=\tilde{e}_{n}$. At $\left(z_1, t_1\right)$ we have $\zeta=\tilde{\tau}_{nn}$,   and $\tilde{\tau}_{\alpha\alpha}\leq\tilde{\tau}_{nn}$, for any $\alpha\in\{1,\cdots,n-1\}$.
	  By \eqref{equ:Dnpsi}, we have
	  $$0=\tilde{\nabla}_{\tilde{e}_n}\Phi=\dot{\Phi}^{nn}\tilde{\tau}_{nn,n}+\sum_{\alpha=1}^{n-1}\dot{\Phi}^{\alpha\alpha}
	  \tilde{\tau}_{\alpha\alpha,n}.$$
	  Then by Lemma \ref{lemma-Ds=0},
	  \begin{align*}
	  	\tilde{\nabla}_{\tilde{e}_n}W|_{(z_1,t_1)}
	  	=\tilde{\tau}_{nn}^{-1}\tilde{\nabla}_{\tilde{e}_n}\tilde{\tau}_{nn}=-\tilde{\tau}_{nn}^{-1}\sum_{\alpha=1}^{n-1}\frac{\dot{\Phi}^{\alpha\alpha}}{\dot{\Phi}^{nn}}\tilde{Q}_{\alpha\alpha n}(\tilde{\tau}_{\alpha\alpha}-\tilde{\tau}_{nn})\leq 0,
	  \end{align*}
	  where we use Condition \ref{condition},   Lemma \ref{lem:Q-translate}, and $[\dot{\Phi}^{ij}]>0$.
\end{itemize}

    This contradicts the Hopf boundary Lemma.
    So ${W}$ attain its maximum value at some interior point.

  (2)  For $z_1\in\operatorname{int} (\widetilde{\W}\cap\overline{\mathbb{R}^{n+1}_+} )$,
   we have  ${W}\left(z_1, t_1\right)=\max _{\widetilde{\W} \times\left[0, t_1\right]} {W}$ and ${W}(z, t)\leq{W}\left(z_1, t_1\right)$ for $t<t_1$. By rotation of local orthonormal frame, we assume that $\xi=\tilde{e}_1$ and $\left(\tilde{\tau}_{i j}\right)=\operatorname{diag}\left(\tilde{\tau}_1, \cdots, \tilde{\tau}_n\right)$ is diagonal at $\left(z_1, t_1\right)$. Firstly, at $\left(z_1, t_1\right)$ we know $\zeta=\tilde{\tau}_{11}$,  $\tilde{\tau}_{11}=\tilde{\tau}_1$ is the largest eigenvalue of $\left(\tilde{\tau}_{i j}\right)$ at the point $\left(z_1, t_1\right)$, and
   \begin{align}
   0\leq	\frac{\partial}{\partial t} {W} & =\frac{1}{\tilde{\tau}_{11}} \frac{\partial}{\partial t} \tilde{\tau}_{11}-a \frac{\partial}{\partial t} \tilde{s}, \label{equ:ptW}
   \\
   0=	\tilde{\nabla}_{\ell} {W} & =\frac{1}{\tilde{\tau}_{11}} \tilde{\nabla}_{\ell} \tilde{\tau}_{11}-a \tilde{\nabla}_{\ell} \tilde{s},
   	\label{equ:DW}
   	\\
   0\geq 	\tilde{\nabla}_k \tilde{\nabla}_{\ell} {W} & =\frac{1}{\tilde{\tau}_{11}} \tilde{\nabla}_k \tilde{\nabla}_{\ell} \tilde{\tau}_{11}-a \tilde{\nabla}_k \tilde{\nabla}_{\ell} \tilde{s}-\frac{1}{\left(\tilde{\tau}_{11}\right)^2} \tilde{\nabla}_k \tilde{\tau}_{11} \tilde{\nabla}_{\ell} \tilde{\tau}_{11} .\label{equ:DDW}
   \end{align}
  From \eqref{equ:evo-s},  \eqref{equ:pf-convex5.37}, and \eqref{equ:ptW}$ \thicksim$\eqref{equ:DDW}, we have
   \begin{align}
    0\leq	& \frac{1}{n}\frac{\partial}{\partial t} {W}-\tilde{s} \Phi^{-2} \dot{\Phi}^{k \ell}\left(\tilde{\nabla}_k \tilde{\nabla}_{\ell} {W}-\frac{1}{2} \tilde{Q}_{k \ell p} \tilde{\nabla}_p {W}\right) \nonumber\\
   	= & \frac{1}{\tilde{\tau}_{11}}\left(\frac{1}{n}\frac{\partial}{\partial t} \tilde{\tau}_{11}-\tilde{s} \Phi^{-2} \dot{\Phi}^{k \ell}\left(\tilde{\nabla}_k \tilde{\nabla}_{\ell} \tilde{\tau}_{11}-\frac{1}{2} \tilde{Q}_{k \ell p} \tilde{\nabla}_p \tilde{\tau}_{11}\right)\right)\nonumber \\
   	& -a\left(\frac{1}{n}\frac{\partial}{\partial t} \tilde{s}-\tilde{s} \Phi^{-2} \dot{\Phi}^{k \ell}\left(\tilde{\nabla}_k \tilde{\nabla}_{\ell} \tilde{s}-\frac{1}{2} \tilde{Q}_{k \ell p} \tilde{\nabla}_p \tilde{s}\right)\right) \nonumber
   \\
   	& +\frac{1}{\left(\tilde{\tau}_{11}\right)^2} \tilde{s} \Phi^{-2} \dot{\Phi}^{k \ell} \tilde{\nabla}_k \tilde{\tau}_{11} \tilde{\nabla}_{\ell} \tilde{\tau}_{11} \nonumber
   	\\
   	\leq & \frac{1}{\tilde{\tau}_{11}}\left(C\ddot{\Phi}^{k \ell, p q} \tilde{\nabla}_1 \tilde{\tau}_{k \ell} \tilde{\nabla}_1 \tilde{\tau}_{p q}+ C \tilde{\tau}_{11}+C\right) \nonumber
   	\\
   	& -a C\left(\sum_k \dot{\phi}^k-1\right)+\frac{C}{\left(\tilde{\tau}_{11}\right)^2}  \dot{\Phi}^{k \ell} \tilde{\nabla}_k \tilde{\tau}_{11} \tilde{\nabla}_{\ell} \tilde{\tau}_{11} \nonumber
   	\\
   	\leq & C\left(\frac{1}{\tilde{\tau}_{11}} \ddot{\Phi}^{k \ell, p q} \tilde{\nabla}_1 \tilde{\tau}_{k \ell} \tilde{\nabla}_1 \tilde{\tau}_{p q}+\frac{1}{\left(\tilde{\tau}_{11}\right)^2} \dot{\Phi}^{k \ell} \tilde{\nabla}_k \tilde{\tau}_{11} \tilde{\nabla}_{\ell} \tilde{\tau}_{11}\right) \nonumber\\
   	& -a C\left(\sum_k \dot{\phi}^k-1\right) +C +\frac{C}{\tilde{\tau}_{11}},\label{equ:w-x-6.9}
   \end{align}
   where the  constants $C$  depend on $F$, $\omega_0$ and the initial hypersurface $\Sigma_0$. 
   Similar to \cite{Wei-Xiong2022}, 
   by Lemma \ref{lem:w-x-lem-5.3} and \ref{lem:w-x-lem-5.2}, we have
   \begin{align}
   	\ddot{\Phi}^{k \ell, p q} \tilde{\nabla}_1 \tilde{\tau}_{k \ell} \tilde{\nabla}_1 \tilde{\tau}_{p q} & =\ddot{\phi}^{k \ell} \tilde{\nabla}_1 \tilde{\tau}_{k k} \tilde{\nabla}_1 \tilde{\tau}_{\ell \ell}+2 \sum_{k>\ell} \frac{\dot{\phi}^k-\dot{\phi}^{\ell}}{\tilde{\tau}_k-\tilde{\tau}_{\ell}}\left(\tilde{\nabla}_1 \tilde{\tau}_{k \ell}\right)^2\nonumber \\
   	& \leq 2 \sum_{k>1} \frac{\dot{\phi}^k-\dot{\phi}^1}{\tilde{\tau}_k-\tilde{\tau}_1}\left(\tilde{\nabla}_1 \tilde{\tau}_{k 1}\right)^2 \nonumber\\
   	& \leq-2 \sum_{k>1} \frac{1}{\tilde{\tau}_1}\left(\dot{\phi}^k-\dot{\phi}^1\right)\left(\tilde{\nabla}_1 \tilde{\tau}_{k 1}\right)^2 .\label{equ:w-x-6.10}
   \end{align}
   From \cite{Xia-2017-convex}*{(4.9)} or  \cite{Wei-Xiong2022}*{(5.2)}, we have
  \begin{align}
  	 \left(\tilde{\nabla}_1 \tilde{\tau}_{k 1}\right)^2
   \geq \frac{1}{2}\left(\tilde{\nabla}_k \tilde{\tau}_{11}\right)^2-C\left(\tilde{\tau}_{11}\right)^2 .\label{equ:w-x-6.11}
  \end{align}
   Combining with \eqref{equ:DW}, \eqref{equ:w-x-6.10},  and \eqref{equ:w-x-6.11}, we have   
   $$
   \begin{aligned}
   	& \frac{1}{\tilde{\tau}_{11}} \ddot{\Phi}^{k \ell, p q} \tilde{\nabla}_1 \tilde{\tau}_{k \ell} \tilde{\nabla}_1 \tilde{\tau}_{p q}+\frac{1}{\left(\tilde{\tau}_{11}\right)^2} \sum_k \dot{\phi}^k\left(\tilde{\nabla}_k \tilde{\tau}_{11}\right)^2 \\
   	\leq & -2 \sum_{k>1} \frac{1}{\tilde{\tau}_{11}^2}\left(\dot{\phi}^k-\dot{\phi}^1\right)\left(\frac{1}{2}\left(\tilde{\nabla}_k \tilde{\tau}_{11}\right)^2-C\left(\tilde{\tau}_{11}\right)^2\right)+\frac{1}{\left(\tilde{\tau}_{11}\right)^2} \sum_k \dot{\phi}^k\left(\tilde{\nabla}_k \tilde{\tau}_{11}\right)^2 \\
   	= & \dot{\phi}^1 \tilde{\tau}_{11}^{-2} \sum_k\left(\tilde{\nabla}_k \tilde{\tau}_{11}\right)^2+2 C \sum_{k>1}\left(\dot{\phi}^k-\dot{\phi}^1\right) \\
   	= & \dot{\phi}^1 \sum_k\left(a \tilde{\nabla}_k \tilde{s}\right)^2+2 C \sum_{k}\dot{\phi}^k
   	\\
   	\leq  & a^2C\dot{\phi }^1+C,
   \end{aligned}
   $$
   at $\left(z_1, t_1\right)$. 
   Putting it into \eqref{equ:w-x-6.9} yields
   \begin{align}\label{equ:w-x-6.13}
   	0 \leq & C \dot{\phi}^1a^2+C
   	 -aC\left(\sum_k \dot{\phi}^k-1\right) +\frac{C}{\tilde{\tau}_{11}},
   \end{align}
   at $\left(z_1, t_1\right)$.
   Now, we apply an observation in \cite{Bo-Shi-Sui,Xia2017,Wei-Xiong2022} which relies on the property that $\phi$ is increasing and concave in $\Gamma_{+}=\{(\lambda_1,\cdots,\lambda_n)\in\mathbb{R}^n|\lambda_i>0,\forall i=1,\cdots,n\}$, and vanishes on $\partial \Gamma_{+}$. Precisely, for $\phi=\left(n\sigma_n / \sigma_{n-1}\right)$ we have the following proposition.

    \begin{proposition}\label{Prop:w-x-prop6.3}
    	(Lemma 2.2 in \cite{Bo-Shi-Sui}). Let $K \subset \Gamma_{+}$ be a compact set and $\gamma>0$. Then there exists a constant $\varepsilon>0$ depending only on $K$ and $\gamma$ such that for ${\tau} \in \Gamma_{+}$  and $\varsigma \in K$ satisfying $\left|\nu_{\tau}-\nu_\varsigma\right| \geq \gamma$, we have
   $$
   \sum_k \dot{\phi}^k(\tau) \varsigma_k-\phi(\varsigma) \geq \varepsilon\left(\sum_k \dot{\phi}^k(\tau)+1\right),
   $$
   where $\nu_\tau$ denotes the unit normal vector of the level set $\left\{x \in \Gamma_{+}: \phi(x)=\phi(\tau)\right\}$ at the point $\tau$, i.e. $\nu_\tau=\bar{D} \phi(\tau) /|\bar{D} \phi(\tau)|$, where $\bar{D}$ denotes the standard gradient on $\mathbb{R}^n$.
    \end{proposition}

   Let $\varsigma=(r, \cdots, r) \in \Gamma_{+}$, $\tau=\tilde{\tau} \in \Gamma_{+}$ where $r$ is the lower bound of the anisotropic capillary support function $\tilde{s}(z, t)$, and let $K=\{\varsigma\}$ in Proposition \ref{Prop:w-x-prop6.3}. There exists a small constant $\gamma>0$ such that $\nu_\varsigma-2 \gamma(1, \cdots, 1) \in \Gamma_{+}$. We have two cases.

    Case (i): If the anisotropic principal curvature radii  $\tilde{\tau}$ at the point $\left(z_1, t_1\right)$ satisfy $\left|\nu_{\tilde{\tau}}-\nu_\varsigma\right| \geq \gamma$, Proposition \ref{Prop:w-x-prop6.3} implies that
   $$
   \sum_k \dot{\phi}^k(\tilde{\tau})-1 \geq \frac{\varepsilon}{r}\left(\sum_k \dot{\phi}^k(\tilde{\tau})+1\right) \geq C\varepsilon.
   $$
   Putting it into \eqref{equ:w-x-6.13} yields
   \begin{align}\label{equ:pf-convex-1}
   	 a \varepsilon C
   	\leq  C \dot{\phi}^1 a^2+	 C +\frac{C}{\tilde{\tau}_{11}}.
   \end{align}
   We assume $\tilde{\tau}_{11}\geq1$ (otherwise, $\tilde{\tau}_{11}\leq1$ which can complete the proof directly), and choose constant $a>0$ large such that
   $$
   \begin{aligned}
   	\frac{1}{2}a \varepsilon C
   	\geq  C +\frac{C}{\tilde{\tau}_{11}}.
   \end{aligned}
   $$
   Putting it into \eqref{equ:pf-convex-1} yields
   $$
   \dot{\phi}^1\geq \frac{C\varepsilon}{2a}.
   $$
   Note that the constants $a$  and $\varepsilon$ independent of $t$, then it
    gives a lower bound on $\dot{\phi}^1$.

   Case (ii): If the anisotropic principal radii  $\tilde{\tau}$ at the point $\left(z_1, t_1\right)$ satisfy $\left|\nu_{\tilde{\tau}}-\nu_\varsigma\right| \leq \gamma$, it follows from $\nu_\varsigma-2 \gamma(1, \cdots, 1) \in \Gamma_{+}$ that  $\nu_{\tilde{\tau}}-\gamma(1, \cdots, 1) \in \Gamma_{+}$ and so   $\dot{\phi}^i \geq C^{\prime}\left(\sum_k\left(\dot{\phi}^k\right)^2\right)^{1 / 2} \geq C'' \sum_k \dot{\phi}^k\geq C$ at ${\tilde{\tau}}$ for some constants $C^{\prime}, C'', C$ and all $i=1, \cdots, n$.

   So we have  $\dot\phi^1\geq C$ for both Case (i) and Case (ii), then
   $$
   \Phi=\sum_k \dot{\phi}^k {\tilde{\tau}}_k \geq \dot{\phi}^1 {\tilde{\tau}}_{11} \geq C {\tilde{\tau}}_{11} .
   $$
   This implies that ${\tilde{\tau}}_{11}$ is bounded from above, since $\Phi\leq C$.

   Therefore, we can choose a large constant $a$ in the definition \eqref{equ:w-x-6.8} of ${W}(z, t)$ such that in both the above two cases, the largest anisotropic principal curvature radius  $\tilde{\tau}_1$ has a uniform upper bound at $\left(z_1, t_1\right)$. By the definition of ${W}(z, t)$ together with  $\frac{1}{C}\leq \tilde{s}\leq C$, we conclude that $\tilde{\tau}_1$ is bounded from above uniformly along the flow
 \eqref{equ:flow-inv-Trans-Gauss} in $t\in[0,T_0)$.
 Then we know $\kappa^F$ has a uniform positive lower bound along the flow \eqref{equ:flow} in $t\in[0,T_0)$, since flow \eqref{equ:flow} is equivalent to \eqref{equ:flow-inv-Trans-Gauss} together with the condition $\tilde{\tau}_{ij}(\tilde{s})>0$ (since \cite{Xia13}*{Proposition 3.2}).  Thus $\kappa^F(T_0)\geq C>0$, this contradicts our assumption of the maximization of $T_0$. Thus we derive the lower bound of $\k_F$ in all time and this proof is completed.

\end{proof}
\section{General quermassintegrals and Alexandrov-Fenchel inequalities}\label{sec 6}
In this section, we define some quermassintegrals for anisotropic capillary hypersurfaces in the half-space, calculate their first variational formulas, and then show the monotonicity of this quermassintegrals along flow under some convexity conditions. Therefore,  Theorem \ref{thm:AF-neq} will be proved.

For the given  Wulff shape $\W\subset\mathbb{R}^{n+1}$ determined by  $F$, and for the constant $\omega_0\in(-F(E_{n+1}),F(-E_{n+1}))$, we denote
 $$\overline{\W}=\widetilde{\W}\cap\{x_{n+1}=0\},$$
  as an $(n-1)$-dimensional Wulff shape in $\mathbb{R}^n$  with support function $\bar{F}$.

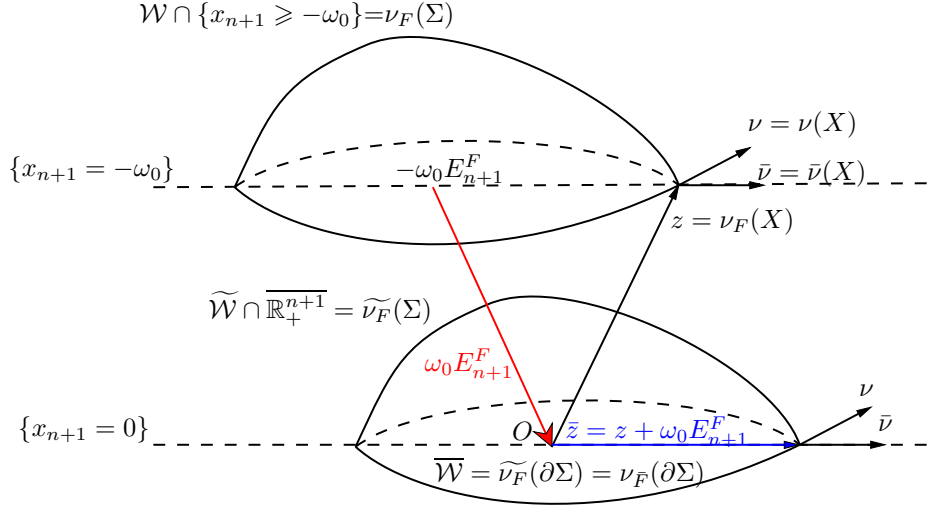
\begin{figure}[h]
	\begin{tikzpicture}[x=0.75pt,y=0.75pt,yscale=-1,xscale=1]
		
		\draw  [dash pattern={on 4.5pt off 4.5pt}]  (129.6,170) -- (520.6,168) ;
		\draw  [dash pattern={on 4.5pt off 4.5pt}]  (129.6,300) -- (520.6,300) ;
		\draw    (170.6,170) .. controls (189.6,129) and (194.6,116) .. (237.6,98) .. controls (280.6,80) and (380.6,129) .. (394.6,169) ;
		\draw    (170.6,170) .. controls (185.03,183.03) and (212.88,194.33) .. (248.62,197.75) .. controls (289.68,201.67) and (341.14,195.2) .. (394.6,169) ;
		\draw  [dash pattern={on 4.5pt off 4.5pt}]  (170.6,170) .. controls (210.6,140) and (356.6,138) .. (394.6,169) ;
		\draw    (394.6,169) -- (435.6,169) ;
		\draw [shift={(437.6,169)}, rotate = 180] [fill={rgb, 255:red, 0; green, 0; blue, 0 }  ][line width=0.08]  [draw opacity=0] (7.2,-1.8) -- (0,0) -- (7.2,1.8) -- cycle    ;
		\draw    (394.6,169) -- (428.83,150.93) ;
		\draw [shift={(430.6,150)}, rotate = 152.18] [fill={rgb, 255:red, 0; green, 0; blue, 0 }  ][line width=0.08]  [draw opacity=0] (7.2,-1.8) -- (0,0) -- (7.2,1.8) -- cycle    ;
		\draw    (231.6,301) .. controls (250.6,260) and (255.6,247) .. (298.6,229) .. controls (341.6,211) and (441.6,260) .. (455.6,300) ;
		\draw    (231.6,301) .. controls (246.03,314.03) and (273.88,325.33) .. (309.62,328.75) .. controls (350.68,332.67) and (402.14,326.2) .. (455.6,300) ;
		\draw  [dash pattern={on 4.5pt off 4.5pt}]  (231.6,301) .. controls (271.6,271) and (417.6,269) .. (455.6,300) ;
		\draw    (455.6,300) -- (496.6,300) ;
		\draw [shift={(498.6,300)}, rotate = 180] [fill={rgb, 255:red, 0; green, 0; blue, 0 }  ][line width=0.08]  [draw opacity=0] (7.2,-1.8) -- (0,0) -- (7.2,1.8) -- cycle    ;
		\draw    (455.6,300) -- (489.83,281.93) ;
		\draw [shift={(491.6,281)}, rotate = 152.18] [fill={rgb, 255:red, 0; green, 0; blue, 0 }  ][line width=0.08]  [draw opacity=0] (7.2,-1.8) -- (0,0) -- (7.2,1.8) -- cycle    ;
		\draw [color={rgb, 255:red, 255; green, 1; blue, 1 }  ,draw opacity=1 ]   (270.6,170) -- (329.35,298.27) ;
		\draw [shift={(330.6,301)}, rotate = 245.39] [fill={rgb, 255:red, 255; green, 1; blue, 1 }  ,fill opacity=1 ][line width=0.08]  [draw opacity=0] (10.72,-5.15) -- (0,0) -- (10.72,5.15) -- (7.12,0) -- cycle    ;
		\draw    (330.6,301) -- (393.73,170.8) ;
		\draw [shift={(394.6,169)}, rotate = 115.87] [fill={rgb, 255:red, 0; green, 0; blue, 0 }  ][line width=0.08]  [draw opacity=0] (7.2,-1.8) -- (0,0) -- (7.2,1.8) -- cycle    ;
		\draw [color={rgb, 255:red, 0; green, 0; blue, 255 }  ,draw opacity=1 ]    (330.6,300) -- (453.6,300.02) ;
		\draw [shift={(455,300)}, rotate = 179.54] [fill={rgb, 255:red, 0; green, 0; blue, 255 }  ][line width=0.08]  [draw opacity=0] (7.2,-1.8) -- (0,0) -- (7.2,1.8) -- cycle    ;
		
		\draw (428,130) node [anchor=north west][inner sep=0.75pt]   [align=left] {$\nu=\nu(X)$};
		\draw (432.6,155) node [anchor=north west][inner sep=0.75pt]   [align=left] {$\bar{\nu}=\bar{\nu}(X)$};
		\draw (250,152) node [anchor=north west][inner sep=0.75pt]   [align=left] {$-\omega_0E_{n+1}^F$};
		\draw (484,268) node [anchor=north west][inner sep=0.75pt]   [align=left] {$\nu$};
		\draw (493.6,284) node [anchor=north west][inner sep=0.75pt]   [align=left] {$\bar{\nu}$};
		\draw (309,287) node [anchor=north west][inner sep=0.75pt]   [align=left] {$O$};
		\draw (265,250) node [anchor=north west][inner sep=0.75pt]   [align=left] {\color{red}{$\omega_0E_{n+1}^F$}};
		\draw (270,305) node [anchor=north west][inner sep=0.75pt]   [align=left] {$\overline{\W}=\widetilde{\nu_{{F}}}(\partial\Sigma)=\nu_{\bar{F}}(\partial\Sigma)$};
		\draw (389,179) node [anchor=north west][inner sep=0.75pt]   [align=left] {$z=\nu_{F}(X)$};
		\draw (336,284) node [anchor=north west][inner sep=0.75pt]   [align=left] {\color{blue}{$\bar{z}=z+\omega_0E_{n+1}^F$}};
		\draw (121,75) node [anchor=north west][inner sep=0.75pt]   [align=left] {$\W\cap\{x_{n+1}\geq-\omega_0\}$=$\nu_{F}(\Sigma)$};text8
		\draw (156,220) node [anchor=north west][inner sep=0.75pt]   [align=left] {$\widetilde{\W}\cap\overline{\mathbb{R}^{n+1}_+}=\widetilde{\nu_{F}}(\Sigma)$};
		\draw (60,284) node [anchor=north west][inner sep=0.75pt]   [align=left] {$\{x_{n+1}=0\}$};
		\draw (55,153) node [anchor=north west][inner sep=0.75pt]   [align=left] {$\{x_{n+1}=-\omega_0\}$};
		\end{tikzpicture}
	\caption{The relationship between $z$ and $\bar{z}$}\label{fig:1}
\end{figure}
We denote $\nu_{\bar{F}}$, $S_{\bar{F}}$ and $H_k^{\bar{F}}(0\leq k\leq n-1,\,H_0^{\bar{F}}=1)$ as the anisotropic unit outward normal,  anisotropic Weingarten map and normalized anisotropic $k$-mean curvature of $\partial\Sigma\subset\mathbb{R}^n$ with respect to Wulff shape $\overline{\W}$.
For $X\in\partial\Sigma$,
 $\<\widetilde{\nu_F}(X),E_{n+1}\>=0$  implies that $\widetilde{\nu_F}(\partial\Sigma)=\overline{\W}$. We know $T_X\Sigma=T_{\nu_F(X)}\W=T_{\widetilde{\nu_{{F}}}(X)}\widetilde{\W}$, then $T_X(\partial\Sigma)=T_X\Sigma\cap\mathbb{R}^n=T_{\widetilde{\nu_F}(X)}\widetilde{\W}\cap\mathbb{R}^n=T_{\widetilde{\nu_F}(X)}\overline{\W}$. Thus at point $\bar{z}=\widetilde{\nu_{{F}}}(X)\in\overline{\W}\subset\mathbb{R}^n$,  $\bar{\nu}(X)$ is also the 
 unit co-normal of 
   $\overline{\W}\subset \mathbb{R}^n$. 
 Then for $X\in\partial\Sigma\subset\mathbb{R}^n$, we can check that (see Fig. \ref{fig:1})
  \begin{align*}
 	\nu_{\bar{F}}(X)=\bar{z}=\widetilde{\nu_{{F}}}(X)=\nu_{{F}}(X)+\omega_0E_{n+1}^F,
 \end{align*}
   and
    \begin{align}
   	F(\nu)&=\<\nu_{{F}},\nu\>=\<\bar{z},\<\nu,\bar{\nu}\>\bar{\nu}+\<\nu,E_{n+1}\>E_{n+1}\> \>-\<\omega_0E_{n+1}^F,\nu\>\nonumber
   \\
   &=\bar{F}(\bar{\nu})\<\nu,\bar{\nu}\>-\omega_0\<E_{n+1}^F,\nu\>.\label{equ:barF(v)}
   \end{align} 

   For $Y,Z\in T(\partial\Sigma)$, we have
  \begin{align}
  	\<S_{\bar{F}}Z,Y\>=&\<\bar{D}_Z\nu_{\bar{F}},Y\>=
  	\<D_Z\nu_{\bar{F}},Y\>=\<D_Z(\omega_0E_{n+1}^F+\nu_F),Y\>\nonumber
  	\\
  	=&\<D_Z\nu_{{F}},Y\>=\<S_{{F}}Z,Y\>.\label{equ:SFab=SF-ab}
  \end{align}
   Besides, from \cite{Arxiv}*{Lemma 3.1} or \cite{Guo-Xia}*{Proposition 4.1}, we know
  \begin{align}\label{equ:SFau=0}
  	\<S_FY,\mu\>=0, \quad Y\in T(\partial\Sigma).
  \end{align}
  \begin{definition}
    	For $1\leq k\leq n$,  the anisotropic capillary quermassintegrals  is defined by
  	\begin{align}\label{equ:Vk}
  		\mathcal{V}_{k+1,\omega_0}(\Sigma)=
  		\frac{1}{n+1}
  		\left(
  		\int_{\Sigma}H^F_kF(\nu)d\mu_g
  		+\frac{\omega_0}{n}
  		\int_{\partial\Sigma} H_{k-1}^{\bar{F}}\bar{F}(\bar{\nu})ds
  		\right).
  	\end{align}
  \end{definition}

We remark that, if $\W=\mathbb{S}^n$, then $F(\mathbb{S}^n)=1,H_k^F=H_k,\nu_{{F}}=\Psi(\nu)=\nu,E_{n+1}^F=E_{n+1},\cos\theta=-\omega_0,\<\nu,\bar{\nu}\>=\sin\theta,\bar{F}(\mathbb{S}^{n-1})=\sin\theta, \overline{\W}=\sin\theta\cdot\mathbb{S}^{n-1},S_{\bar{F}}=\sin\theta\bar{h}$,   and $ H^{\bar{F}}_{k-1}=\sin^{k-1}\theta H^{\partial\Sigma}_{k-1}$, where $H_k$ is normalized $k$-mean curvature of $(\Sigma,g)\subset(\mathbb{R}^{n+1}, \<\cdot,\cdot\>)$, $\bar{h}$ and $H^{\partial\Sigma}_{k-1}$ are the Weingarten map and  normalized $k$-mean curvature of $\partial\Sigma\subset\mathbb{R}^n$, and $\theta$ is the contact angle defined by \eqref{equ:def-theta}. Thus we have
	\begin{align*}
	\mathcal{V}_{k+1,\omega_0}(\Sigma)=
	\frac{1}{n+1}
	\left(
	\int_{\Sigma}H_kd\mu_g
	-\frac{\cos\theta\sin^k\theta}{n}
	\int_{\partial\Sigma} H_{k-1}^{\partial\Sigma}ds
	\right), \quad k=1,\cdots,n,
\end{align*}
which matches the definition of  isotropic quermassintegrals \cite{Wang-Weng-Xia}*{eq.(1.8)} for  capillary hypersurfaces in the half-space.

The quermassintegrals defined above is suitable, since they satisfy the following variational formulas:
\begin{theorem}\label{Thm:p_t-Vk}
	 Let $\Sigma_t \subset \overline{\mathbb{R}}_{+}^{n+1}$ be a family of smooth anisotropic $\omega_0$-capillary hypersurfaces supported by $\partial \mathbb{R}_{+}^{n+1}$ with
	 $\omega_0 \in\left(-F(E_{n+1}),F(-E_{n+1})\right)$, given by the embedding $X(\cdot, t): \Sigma \rightarrow \overline{\mathbb{R}}_{+}^{n+1}$, satisfying
\begin{align}\label{equ:p_tX=fv}
	\partial_t X=f \nu_F+T,
\end{align}
for a smooth function $f$ and vector field $T\in T\Sigma_t$. Then for $-1 \leq k \leq n$,
\begin{align}
	\frac{d}{d t} \mathcal{V}_{k+1, \omega_0}\left({\Sigma_t}\right)=\frac{n-k}{n+1} \int_{\Sigma_t} f H^F_{k+1} d \mu_F.
\label{equ:Thm5}
\end{align}
\end{theorem}

Before proving Theorem \ref{Thm:p_t-Vk}, we remark that if $\Sigma_t \subset \overline{\mathbb{R}}_{+}^{n+1}$ is a family of smooth anisotropic $\omega_0$-capillary hypersurfaces evolving by \eqref{equ:p_tX=fv}, the tangential component $T$ of $\partial_t X$ must satisfy
$$
\left.T\right|_{\partial \Sigma_t}=f \frac{\omega_0}{\<\mu,E_{n+1}\>} \mu+\overline{T},
$$
where $\overline{T} \in T\left(\partial \Sigma_t\right)$, and $\<E_{n+1},\mu\>\neq 0$ since 
$\omega_0\in \left(-F(E_{n+1}),F(-E_{n+1})\right)$.
 In fact, the restriction of $X(\cdot, t)$ on $\partial \Sigma$ is contained in $\mathbb{R}^n$ and hence,
$$
f \nu_F+\left.T\right|_{\partial \Sigma_t}=\left.\partial_t X\right|_{\partial \Sigma} \in T \mathbb{R}^n .
$$
Assume $T=\lambda\mu+\overline{T}$, then $\overline{T}=T-\lambda\mu=(T+f\nu_{{F}})-(f\nu_{{F}}+\lambda\mu)\in T\mathbb{R}^n$, thus $f\nu_{{F}}+\lambda\mu\in T\mathbb{R}^n $, that is
\begin{align*}
	0=\<f\nu_{{F}}+\lambda\mu,E_{n+1}\>=-f\omega_0+\lambda\<\mu,E_{n+1}\>,
\end{align*}
which implies that $\lambda=f\frac{\omega_0}{\<\mu,E_{n+1}\>}$.
 For simplicity, in the following, we always assume that
\begin{align}\label{equ:T=}
	\left.T\right|_{\partial \Sigma_t}= \frac{f\omega_0}{\<\mu,E_{n+1}\>} \mu.
\end{align}
Flow \eqref{equ:p_tX=fv} can also be written as \begin{align}\label{equ:ptX=fv+T}
	\partial_tX=\hat{f}\nu+\widehat{T},
\end{align}
with $\hat{f}=fF(\nu)$ and $\widehat{T}=f\nabla^{\mathbb{S}} {F}(\nu)+T$.

\begin{lemma}[\cite{HL08}]\label{lemma2.5}
	Under the evolution \eqref{equ:ptX=fv+T}, we have
	\begin{align}
		\frac{\partial }{\partial t} \nu_t=&-\nabla\hat{f}+\mathrm{d}\nu_t(\widehat{T}),	\label{equ:ptv}
		\\[2pt]
		\frac{\p }{\p t} (d \mu_g)=&\(\operatorname{div} \widehat{T}+n H \hat{f}\) d \mu_g, \nonumber
			\\[2pt]
		\sigma_r^{\prime}(t)=&-\operatorname{div}\(T_{r-1}(\nabla \hat{f})\)-\hat{f}\left\langle T_{r-1} \circ \mathrm{d} \nu_t, \mathrm{~d} \nu_t\right\rangle+\left\langle\nabla \sigma_r, \widehat{T}\right\rangle, \nonumber
			\\[2pt]
		\left(F\left(\nu_t\right)\right)^{\prime}=&\left\langle\nabla^{\mathbb{S}} F, \nu_t^{\prime}\right\rangle, \nonumber
	\end{align}
	where  $nH=\operatorname{tr}(\mathrm{d}\nu)$. 
\end{lemma}

On $\partial\Sigma$, by \eqref{equ:w0-capillary}, \eqref{equ:ptv},  $E_{n+1}=\<E_{n+1},\mu\>\mu+\<E_{n+1},\nu\>\nu$, $\<D^2F(\nu)\nu,\cdot\>\equiv0$, $\nabla F(\nu)=\mathrm{d}\nu\nabla^{\mathbb{S}}F\circ\nu$, \eqref{equ:SFau=0},  and \eqref{equ:T=}, we have
\begin{align*}
	0=&\<\partial_t\nu_{{F}},E_{n+1}\>
	=\<A_F(\nu)\partial_t\nu,\<E_{n+1},\mu\>\mu\>
     \\
	=&\<E_{n+1},\mu\>\< -A_F(\nu)\nabla \hat{f}+ A_F(\nu)(\mathrm{d}\nu_t(\widehat{T}))  ,\mu\>
	\\
	=&-\< A_F(\nu)\nabla (fF(\nu)),\<E_{n+1},\mu\>\mu\>
	+
	\<S_F(f\nabla^{\mathbb{S}} {F}(\nu)+T) ,\<E_{n+1},\mu\>\mu\>
	\\=&-F(\nu)\<E_{n+1},\mu\>\< A_F(\nu)\nabla f,\mu\>
	+
	\frac{\omega_0 f }{\<\mu,E_{n+1}\>}\<E_{n+1},\mu\>\<S_F\mu  ,\mu\>,
\end{align*}
which implies
\begin{align}\label{equ:DT}
\< A_F(\nu)\nabla \hat{f},\mu\>	=\<S_F(\widehat{T}),\mu\>=\<S_F\mu,\mu\> \<\widehat{T},\mu\>=\<S_F\mu,\mu\> \<\partial_tX,\mu\>,
\end{align}
and
\begin{align*}
	\< A_F(\nu)\nabla f,\mu\>
	=
	\frac{\omega_0 f }{F(\nu)\<\mu,E_{n+1}\>}\<S_F\mu  ,\mu\>.
\end{align*}

\begin{proof}[Proof of Theorem \ref{Thm:p_t-Vk}]
	The case $k=-1$ and case $k=0$ are  from Lemma \ref{lemma-evolution-2}, thus we just show the cases $k=1,\cdots,n.$
		By Lemma \ref{lemma2.5}, Along flow \eqref{equ:ptX=fv+T} we have
	\begin{align}
		\frac{\partial }{\partial t}
		\left(	
		\int_{\Sigma_t}
		\sigma_{k}F(\nu)
		d \mu_g
		\right)		
		=&\int_{\Sigma_t}
		\<\nabla^{\mathbb{S}}F(\nu),\partial_t\nu\>\sigma_k
		+\partial_t\sigma_kF(\nu)
		d \mu_g
		+
		\int_{\Sigma_t}
		\sigma_kF(\nu)
		\partial_t(d \mu_g)
		\nonumber
		\\
		=&\int_{\Sigma_t}
       \sigma_k
		\left\{
		-\< \nabla^{\mathbb{S}} F\circ \nu ,  \nabla \hat{f} \> + \<\mathrm{d}\nu \nabla^{\mathbb{S}} F(\nu),\widehat{T} \>
		\right\}
		\nonumber
		\\
		+&F(\nu)
		\left\{
		-\operatorname{div}(T_{k-1}(\nabla \hat{f}))-\hat{f}\left\langle T_{k-1} \circ \mathrm{d} \nu, \mathrm{~d} \nu\right\rangle+\left\langle\nabla \sigma_k, \widehat{T}\right\rangle
		\right\}	
		\nonumber
		\\
		+&F(\nu)\sigma_{k}
		\left\{
		nH\hat{f}
		+
		\operatorname{div}\widehat{T}
		\right\}
		d \mu_g.
		\label{equ:pfthm5-1}
	\end{align}
	By divergence formula, we know
	\begin{align}
		\sigma_k\<\mathrm{d}\nu \nabla^{\mathbb{S}} F(\nu),\widehat{T} \>
		+F(\nu)\left\langle\nabla \sigma_k, \widehat{T}\right\rangle
		+F(\nu)\sigma_k	\operatorname{div}\widehat{T}
		=
		\div (F(\nu)\sigma_k\widehat{T}),
	\label{equ:pfthm5-2}
	\end{align}
		and
	\begin{align}
		-\<\nabla\hat{f},\sigma_k\nabla^{\mathbb{S}}F(\nu)\>=\hat{f}\div(\sigma_k\nabla^{\mathbb{S}}F(\nu))-\div(\hat{f}\sigma_k\nabla^{\mathbb{S}}F(\nu)).\label{equ:pfthm5-4}
	\end{align}
	Moreover, using  the symmetry of $T_{k-1} $, we see
	\begin{align}
	&-\div (\hat{f}T_{k-1}(\nabla F(\nu)))=-\hat{f}\div(T_{k-1}\nabla F(\nu))-\<\nabla \hat{f},T_{k-1}(\nabla F(\nu))\>\nonumber
	\\
	=&-\hat{f}\div(T_{k-1}\nabla F(\nu))-\div(F(\nu)T_{k-1}(\nabla \hat{f}))+F(\nu)\div(T_{k-1}(\nabla \hat{f})).\label{equ:pfthm5-3}
	\end{align}
	Combining with \eqref{equ:pfthm5-4}, \eqref{equ:pfthm5-3},  and $P_k=\sigma_k I-T_{k-1} \mathrm{~d} \nu$, we deduce
	\begin{align}
		&-\sigma_k\<\nabla\hat{f},\nabla^{\mathbb{S}}F(\nu)\>-F(\nu)\div\(T_{k-1}(\nabla \hat{f})\)
		\nonumber
		\\
		=&\div \(\hat{f}T_{k-1}(\nabla F(\nu))\)-\hat{f}\div\(T_{k-1}\nabla F(\nu)\)-\div\(F(\nu)T_{k-1}(\nabla \hat{f})\)
		\nonumber
		\\
		&+\hat{f}\div\(\sigma_k\nabla^{\mathbb{S}}F(\nu)\)-\div\(\hat{f}\sigma_k\nabla^{\mathbb{S}}F(\nu)\)
		\nonumber
		\\
		=&\hat{f}\div\(P_k\nabla^{\mathbb{S}}F(\nu)\)-\div\(\hat{f}P_k\nabla^{\mathbb{S}}F(\nu)\)-\div\(F(\nu)T_{k-1}(\nabla \hat{f})\),\label{equ:pfthm5-5}
	\end{align}
	and
	\begin{align}
		nH\hat{f}F(\nu)\sigma_{k}-F(\nu)\hat{f}\left\langle T_{k-1} \circ \mathrm{d} \nu, \mathrm{~d} \nu\right\rangle=F(\nu)\tr (P_k\circ\mathrm{d}\nu)\hat{f}.	\label{equ:pfthm5-6}
	\end{align}
	By \eqref{equ:pfthm5-5}, \eqref{equ:pfthm5-6}, and \eqref{equ:lemma2.1.1}, we can calculate that
	\begin{align}
		&-\sigma_k\<\nabla\hat{f},\nabla^{\mathbb{S}}F(\nu)\>-F(\nu)\div\(T_{k-1}(\nabla \hat{f})\)
		+nH\hat{f}F(\nu)\sigma_{k}-F(\nu)\hat{f}\left\langle T_{k-1} \circ \mathrm{d} \nu, \mathrm{~d} \nu\right\rangle\nonumber
		\\
		=&\hat{f}\div\(P_k\nabla^{\mathbb{S}}F(\nu)\)-\div\(\hat{f}P_k\nabla^{\mathbb{S}}F(\nu)\)-\div\(F(\nu)T_{k-1}(\nabla \hat{f})\)+F(\nu)\tr \(P_k\circ\mathrm{d}\nu\)\hat{f}\nonumber
		\\
		=&(k+1)\hat{f}\sigma_k-\div\(\hat{f}P_k\nabla^{\mathbb{S}}F(\nu)\)-\div\(F(\nu)T_{k-1}(\nabla \hat{f})\).
		\label{equ:pfthm5-7}
	\end{align}
	Putting \eqref{equ:pfthm5-2} and  \eqref{equ:pfthm5-7} into  \eqref{equ:pfthm5-1} yields
	\begin{align}
		\frac{\partial }{\partial t}
		\left(	
		\int_{\Sigma_t}
		\sigma_{k}F(\nu)
		d \mu_g
		\right)		
		=&
		\int_{\partial\Sigma_t}\<F(\nu)\sigma_k\widehat{T} - \hat{f}P_k\nabla^{\mathbb{S}}F(\nu)  -  F(\nu)T_{k-1}(\nabla \hat{f}),\mu\>
		ds\nonumber
		\\
		&+\int_{\Sigma_t}
		(k+1)\sigma_{k+1}\hat{f}
		d \mu_g.	\label{equ:pfthm5-8}
	\end{align}
	 Next, on  $\partial\Sigma_t$, from \eqref{equ:SFau=0} we know $\<P_kY,\mu\>=0$ for any $Y\in T(\partial\Sigma)$. Then  we have
	\begin{align}\label{equ:pfthm5-9}
		\<F(\nu)\sigma_k\widehat{T} ,\mu\>=F(\nu)\sigma_k\<\widehat{T}+\hat{f}\nu ,\mu\>=F(\nu)\sigma_k\<\partial_tX ,\mu\>=\sigma_k\<\partial_tX ,\mu\>\<\nu_F,\nu\>,
	\end{align}
	and
	\begin{align}
		\<- \hat{f}P_k\nabla^{\mathbb{S}}F(\nu),\mu\>=-\hat{f}\< P_k\mu,\mu\> \<\nabla^{\mathbb{S}}F(\nu),\mu \>=-\<\partial_tX,\nu\>\< P_k\mu,\mu\> \<\nu_F ,\mu \>.\label{equ:pfthm5-10}
	\end{align}
	By \eqref{equ:T_k=P_kAF}, \eqref{equ:DT}, we can derive
	\begin{align}
		\< -F(\nu)T_{k-1}(\nabla \hat{f}),\mu\>&=-F(\nu)\< P_{k-1}\mu,\mu\>\<A_F(\nu)\nabla \hat{f},\mu\>\nonumber
		\\
		&=-\< P_{k-1}\mu,\mu\> \<S_F\mu,\mu\> \<\partial_tX,\mu\>\<\nu_F,\nu\>.\label{equ:pfthm5-11}
	\end{align}
	By \eqref{equ:pfthm5-9}$\thicksim$\eqref{equ:pfthm5-11} and $\<P_k\mu,\mu\>=\sigma_k-\<P_{k-1}\mu,\mu\> \<S_F\mu,\mu\>$, we deduce
	\begin{align}
		&\<F(\nu)\sigma_k\widehat{T} - \hat{f}P_k\nabla^{\mathbb{S}}F(\nu)  -  F(\nu)T_{k-1}(\nabla \hat{f}),\mu\>\nonumber
        \\
		=&\<P_k\mu,\mu\> \<\partial_tX,\<\nu_{{F}},\nu\>\mu-\<\nu_{{F}},\mu\>\nu\>\nonumber
		\\
		=&\<P_k\mu,\mu\> \<\partial_tX,\mathcal{R}(\mathcal{P}(\nu_{{F}}))\>\nonumber
         \\
		=&\<P_k\mu,\mu\> \<\partial_tX,\<\nu_{{F}},E_{n+1}\>\bar{\nu}-\<\nu_{{F}},\bar{\nu}\>E_{n+1}\>\nonumber
		\\
		=&-\omega_0\<P_k\mu,\mu\> \<\partial_tX,\bar{\nu}\>.\label{equ:pfthm5-12}
	\end{align}
	By \eqref{equ:SFab=SF-ab} and \eqref{equ:SFau=0}, we can check that by Mathematical induction method
	\begin{align}
		\label{equ:pfthm5-13}
		\<P_k\mu,\mu\>=\sigma_{k}(S_{\bar{F}}), \quad k=1,\cdots,n,
	\end{align}
	since $\<P_k\mu,\mu\>=\sigma_k(S_F)-\<P_{k-1}S_F\mu,\mu\>=\sigma_k(S_F)-\<P_{k-1}\mu,\mu\>\<S_F\mu,\mu\>$.
		Combining with \eqref{equ:pfthm5-8}, \eqref{equ:pfthm5-12}, and \eqref{equ:pfthm5-13}, we have
	\begin{align}
		\frac{\partial }{\partial t}
		\left(	
		\int_{\Sigma_t}
		\sigma_{k}F(\nu)
		d \mu_g
		\right)		
		=\int_{\Sigma_t}
		(k+1)\sigma_{k+1}\hat{f}
		d \mu_g
		-\omega_0\int_{\partial\Sigma_t}
		\sigma_{k}(S_{\bar{F}})\<\partial_tX,\bar{\nu}\>
		ds
		.	\label{equ:pfthm5-14}
	\end{align}
			On the other hand, since $\partial\Sigma$ is a closed hypersurface in $\mathbb{R}^n$, by variational formula in \cite{HL08}*{Theorem 3.3}, we have
	\begin{align}
		\frac{\partial }{\partial t}
		\left(	
		\int_{\partial\Sigma_t}
		\sigma_{k-1}(S_{\bar{F}})\bar{F}(\bar{\nu})
		d \mu_s
		\right)		
		=k\int_{\partial\Sigma_t}
		\sigma_{k}(S_{\bar{F}})\<\partial_tX,\bar{\nu}\>
		d s
		.	\label{equ:pfthm5-15}
	\end{align}
	Combining with $\hat{f}=fF(\nu)$,   $\frac{k}{n\binom{n-1}{k-1}}=\frac{1}{\binom{n}{k}}$,  \eqref{equ:pfthm5-14}, and \eqref{equ:pfthm5-15}, we obtain
	$$	\frac{d}{d t} \mathcal{V}_{k+1, \omega_0}\left({\Sigma_t}\right)=\frac{k+1}{(n+1)\binom{n}{k}} \int_{\Sigma_t} f \sigma_{k+1} d \mu_F,\quad k=1,\cdots,n.$$
   Since $\sigma_{n+1}=0$, and 	$\frac{k+1}{\binom{n}{k}}=\frac{n-k}{\binom{n}{k+1}}$ for $1\leq k<n$, we derive
	 \eqref{equ:Thm5} for $k=1,\cdots,n$.
\end{proof}

Then we can derive  monotonicity results as follows:
\begin{theorem}\label{Thm6}
	If $\Sigma_t$ is a smooth strictly anisotropic $k$-convex (i.e. $H^F_k>0$, $H_k^F$ defined by \eqref{equ:Hk}), star-shaped solution to the  flow \eqref{equ:flow}, then  $\mathcal{V}_{k, \omega_0}\left({\Sigma_t}\right)$ (defined by \eqref{equ:Vk}) is non-increasing for $1 \leq k < n$, and $\mathcal{V}_{k,\omega_0}({\Sigma_t})$ is a constant function if and only if  $\Sigma_t$ is an $\omega_0$-capillary Wulff shape.
\end{theorem}
\begin{proof}
	Using Theorem \ref{Thm:p_t-Vk}, and anisotropic capillary Minkowski Formula \eqref{equ:Minkow}, we see
	\begin{align}
		\partial_t \mathcal{V}_{k, \omega_0}\left(\widehat{\Sigma_t}\right) & =\frac{n+1-k}{n+1} \int_{\Sigma_t} f H^F_k d \mu_F\nonumber \\
		& =\frac{n+1-k}{n+1} \int_{\Sigma_t}\(n H^F_k\left(1+\omega_0G(\nu_F)(\nu_F, E^F_{n+1})\right)-H_F H^F_k\hat{ u}\) d \mu_F\nonumber \\
		& \leq \frac{n(n+1-k)}{n+1} \int_{\Sigma_t}\(H^F_k\left(1+\omega_0G(\nu_F)(\nu_F, E^F_{n+1})\right)-H^F_{k+1}\hat{ u}\) d \mu_F \nonumber\\
		& =0,\quad \text{for\ } 1\leq k<n, \label{equ:ptVk}
	\end{align}
	where we have used the Newton-MacLaurin inequality $H^F_1 H^F_k \geq H^F_{k+1}$ and $\hat{ u}>0$ for strictly anisotropic $k$-convex, star-shape hypersurfaces,  and  equality holds above if and only if  $\Sigma$ is an $\omega_0$-capillary Wulff shape.  
\end{proof}

 We finally arrive at proving    the generalized Alexandrov-Fenchel inequalities for anisotropic capillary hypersurfaces in the half-space.
\begin{proof}[Proof of  Theorem \ref{thm:AF-neq}]
	Taking flow \eqref{equ:flow} with the initial hypersurface $\Sigma$,  from Theorem \ref{thm:flow}, we know it converges to an $\omega_0$-capillary Wulff shape ${\W_{r_0,\omega_0}}$. And by Lemma \ref{lemma-evolution}
	we can see
	\begin{align*}
		\mathcal{V}_{0,\omega_0}(\Sigma)= |\widehat{\W_{ r_0,\omega_0}}|=\int_{\W_{ r_0,\omega_0}} \<X,\nu\> d\mu_g=r_0^{n+1}|\widehat{\W_{1,\omega_0}}|,
	\end{align*}
and	we can use Theorem \ref{Thm6} to derive following inequalities, since the convexity is preserved along  the flow by Theorem \ref{thm:convex preservation},
	\begin{align*}
		&\mathcal{V}_{k,\omega_0}(\Sigma)\geq \mathcal{V}_{k,\omega_0}(\W_{ r_0,\omega_0})\\
		=&\frac{1}{n+1}
		\left(
		\int_{\W_{ r_0,\omega_0}}H^F_{k-1}F(\nu)d\mu_g
		+\frac{\omega_0}{n}
		\int_{\partial \W_{ r_0,\omega_0}} H_{k-2}^{\bar{F}}\bar{F}(\bar{\nu})ds
		\right)
		\\
		=&\frac{1}{n+1}
		\left(
		r_0^{n-(k-1)}\int_{\W_{1,\omega_0}}H^F_{k-1}F(\nu)d\mu_g
		+\frac{\omega_0}{n}
		r_0^{(n-1)-(k-2)}\int_{\partial \W_{1,\omega_0}} H_{k-2}^{\bar{F}}\bar{F}(\bar{\nu})ds
		\right)
		\\
=&r_0^{n+1-k}\mathcal{V}_{k,\omega_0}(\W_{1,\omega_0}), \quad \text{for} \ k=1,\cdots,n-1.
	\end{align*}
	Thus
	\begin{align*}
		\mathcal{V}_{k,\omega_0}(\Sigma)\geq
		r_0^{n+1-k}\mathcal{V}_{k,\omega_0}(\W_{1,\omega_0})
		=\left(\frac{\mathcal{V}_{0,\omega_0}(\Sigma)}{\mathcal{V}_{0,\omega_0}(\W_{1,\omega_0})}
		\right)
		^{\frac{n+1-k}{n+1}}\mathcal{V}_{k,\omega_0}(\W_{1,\omega_0}),
	\end{align*}
	where $ k=1,\cdots,n-1,$ and equality holds above if and only if equality holds in \eqref{equ:ptVk}, that means $\Sigma$ is an $\omega_0$-capillary Wulff shape.
	Then we complete the proof.
\end{proof}

\section{The Alexandrov-Fenchel inequalities  for capillary convex bodies}\label{sec 7}
In this section, we list some definitions and properties of capillary convex bodies geometry (see \cite{Xia-arxiv}),  and derive  Lemma \ref{lemma:7.1} and  \ref{lemma:7.2}, then prove  Corollary \ref{cor:7.3}.
\begin{definition}[\cite{Xia-arxiv}]\label{def:capillary-condex-body}
	For $\theta \in(0, \pi)$, we say that $\Sigma$ is a capillary hypersurface (with constant contact angle $\theta$) in $\overline{\mathbb{R}_{+}^{n+1}}$ if
$$
\langle\nu, E_{n+1}\rangle=\cos \theta, \quad \text { along } \partial \Sigma .
$$
We call $\widehat{\Sigma}$ a capillary convex body if $\widehat{\Sigma}$ is a convex body (a compact convex set with nonempty interior) in $\mathbb{R}^{n+1}$ and $\Sigma$ is a capillary hypersurface in $\overline{\mathbb{R}_{+}^{n+1}}$. The class of capillary convex bodies in $\overline{\mathbb{R}_{+}^{n+1}}$ is denoted by $\mathcal{K}_\theta$.
\end{definition}
We denote
$$
\mathcal{C}_{\theta}:=\left\{\xi \in \overline{\mathbb{R}_{+}^{n+1}}: | \xi+\cos \theta E_{n+1} \mid=1\right\}.
$$
It's easy to check that $\widehat{\mathcal{C}_{\theta}}\in \mathcal{K}_\theta$.
In the following, we denote $\sigma$  and ${\rm d} \mu_{\sigma}$ the round metric and its associated volume form on $\mathcal{C}_\theta$ respectively. And $\nabla^{\sigma}$ is the Levi-Civita connection of $\sigma$ on $\mathcal{C}_\theta$.

\begin{definition}[\cite{Convex-book,Xia-arxiv}]
The mixed discriminant $Q:\left(\mathbb{R}^{n \times n}\right)^n \rightarrow \mathbb{R}$ is defined by	
	\begin{align}
		\operatorname{det}\left(\lambda_1 A_1+\cdots+\lambda_m A_m\right)=\sum_{i_1, \cdots, i_n=1}^m \lambda_{i_1} \cdots \lambda_{i_n} Q\left(A_{i_1}, \cdots, A_{i_n}\right),\label{equ:QA}
	\end{align}
	for $m \in \mathbb{N}, \lambda_1, \cdots, \lambda_m \geq 0$ and the real symmetric matrices $A_1, \cdots, A_m \in \mathbb{R}^{n \times n}$. If $\left(A_k\right)_{i j}$ denotes the $(i, j)$-element of the matrix $A_k$, then
	$$
	Q\left(A_1, \cdots, A_n\right)=\frac{1}{n!} \sum_{i_1, \cdots, i_n, j_1, \cdots, j_n} \delta_{j_1 \cdots j_n}^{i_1 \cdots i_n}\left(A_1\right)_{i_1 j_1} \cdots\left(A_n\right)_{i_n j_n}.
	$$
\end{definition}
	From \eqref{equ:QA}, we can check that for any invertible matrix $B\in \mathbb{R}^{n \times n}$, we have
	\begin{align}\label{equ:Q(AB)=Q(A)det(B)}
		Q\left(A_1B, \cdots, A_nB\right)=Q\left(A_1, \cdots, A_n\right)\det(B).
	\end{align}

	\begin{definition}[\cite{Xia-arxiv}]
		For any $f \in C^2\left(\mathcal{C}_\theta\right)$, we denote $A[f]:=\nabla^{\sigma}\nabla^{\sigma} f+f \sigma$. A multi-variable function $V:\left(C^2\left(\mathcal{C}_\theta\right)\right)^{n+1} \rightarrow \mathbb{R}$ is defined by	
	$$
	V\left(f_1, \cdots, f_{n+1}\right)=\frac{1}{n+1} \int_{\mathcal{C}_\theta} f_1 Q\left(A\left[f_2\right], \cdots, A\left[f_{n+1}\right]\right) d \mu_{\sigma} ,
	$$	
	for $f_i \in C^2\left(\mathcal{C}_\theta\right), 1 \leq i \leq n+1$.
	\end{definition}
	
Mei, Wang, Weng, and Xia \cite{Xia-arxiv}  introduced the Minkowski sum $K:=\sum_{i=1}^m \lambda_i K_i$  for $K_1, \cdots, K_m \in \mathcal{K}_\theta$, and proved that $V$ is a symmetric and multi-linear function, moreover
the mixed volume $V\left(K_{i_1}, \cdots, K_{i_{n+1}}\right)$ defined by
$$
|K|=\sum_{i_1, \cdots, i_{n+1}=1}^m \lambda_{i_1} \cdots \lambda_{i_{n+1}} V\left(K_{i_1}, \cdots, K_{i_{n+1}}\right),
$$
for capillary convex bodies in $\mathcal{K}_\theta$ can be written as an integral over $\mathcal{C}_\theta$ by using the capillary support function.
\begin{proposition}[\cite{Xia-arxiv}*{Proposition 2.12}]\label{prop7.3}
	Let $K_1, \cdots, K_m \in \mathcal{K}_\theta$ and $h_i$ be the capillary support function of $K_i$ (where $h_i(\xi)=u_i(\nu^{-1}(\xi+\cos\theta E_{n+1}))$, for $\xi\in\mathcal{C}_{\theta}$, and $u_i$ is the support function of $K_i$). Then we have
\begin{align}\label{equ:VKKKLLL}
	V\left(K_{i_1}, \cdots, K_{i_{n+1}}\right)&=V\left(h_{i_1}, \cdots, h_{i_{n+1}}\right)
	\\
	&=\frac{1}{n+1} \int_{\mathcal{C}_\theta} h_{i_1} Q\left(A\left[h_{i_2}\right], \cdots, A\left[h_{i_{n+1}}\right]\right) d \mu_{\sigma} .\nonumber
\end{align}
\end{proposition}

\begin{lemma}\label{lemma:7.1}
	If $\W_{1,\omega_0}=(\W+\omega_0E_{n+1}^F)\cap\overline{\mathbb{R}^{n+1}_+}$ (also defined by \eqref{equ:crwo}) is a capillary hypersurface with constant contact angle $\theta$ (i.e., $\widehat{\W_{1,\omega_0}}\in \mathcal{K}_{\theta}$) for $\theta\in (0,\pi)$, then we have
	\begin{itemize}
		\item[ (i) ] Condition \ref{condition}  is equivalent to $\theta\in(0,\frac{\pi}{2}]$;
		\item[(ii)] An anisotropic $\omega_0$-capillary hypersurface $\Sigma$ is also a capillary hypersurface with constant contact angle $\theta$.
	\end{itemize}
\end{lemma}
\begin{proof}
	(i) From Lemma \ref{lem:Q-translate}, we know that the Wulff shape $\W$ and $\omega_0$ satisfies Condition \ref{condition} if and only if for new Wulff shape $\widetilde{\W}=\W+\omega_0E_{n+1}^F$ (with respect to $\tilde{F}^0$), 
	 \begin{align}\label{equ:pf-lemma7.1-0}
		\tilde{Q}(\tilde{z})(e_{\alpha},e_{\alpha},e_n)\leq 0, \quad\text{for}\  {\tilde{z}}\in\partial \W_{1,\omega_0},
	\end{align}
	with standard orthogonal frame  
	$\{e_{\alpha}\}_{\alpha=1}^{n-1}\in T_{\tilde{z}}{\left(\partial \W_{1,\omega_0}\right)}$ with respect to $\tilde{g}$, and
	$e_n=\frac{A_F(\nu)\mu}{\left|A_F(\nu)\mu\right|_{\tilde{g}}}=\frac{A_F(\nu)\mu}{\left(\tilde{G}(\tilde{z})(A_F(\nu)\mu,A_F(\nu)\mu)\right)^{\frac{1}{2}}}$ for ${\tilde{z}}\in\partial \W_{1,\omega_0}$, where  $\nu$ and $\mu$ are the unit outward normal of $\partial \W_{1,\omega_0}\subset\overline{\mathbb{R}^{n+1}_+}$ and $ \partial \W_{1,\omega_0}\subset  \W_{1,\omega_0}$.
	We check that the inequality \eqref{equ:pf-lemma7.1-0} holds if and only if $\theta\in(0,\frac{\pi}{2}]$.

	 Since $\widehat{\W_{1,\omega_0}}\in \mathcal{K}_{\theta}$ for $\theta\in (0,\pi)$, we conclude that
	\begin{align}\label{equ:pf-lemma7.1-1}
		\<\nu({\tilde{z}}),E_{n+1}\>=\cos\theta, \quad \text{for}\ {\tilde{z}}\in \partial \W_{1,\omega_0}.
	\end{align}
	By \eqref{equ:crwo} we have $\W_{1,\omega_0}\in\widetilde{\W}$, then we see  $\W_{1,\omega_0}=\{{\tilde{z}}\in\overline{\mathbb{R}^{n+1}_+}:\tilde{F}^0({\tilde{z}})=1\}$, which implies
	\begin{align}\label{equ:pf-lemma7.1-2}
		\nu({\tilde{z}})=\frac{D\tilde{F}^0(\tilde{z})}{\<D\tilde{F}^0(\tilde{z}),D\tilde{F}^0(\tilde{z})\>^{\frac{1}{2}}}:=\frac{D\tilde{F}^0(\tilde{z})}{\left|D\tilde{F}^0(\tilde{z})\right|},
		 \quad \text{for}\ {\tilde{z}}\in   \W_{1,\omega_0}.
	\end{align}
	Taking $\{\epsilon_i\}_{i=1}^{n+1}$ the standard orthogonal frame on $\mathbb{R}^{n+1}$ with $\epsilon_{n+1}=E_{n+1}$, denote  that  $\tilde{F}^0_i=D_{\epsilon_i}\tilde{F}^0(\tilde{z}),\tilde{F}^0_{i,\alpha}=D_{e_{\alpha}}D_{\epsilon_i}\tilde{F}^0(\tilde{z}), \tilde{F}^0_{i,\alpha\beta}=D_{e_{\beta}}D_{e_{\alpha}}D_{\epsilon_i}\tilde{F}^0(\tilde{z})$ for $i=1,\cdots,n+1$ and $\alpha,\beta=1,\cdots,n-1$. Combining with \eqref{equ:pf-lemma7.1-1} and \eqref{equ:pf-lemma7.1-2}, we have
	\begin{align*}
		\left(\tilde{F}^0_{n+1}\right)^2=\cos^2\theta\sum_{i=1}^{n+1}\left(\tilde{F}^0_{i}\right)^2, \quad\text{for}\ {\tilde{z}}\in\partial \W_{1,\omega_0}.
	\end{align*}
	Taking the covariant derivative $D_{e_{\alpha}}$  on both sides of the above equation yields
	\begin{align}\label{equ:pf-lemma7.1-3}
		\tilde{F}^0_{n+1}\tilde{F}^0_{n+1,\alpha}=\cos^2\theta \sum_{i=1}^{n+1}\tilde{F}^0_{i}\tilde{F}^0_{i,\alpha},\quad\forall \alpha=1,\cdots,n=1,\, {\tilde{z}}\in\partial \W_{1,\omega_0}.
	\end{align}
	Taking the covariant derivative $D_{e_{\beta}}$  again, for any $\alpha,\beta\in\{1,\cdots,n-1\}$ and ${\tilde{z}}\in\partial \W_{1,\omega_0}$, we have
	\begin{align}\label{equ:pf-lemma7.1-4}
		\tilde{F}^0_{n+1,\beta}\tilde{F}^0_{n+1,\alpha}+\tilde{F}^0_{n+1}\tilde{F}^0_{n+1,\alpha\beta}=\cos^2\theta \sum_{i=1}^{n+1}\tilde{F}^0_{i,\beta}\tilde{F}^0_{i,\alpha}+\cos^2\theta \sum_{i=1}^{n+1}\tilde{F}^0_{i}\tilde{F}^0_{i,\alpha\beta}.
	\end{align}
	From \eqref{equ:pf-lemma7.1-1}, \eqref{equ:pf-lemma7.1-2},  \eqref{equ:G}, and the fact that $\<D\tilde{F}^0(\tilde{z}),e_{\alpha}\>=\left|D\tilde{F}^0(\tilde{z})\right|\cdot\<\nu,e_{\alpha}\>=0$, we can rewrite  \eqref{equ:pf-lemma7.1-3}  as
	\begin{align*}
		\cos\theta\cdot\left|D\tilde{F}^0(\tilde{z})\right|\cdot \tilde{G}(\tilde{z})(e_{\alpha},E_{n+1})=\cos^2\theta \cdot \left|D\tilde{F}^0(\tilde{z})\right|\cdot \tilde{G}(\tilde{z})(e_{\alpha},\nu),
	\end{align*}
	which implies
	\begin{align*}
      0&= \cos \theta\cdot \tilde{G}(\tilde{z})(e_{\alpha},\cos\theta\cdot\nu-E_{n+1})
      \\
      &=\cos\theta\cdot\sin\theta\cdot \tilde{G}(\tilde{z})(e_{\alpha},\mu),
	\end{align*}
	where we use the fact that $\left|D\tilde{F}^0(\tilde{z})\right|\neq 0$ and $E_{n+1}=-\sin\theta\cdot\mu+\cos\theta\cdot\nu$ (\cite{Wang-Weng-Xia}*{eq.(2.7)}). Then we have either $\theta=\frac{\pi}{2}$, or
	\begin{align}
		\label{equ:pf-lemma7.1-5}
		\tilde{G}(\tilde{z})(e_{\alpha},\mu)=0, \quad\forall \alpha=1,\cdots,n-1,\, {\tilde{z}}\in\partial \W_{1,\omega_0}.
	\end{align}
	
	\textbf{Case 1:} If $\theta\neq\frac{\pi}{2}$, then equation \eqref{equ:pf-lemma7.1-5} holds, which implies that
	\begin{align}
		\mu=&\sum_{\alpha=1}^{n-1}\tilde{G}(\tilde{z})(\mu,e_{\alpha})e_{\alpha}+\tilde{G}(\tilde{z})(\mu,e_n)e_n\nonumber
		\\
		=&\frac{\tilde{G}(\tilde{z})(A_F(\nu)\mu,\mu)}{\tilde{G}(\tilde{z})(A_F(\nu)\mu,A_F(\nu)\mu)}A_F(\nu)\mu\nonumber
		\\
		\overset{\eqref{equ:G(muF,Y)=<mu,Y>}}{=}&\frac{A_F(\nu)\mu}{\<A_F(\nu)\mu,\mu\>}\label{equ:pf-lemma7.1-6}.
	\end{align}
	From \eqref{equ:pf-lemma7.1-1}, \eqref{equ:pf-lemma7.1-2},  \eqref{equ:G}, and $\<D\tilde{F}^0(\tilde{z}),e_{\alpha}\>=\left|D\tilde{F}^0(\tilde{z})\right|\cdot\<\nu,e_{\alpha}\>=0$, the equation \eqref{equ:pf-lemma7.1-4} can be written as
	\begin{align}
		&\cos\theta\cdot\left|D\tilde{F}^0(\tilde{z})\right|\cdot\left(\tilde{F}^0_{n+1,\alpha\beta}-\cos\theta\cdot\sum_{i=1}^{n+1}\<\nu,\epsilon_i\>\tilde{F}^0_{i,\alpha\beta}\right)\nonumber
		\\
		=&\cos^2\theta\cdot\sum_{i=1}^{n+1}\tilde{G}(\tilde{z})(e_{\alpha},\epsilon_i)\tilde{G}(\tilde{z})(e_{\beta},\epsilon_i)	-\tilde{G}(\tilde{z})(e_{\alpha},E_{n+1})		\tilde{G}(\tilde{z})(e_{\beta},E_{n+1}).\label{equ:pf-lemma7.1-7}
	\end{align}
	By $E_{n+1}=-\sin\theta\cdot\mu+\cos\theta\cdot\nu$, \eqref{equ:Q}, $\<D\tilde{F}^0(\tilde{z}),e_{\alpha}\>=0$,  and $\<D\tilde{F}^0(\tilde{z}),\mu\>=|D\tilde{F}^0(\tilde{z})|\cdot\<\nu,\mu\>=0$, we obtain
	\begin{align}
		&\tilde{F}^0_{n+1,\alpha\beta}-\cos\theta\cdot\sum_{i=1}^{n+1}\<\nu,\epsilon_i\>\tilde{F}^0_{i,\alpha\beta}=\sum_{i=1}^{n+1}\<E_{n+1}-\cos\theta\cdot\nu,\epsilon_i\>\tilde{F}^0_{i,\alpha\beta}\nonumber
		\\
		=&-\sin\theta\sum_{i=1}^{n+1}\<\mu,\epsilon_i\>\tilde{F}^0_{i,\alpha\beta}\nonumber
		=-\sin\theta\cdot   \tilde{Q}(\tilde{z})(e_{\alpha},e_{\beta},\mu)
		\\
		\overset{\eqref{equ:pf-lemma7.1-6}}{=}&-\frac{\sin\theta}{\<A_F(\nu)\mu,\mu\>} \tilde{Q}(\tilde{z})(e_{\alpha},e_{\beta},A_F(\nu)\mu),\label{equ:pf-lemma7.1-8}
	\end{align}
	and
	\begin{align}
		\tilde{G}(\tilde{z})(e_{\alpha},E_{n+1})		\tilde{G}(\tilde{z})(e_{\beta},E_{n+1})\overset{\eqref{equ:pf-lemma7.1-5}}{=}\cos^2\theta\cdot \tilde{G}(\tilde{z})(e_{\alpha},\nu)		\tilde{G}(\tilde{z})(e_{\beta},\nu)
		\label{equ:pf-lemma7.1-9}.
	\end{align}
	Putting \eqref{equ:pf-lemma7.1-8}, \eqref{equ:pf-lemma7.1-9}, $\cos\theta\neq 0$, and $\left|D\tilde{F}^0(\tilde{z})\right|=\frac{1}{F(\nu)}$ into \eqref{equ:pf-lemma7.1-7} yields
	\begin{align}
		 &\tilde{Q}(\tilde{z})(e_{\alpha},e_{\alpha},A_F(\nu)\mu)\nonumber
		 \\
		 =&-\cot\theta\cdot F(\nu) \<A_F(\nu)\mu,\mu\>\left(\sum_{i=1}^{n+1}\left(\tilde{G}(\tilde{z})(e_{\alpha},\epsilon_i)\right)^2-\left(\tilde{G}(\tilde{z})(e_{\alpha},\nu)\right)^2\right).\label{equ:pf-lemma7.1-10}
	\end{align}
	We take another standard orthogonal frame   $\{\eta_i\}_{i=1}^{n+1}$  on $T_{\tilde{z}}\mathbb{R}^{n+1}$ with ${\tilde{z}}\in\partial \W_{1,\omega_0}$ and $\eta_{n+1}=\nu$,
	then $\{\eta_{i}\}_{i=1}^{n}\in T_{\tilde{z}}\W_{1,\omega_0}$, and there exists an $i_0\in\{1,\cdots,n\}$ such that  $\tilde{G}(\tilde{z})(e_{\alpha},\eta_{i_0})\neq 0$, otherwise we have $e_{\alpha}\notin T_{\tilde{z}}\W_{1,\omega_0}$ which leads to a contradiction.
	So we have \begin{align}
		&\sum_{i=1}^{n+1}\left(\tilde{G}(\tilde{z})(e_{\alpha},\epsilon_i)\right)^2-\left(\tilde{G}(\tilde{z})(e_{\alpha},\nu)\right)^2=\sum_{i=1}^{n+1}\left(\tilde{G}(\tilde{z})(e_{\alpha},\eta_i)\right)^2-\left(\tilde{G}(\tilde{z})(e_{\alpha},\nu)\right)^2\nonumber
		\\
		=&\sum_{i=1}^{n}\left(\tilde{G}(\tilde{z})(e_{\alpha},\eta_i)\right)^2\geq \left(\tilde{G}(\tilde{z})(e_{\alpha},\eta_{i_0})\right)^2>0.\label{equ:pf-lemma7.1-11}
	\end{align}
	From \eqref{equ:pf-lemma7.1-10} and \eqref{equ:pf-lemma7.1-11} we know that, for case $\theta\neq\frac{\pi}{2}$, the  inequality \eqref{equ:pf-lemma7.1-0} holds if and only if $\cot\theta\geq 0$, i.e., $\theta\in(0,\frac{\pi}{2})$.
	
	\textbf{Case 2:} If $\theta=\frac{\pi}{2}$, we have $E_{n+1}=-\sin\theta\cdot\mu+\cos\theta\cdot \nu=-\mu.$ From \eqref{equ:pf-lemma7.1-1} and \eqref{equ:pf-lemma7.1-2}, we have
	\begin{align*}
		\tilde{F}^0_{n+1}=0, \quad\text{for}\ {\tilde{z}}\in \partial \W_{1,\omega_0}.
	\end{align*}
	Similarly, for $\alpha,\beta\in\{1,\cdots,n-1\}$, we take the covariant derivative   on both sides of the above equation yields
	\begin{align*}
		\tilde{F}^0_{n+1,\alpha}=0, \quad\text{for}\ {\tilde{z}}\in \partial \W_{1,\omega_0};
		\\
		\tilde{F}^0_{n+1,\alpha\beta}=0, \quad\text{for}\ {\tilde{z}}\in \partial \W_{1,\omega_0},
	\end{align*}
	which implies
	\begin{align*}
		\tilde{G}(\tilde{z})(\mu,e_{\alpha})&=-\tilde{G}(\tilde{z})(E_{n+1},e_{\alpha})=-F^{0}_{n+1,\alpha}=0;\\
		\tilde{Q}(\tilde{z})(\mu,e_\alpha,e_{\alpha}
		)&=-\tilde{F}^0_{n+1,\alpha\alpha}=0.
	\end{align*}
	By the first equation above, we have $\mu=\sum_{i=1}^{n}\tilde{G}(\tilde{z})(\mu,e_i)=\lambda e_n$, where $\lambda=\tilde{G}(\tilde{z})(\mu,e_n)\neq 0$, combining with the second equation above, we have
	\begin{align*}
		\tilde{Q}(\tilde{z})(e_n,e_{\alpha},e_{\alpha})=0.
	\end{align*}
	Combining with \textbf{Case 1} and \textbf{Case 2}, we complete the proof of (i).
	\\
	
	(ii) Denote by $\nu_F^{\Sigma}$ the anisotropic unit outward normal of $\Sigma$. Since $\Sigma$ is an anisotropic $\omega_0$-capillary hypersurface, we have $\nu_{F}^{\Sigma}(\partial\Sigma)=\W\cap\{x_{n+1}=-\omega_0\}$, then  $\nu_F^{\Sigma}(\Sigma)=\W\cap\{x_{n+1}\geq-\omega_0\}=\W_{1,\omega_0}-\omega_0E_{n+1}^F$.
	It means that, for $z\in \W_{1,\omega_0}$, there exists a unique point $\xi\in\Sigma$ such that $\nu_F^{\Sigma}(\xi)=x:=z-\omega_0E_{n+1}^F$, and $T_z\W_{1,\omega_0}=T_x\W=T_{\xi}\Sigma$ by the definition of anisotropic Gauss map $\nu_{F}^{\Sigma}$. Denote $\nu^{\Sigma}$ the  unit outward normal of $\Sigma$, then we have $\nu^{\Sigma}(\xi)=\nu(z)$. From \eqref{equ:pf-lemma7.1-1}, we have
	\begin{align*}
		\<\nu^{\Sigma}(\xi),E_{n+1}\>=\cos\theta, \quad \text{for}\ \xi\in \partial \Sigma,
	\end{align*}
	which means that $\Sigma$ is also a capillary hypersurface with constant contact angle $\theta$.
\end{proof}

Inspired by Mei, Wang, Weng, and Xia \cite{Xia-arxiv}*{Lemma 2.14}, we have the following lemma.
\begin{lemma}
	\label{lemma:7.2}
	The anisotropic capillary quermassintegrals $\mathcal{V}_{k+1,\omega_0}$ ($k=0,\cdots,n$) defined by \eqref{equ:v1} and \eqref{equ:Vk} have the following forms.
	\begin{itemize}
		\item[ (i)] We have the following expression for $\mathcal{V}_{k+1,\omega_0}(\Sigma)$: \begin{align}\label{equ:Vk=Vkkklll}
			\mathcal{V}_{k+1,\omega_0}(\Sigma)=\frac{1}{n+1}\int_{\Sigma}H_k^F\(1+\omega_0G(\nu_{{F}})(\nu_{{F}},E_{n+1}^F)\) {\rm d}\mu_F.
		\end{align}
		\item[(ii)]  Specifically,
		 if $\widehat{\W_{1,\omega_0}}\in \mathcal{K}_{\theta}$ for $\theta\in(0,\pi)$, then we have
		\begin{align*}
		\mathcal{V}_{k+1,\omega_0}(\Sigma)	=V\(\underbrace{\widehat{\Sigma}, \cdots, \widehat{\Sigma}}_{(n-k) \text { copies }}, \underbrace{\widehat{\W_{1,\omega_0}}, \cdots, \widehat{\W_{1,\omega_0}}}_{(k+1) \text { copies }}\),
		\end{align*}
		where the right hand side of the equation is the mixed volume which is written as \eqref{equ:VKKKLLL}.
	\end{itemize}
\end{lemma}
\begin{proof}
	(i) First, we consider the case $k = 0$. By applying integration by parts, we obtain
	\begin{align*}
		0=&\int_{\widehat{\Sigma}}\div (E_{n+1}^F) {\,\rm d}\mathcal{H}^{n+1}
		=\int_{\Sigma}\<E^F_{n+1},\nu\> {\,\rm d}\mu_g+\int_{\widehat{\partial\Sigma}}\<E^F_{n+1},-E_{n+1}\> {\,\rm d}\mathcal{H}^{n}
		\\
		=&\int_{\Sigma}\<E^F_{n+1},\nu\> {\,\rm d}\mu_g-|\widehat{\partial\Sigma}|,
	\end{align*}
	where $\mathcal{H}^{n+1}$ and $\mathcal{H}^n$ are the $(n+1)$-dimensional and $n$-dimensional Hausdorff measures, respectively.
	Then \eqref{equ:v1} can be written as
	\begin{align*}
		\mathcal{V}_{1,\omega_0}(\Sigma)=&\frac{1}{n+1}\left(\int_{\Sigma}F(\nu){\,\rm d}\mu_g +\omega_0\int_{\Sigma}\<E^F_{n+1},\nu\>
		 {\,\rm d}\mu_g\right)
		 \\
		 =&\frac{1}{n+1}\int_{\Sigma}\left(1+\omega_0G(\nu_{{F}})(\nu_{{F}},E_{n+1}^F)\right) {\rm d}\mu_F.
	\end{align*}
	
	Next, we consider the case $1\leq k\leq n$. Since ${\rm d}\mu_F=F(\nu){\rm d}\mu_g$ and $G(\nu_{F})(\nu_{F},E_{n+1}^F)=\frac{\<\nu,E_{n+1}^F\>}{F(\nu)}$, comparing with \eqref{equ:Vk}, we just need to check that
	\begin{align*}
		\int_{\Sigma}H^F_k\<\nu,E_{n+1}^F\>{\,\rm d}\mu_g=\frac{1}{n}\int_{\partial\Sigma}H_{k-1}^{\bar{F}}\bar{F}(\bar{\nu}){\,\rm d}s.
	\end{align*}
	Since $k\binom{n}{k}=n\binom{n-1}{k-1}$ and \eqref{equ:pfthm5-13},  we just need to prove the following \textbf{Claim}:
	\begin{align}
		\label{equ:pf-lemma7.2-0}
		\int_{\Sigma}k\sigma_k\<\nu,E_{n+1}^F\>{\,\rm d}\mu_g=\int_{\partial\Sigma}\<P_{k-1}\mu,\mu\>\bar{F}(\bar{\nu}){\,\rm d}s.
	\end{align}
	Now, let us prove the Claim above.
	Using \eqref{equ:lemma2.1.1}, we have
	\begin{align}
		\label{equ:pf-lemma7.2-1}
	\div \(P_{k-1}\left(\nabla^{\mathbb{S}} F\right) \circ \nu\)+F\left(\nu\right) \operatorname{tr}\(P_{k-1} \circ \mathrm{d} \nu\)=k \sigma_{k}.
	\end{align}
	Moreover, it's easy to  check that
	\begin{align}
		\label{equ:pf-lemma7.2-2}
		\div \(P_{k-1}(E^F_{n+1})^{\top}\)=-\<E^F_{n+1},\nu\>\tr \(P_{k-1}\circ {\rm d} \nu\),
	\end{align}
	and we can calculate that
	\begin{align}
		&\<\nabla\left(\<\nu,E^F_{n+1}\>\right),P_{k-1}(\nabla^{\mathbb{S}}F\circ\nu)\>=\<\mathrm{d}\nu(E^F_{n+1})^{\top},P_{k-1}(\nabla^{\mathbb{S}}F\circ\nu)\>\nonumber
		\\
		=&\<(E^F_{n+1})^{\top},\mathrm{d}\nu\circ P_{k-1}(\nabla^{\mathbb{S}}F\circ\nu)\>=\<\mathrm{d}\nu\circ P_{k-1}(E^F_{n+1})^{\top},\nabla^{\mathbb{S}}F\circ\nu\>\nonumber
		\\
		=&\< P_{k-1}(E^F_{n+1})^{\top},\mathrm{d}\nu(\nabla^{\mathbb{S}}F\circ\nu)\>=\< P_{k-1}(E^F_{n+1})^{\top},\nabla F(\nu)\>,	\label{equ:pf-lemma7.2-3}
	\end{align}
	where we use the symmetry of $\mathrm{d}\nu$ and $\mathrm{d}\nu\circ P_{k-1}$.
	
	Putting \eqref{equ:pf-lemma7.2-1}$\thicksim$\eqref{equ:pf-lemma7.2-3} into the left hand side of \eqref{equ:pf-lemma7.2-0}, by divergence Theorem,  we obtain
	\begin{align}
	(LHS):=&	\int_{\Sigma}k\sigma_k\<\nu,E_{n+1}^F\>{\,\rm d}\mu_g\nonumber
		\\
		\overset{\eqref{equ:pf-lemma7.2-1}}{=}&\int_{\Sigma}\(\div \left(P_{k-1}\left(\nabla^{\mathbb{S}} F\right) \circ \nu\right)+F\left(\nu\right) \operatorname{tr}\left(P_{k-1} \circ \mathrm{d} \nu\right)\)\<\nu,E^F_{n+1}\>
		{\,\rm d}\mu_g\nonumber
		\\
		\overset{\eqref{equ:pf-lemma7.2-2}}{=}&\int_{\Sigma}
		\<\nu,E^F_{n+1}\>\div \(P_{k-1}\left(\nabla^{\mathbb{S}} F\right) \circ \nu\)-F(\nu)\div\(P_{k-1}(E^F_{n+1})^{\top}\)
		 {\,\rm d}\mu_g\nonumber
		\\
		\overset{\eqref{equ:pf-lemma7.2-3}}{=}&\int_{\Sigma}
		\div\(P_{k-1}\left(\<\nu,E^F_{n+1}\>\nabla^{\mathbb{S}}F\circ \nu-F(\nu)(E_{n+1}^F)^{\top}\right)\)
		 {\,\rm d}\mu_g\nonumber
		\\
		=&\int_{\partial\Sigma}
		\(\<\nu,E_{n+1}^F\>\<\nu_{{F}},\mu\>-F(\nu)\<E_{n+1}^F,\mu\>\)\<P_{k-1}\mu,\mu\>
		 {\,\rm d}s,\label{equ:pf-lemma7.2-4}
	\end{align}
	where  in the last equality we use the fact that $\<P_{k-1}Y,\mu\>=0$ for any $Y\in T(\partial\Sigma)$
	 from \eqref{equ:SFau=0}.
	
	Along $\partial\Sigma$, using $\<\nu,\bar{\nu}\>=-\<\mu,E_{n+1}\>\neq 0$ and \eqref{equ:barF(v)}, we have
	\begin{align}
		\bar{F}(\bar{\nu})=\frac{F(\nu)+\omega_0\<E_{n+1}^F,\nu\>}{\<\nu,\bar{\nu}\>}=-\frac{\<\nu_{F}+\omega_0E_{n+1}^F,\nu\>}{\<\mu,E_{n+1}\>}.
		\label{equ:pf-lemma7.2-5}
	\end{align}
	Since $\Sigma$ is an anisotropic $\omega_0$-capillary hypersurface and $\<E_{n+1}^F,E_{n+1}\>=1$, we have
	\begin{align}
		0=&\<\nu_{F}+\omega_0E_{n+1}^F,E_{n+1}\>\nonumber
		\\=&\<\nu_{F}+\omega_0E_{n+1}^F,\nu\>\<\nu,E_{n+1}\>+\<\nu_{F}+\omega_0E_{n+1}^F,\mu\>\<\mu,E_{n+1}\>.
		\label{equ:pf-lemma7.2-6}
	\end{align}

Putting \eqref{equ:pf-lemma7.2-6} into \eqref{equ:pf-lemma7.2-5} we have
\begin{align}
	\bar{F}(\bar{\nu})=\frac{\omega_0\<E_{n+1}^F,\mu\>+\<\nu_{F},\mu\>}{\<E_{n+1},\nu\>}.
	\label{equ:pf-lemma7.2-7}
\end{align}
From \cite{Jia-Wang-Xia-Zhang2023}*{eq.(3.6)(3.8)} we have,
\begin{align*}
		-\omega_0  =&\left\langle E_{n+1}, \nu_F\right\rangle=\left\langle \nu, E_{n+1}\right\rangle F(\nu)+\left\langle\mu, E_{n+1}\right\rangle\left\langle \nu_F, \mu\right\rangle,
\\
1=&\left\langle E_{n+1}^F, E_{n+1}\right\rangle=\left\langle \nu, E_{n+1}\right\rangle\left\langle E_{n+1}^F, \nu\right\rangle+\left\langle\mu, E_{n+1}\right\rangle\left\langle E_{n+1}^F, \mu\right\rangle .
\end{align*}
Putting above equations into \eqref{equ:pf-lemma7.2-7} yields
\begin{align}
	\bar{F}(\bar{\nu})=&-\frac{\<E_{n+1}^F,\mu\>}{\<E_{n+1},\nu\>}\(\left\langle \nu, E_{n+1}\right\rangle F(\nu)+\left\langle\mu, E_{n+1}\right\rangle\left\langle \nu_F, \mu\right\rangle\)\nonumber
	\\
	&+\frac{\<\nu_{F},\mu\>}{\<E_{n+1},\nu\>}\(\left\langle \nu, E_{n+1}\right\rangle\left\langle E_{n+1}^F, \nu\right\rangle+\left\langle\mu, E_{n+1}\right\rangle\left\langle E_{n+1}^F, \mu\right\rangle\)\nonumber
	\\
	=&\<\nu,E_{n+1}^F\>\<\nu_{{F}},\mu\>-F(\nu)\<E_{n+1}^F,\mu\>.
	\label{equ:pf-lemma7.2-8}
\end{align}
Combination of  \eqref{equ:pf-lemma7.2-4} and \eqref{equ:pf-lemma7.2-8} yields the desired equality \eqref{equ:pf-lemma7.2-0}.
\\

(ii) If $\widehat{\W_{1,\omega_0}}\in \mathcal{K}_{\theta}$ for $\theta\in(0,\pi)$, then we have $\widehat{\Sigma}\in \mathcal{K}_{\theta}$ by Lemma \ref{lemma:7.1} (ii).
Using Proposition \ref{prop7.3} and \eqref{equ:Q(AB)=Q(A)det(B)}, we have
\begin{align}
	&V\(\underbrace{\widehat{\Sigma}, \cdots, \widehat{\Sigma}}_{(n-k) \text { copies }}, \underbrace{\widehat{\W_{1,\omega_0}}, \cdots, \widehat{\W_{1,\omega_0}}}_{(k+1) \text { copies }}\)\nonumber
	\\
	=&\frac{1}{n+1}\int_{\mathcal{C}_{\theta}}\ell(\xi)Q(\underbrace{A[h(\xi)],\cdots,A[h(\xi)]}_{(n-k) \text{ copies }},\underbrace{A[\ell(\xi)],\cdots,A[\ell(\xi)]}_{k \text{ copies }}){\,\rm d}\mu_{\sigma}\nonumber
	\\
	\overset{\eqref{equ:Q(AB)=Q(A)det(B)}}{=}&\frac{1}{n+1}\int_{\mathcal{C}_{\theta}}\ell(\xi)\det(A[\ell(\xi)])Q(\underbrace{A[h(\xi)](A[\ell(\xi)])^{-1},\cdots,A[h(\xi)](A[\ell(\xi)])^{-1}}_{(n-k) \text{ copies }},\underbrace{I_n,\cdots,I_n}_{k \text{ copies }}){\,\rm d}\mu_{\sigma}\nonumber
	\\
	=&\frac{1}{n+1}\int_{\mathcal{C}_{\theta}}\ell(\xi)\det(A[\ell(\xi)])\binom{n}{k}^{-1}\sigma_{n-k}(\lambda)
	{\,\rm d}\mu_{\sigma},\label{equ:pf-lemma7.2-12}
\end{align}
where $\lambda=\{\lambda_1,\cdots ,\lambda_n\}$ are the eigenvalues of  ${A[h(\xi)](A[\ell(\xi)])^{-1}}$,  
$\ell(\xi)$
and $h(\xi)$ are capillary support functions of $\W_{1,\omega_0}$ and $\Sigma$ respectively,  $I_n$ is an $(n\times n)$ identity matrix, and $A[l(\xi)]>0$ by \cite{Xia-arxiv}*{Lemma 2.4 (3)}.

 Since $\W_{1,\omega_0}=\left(\W+\omega_0E_{n+1}^F\right)\cap\overline{\mathbb{R}_+^{n+1}}$, the support function of $\W_{1,\omega_0}$ is $F(z)+\omega_0\<z,E_{n+1}^F\>$ for $z\in \mathbb{S}_{\theta}=\mathcal{C}_{\theta}+\cos\theta E_{n+1}$.
For any $\xi\in\mathcal{C}_{\theta}$, we denote $z(\xi)=\xi+\cos\theta E_{n+1}$, and reparametrize the position vector of  $\Sigma$ as $X(\xi)=\nu^{-1}(z(\xi))$, we have $T_{\xi}\mathcal{C}_{\theta}=T_{z(\xi)}\mathbb{S}_{\theta}=T_{X(\xi)}\Sigma=T_{\nu_{F}(X(\xi))}\W=T_{\nu_{F}(X(\xi))+\omega_0E_{n+1}^F}\W_{1,\omega_0}$. Then  it follows that  $$h(\xi)=\<X(\xi),z(\xi)\>=u\(\nu^{-1}(z(\xi))\)=u\(\nu^{-1}(\xi+\cos\theta E_{n+1})\),$$
 and \begin{align}\label{equ:pf-lemma7.2-9}
 	\ell(\xi)=F(z(\xi))+\omega_0\<z(\xi),E_{n+1}^F\>,
 \end{align} since  the capillary support function is obtained by reparameterizing the support function on  $\mathcal{C}_{\theta}$ (see \cite{Xia-arxiv}*{Page 6}). 

We denote by  $B(\xi),\bar{B}(\xi)$  the second fundamental form of point $X(\xi)\in(\Sigma,g)$ and point  $\nu_{F}(X(\xi))+\omega_0E_{n+1}^F\in(\W_{1,\omega_0},\bar{g})$, where $g$ and $\bar{g}$ are induced metric of $\Sigma\subset\mathbb{R}^{n+1}$ and $\W_{1,\omega_0}\subset\mathbb{R}^{n+1}$, respectively. Then $\bar{B}(\xi)$ is also the second fundamental form of point $\nu_{F}(X(\xi))\in\W$, which implies that the  anisotropic principal curvatures  $\kappa_i^F (i=1,\cdots,n)$ are the eigenvalues of $(\bar{g}^{-1}\bar{B})^{-1}g^{-1}B$.
By the proof of \cite{Xia-arxiv}*{Lemma 2.4 (3)}, we have
\begin{align*}
   B^{-1}g=A[h]\quad\text{and}\quad   \bar{B}^{-1}\bar{g}=A[\ell],
\end{align*}
then $\kappa_i^F (i=1,\cdots,n)$ are the eigenvalues of $A[{\ell}](A[h])^{-1}$, that is $\lambda_i=\frac{1}{\kappa_i^F}\, (i=1,\cdots,n)$, 
which implies
\begin{align}
	\det(A[\ell])\sigma_{n-k}(\lambda)=&\det(A[\ell])\frac{\sigma_k(\kappa^F)}{\sigma_n(\kappa^F)}=\det(A[\ell])\frac{\sigma_k(\kappa^F)}{\det(A[\ell])\det(A[h]^{-1})}\nonumber
	\\
	=&\det(A[h])\sigma_k(\kappa^F).
	\label{equ:pf-lemma7.2-10}
\end{align}
Since $\mathcal{C}_{\theta}=\mathbb{S}_{\theta}-\cos\theta E_{n+1}=\nu(\Sigma)-\cos\theta E_{n+1}$ (\cite{Xia-arxiv}*{Lemma 2.2}), we have
\begin{align}
	{\rm d}\mu_{\sigma}=\det(g^{-1}B){\rm d}\mu_{g}=\frac{1}{\det(A[h])}{\rm d}\mu_{g}.
	\label{equ:pf-lemma7.2-11}
\end{align}
Putting $\nu(X)=\xi+\cos\theta E_{n+1}=z(\xi)$ and  \eqref{equ:pf-lemma7.2-9}$\thicksim$\eqref{equ:pf-lemma7.2-11} into \eqref{equ:pf-lemma7.2-12} yields
\begin{align*}
	V\(\underbrace{\widehat{\Sigma}, \cdots, \widehat{\Sigma}}_{(n-k) \text { copies }}, \underbrace{\widehat{\W_{1,\omega_0}}, \cdots, \widehat{\W_{1,\omega_0}}}_{(k+1) \text { copies }}\)=&\frac{1}{n+1}\int_{\Sigma}H_k^F\(F(\nu)+\<\nu,E_{n+1}^F\>\){\rm d}\mu_g
	\\
	\overset{\eqref{equ:Vk=Vkkklll}}{=}&\mathcal{V}_{k+1,\omega_0}(\Sigma).
\end{align*}
This completes the proof.
\end{proof}
By combining Lemma \ref{lemma:7.1} (i), Lemma \ref{lemma:7.2} (ii), and Theorems \ref{thm:iso-neq}, \ref{thm:AF-neq}, we can prove  Corollary \ref{cor:7.3} as follows.
\begin{proof}[\textbf{Proof of Corollary \ref{cor:7.3}}]
	(i) For $L\in\mathcal{K}_{\theta}$, it's obvious that there exists a convex body $\bar{L}\subset\mathbb{R}^{n+1}$, such that $L=\bar{L}\cap\overline{\mathbb{R}_{+}^{n+1}}$. Let ${\Sigma}=\partial K\cap\mathbb{R}_+^{n+1}$ and $\W=\partial \bar{L}$, where $\W$ determines a support function $F$, then $\partial L\cap\mathbb{R}_+^{n+1}$ is a $\omega_0$-capillary Wulff shape $\W_{1,\omega_0}$ with $\omega_0=0$.
	
	From \cite{Xia-arxiv}*{Lemma 2.2} and $L\in\mathcal{K}_{\theta}$,  the Gauss maps of  $\W_{1,0}$ is a diffeomorphism 
	with its image  $\mathbb{S}_{\theta}$. Since $\Sigma$ is a $\theta$-capillary hypersurface, the Gauss map of $\Sigma$ maps onto $\mathbb{S}_{\theta}$. By composing the Gauss map of $\Sigma$  with the inverse Gauss map of $\W_{1,0}$, we  obtain an anisotropic Gauss map  of $\Sigma$, and at the corresponding point, $\Sigma$ and $\W_{1,0}$ have the same anisotropic unit outward normal, so $\Sigma$ is also an anisotropic $\omega_0$-capillary hypersurface with $\omega_0=0$.
	
	By Lemma \ref{lemma:7.2} (i) and Theorem \ref{thm:iso-neq}, we have
	\begin{align*}
	\Vol(K)^{\frac{n}{n+1}}\Vol(L)^{\frac{1}{n+1}}\leq
		&\frac{1}{n+1}\int_{\Sigma}\left(1+\omega_0G(\nu_{{F}})(\nu_{{F}},E_{n+1}^F)\right) {\rm d}\mu_F\\
		=& \frac{1}{n+1}\int_{\Sigma}\left(F(\nu)+\omega_0\<\nu,E_{n+1}^F\>\right) {\rm d}\mu_g\\
		:=&\frac{1}{n+1}\int_{\Sigma}\ell(\nu(X)) {\rm d}\mu_g
		,
	\end{align*}
	where $\ell(z)=F(z)+\omega_0\<z,E_{n+1}^F\>$ ($z\in\mathbb{S}^n$) is the  support function of $\mathcal{C}_{\theta}$,  and the equality holds if and only if $\Sigma$ is an $\omega_0$-capillary Wulff shape, i.e., there exists a constant $c\in\mathbb{R}^+$ and a constant vector $b\in\partial\overline{\mathbb{R}^{n+1}_+}$ such that  $K=cL+b$. 
	
(ii) Similar to (i), for $K,L\in\mathcal{K}_{\theta}$, we can take a Wulff shape $\W$ and an $\omega_0$-capillary hypersurface $\Sigma$ such that $K=\widehat{\Sigma}$ and $L=\widehat{\W_{1,\omega_0}}$ with $\omega_0=0$.
By Lemma \ref{lemma:7.1} (i) and $\theta\in(0,\frac{\pi}{2}]$, we know $\W$ and $\omega_0$ satisfy Condition \ref{condition}.

 From Theorem \ref{thm:AF-neq}, we have
 \begin{align*}
 	\mathcal{V}_{k+1,\omega_0}(\Sigma)\geq\Vol(K)^{\frac{n-k}{n+1}}\Vol(L)^{1-\frac{n-k}{n+1}},\quad k=0,\cdots,n-2.
 \end{align*}
 Combining with Lemma \ref{lemma:7.2} (ii), we derive \eqref{equ:cor-7.3}.
 The equality holds if and only if $\Sigma$ is an $\omega_0$-capillary Wulff shape with $\omega_0=0$, i.e., there exists a constant $c\in\mathbb{R}^+$ and a constant vector $b\in\partial\overline{\mathbb{R}^{n+1}_+}$ such that  $K=cL+b$. 

 Then we complete the proof.
\end{proof}

{\appendix \section{Translated Wulff shapes and examples}\label{Appendix}
	
	\setcounter{equation}{0}
	\renewcommand{\theequation}{A.\arabic{equation}}
For a constant vector $\eta\in\mathbb{R}^{n+1}$, we  take $\mathcal{T}_{\eta}:\mathbb{R}^{n+1}\rightarrow\mathbb{R}^{n+1}$ as a  translational transformation defined by $\mathcal{T}_{\eta}(z)=z+\eta$. Given a Wulff shape $\W$ with respect to $F$,  containing the origin $O$, we denote a translated Wulff shape $\widetilde{\W}=\mathcal{T}_{\eta}(\W)$ as a new Wulff shape with respect to $\tilde{F}$.  ${F}^0\in C^{\infty}(\mathbb{R}^{n+1})$ is the dual Minkowski norm of $F$, and $\tilde{F}^0\in C^{\infty}(\mathbb{R}^{n+1})$ is the dual Minkowski norm of $\tilde{F}$.
$G$ and $\tilde{G}$ are new metrics defined by \eqref{equ:G} with respect to $F^0$ and $\tilde{F}^0$, respectively. $Q$ and $\tilde{Q}$ are $(0,3)$-tensors defined by \eqref{equ:Q} with respect to $F^0$ and $\tilde{F}^0$, respectively.
We have following proposition (the special case that will be used frequently is written as Lemma \ref{lem:Q-translate}):

\begin{proposition}\label{prop:Q-Q}
If $\eta$ and $\W$ satisfy $1+G(z)(z,\eta)> 0$ for all	 $z\in \W$, then
for any $z\in\W$  and
$X,Y,Z\in T_z\W=T_{\tilde{z}}\widetilde{\W}$, where $\tilde{z}:=\mathcal{T}_{\eta}(z)=z+\eta\in \widetilde{\W}$, we have
\begin{align}
	&\tilde{G}(\tilde{z})(X,Y)=\frac{1}{1+G(z)(z,\eta)}G(z)(X,Y),\label{equ:A.G}
	\\
	&\tilde{Q}(\tilde{z})(X,Y,Z)=\frac{Q(z)(X,Y,Z)}{1+G(z)(z,\eta)}\nonumber
	\\&-
	\frac{G(z)(Z,Y)G(z)(X,\eta)+G(z)(X,Y)G(z)(Z,\eta)+G(z)(X,Z)G(z)(Y,\eta)}{\left(1+G(z)(z,\eta)\right)^2}.\label{equ:A.Qaan}
\end{align}
\end{proposition}
\begin{proof}
For any $x\in\mathbb{R}^{n+1}$, we know $\frac{x}{\tilde{F}^0(x)}\in\widetilde{\W}$ and $\mathcal{T}^{-1}\left(\frac{x}{\tilde{F}^0(x)}\right)=\frac{x}{\tilde{F}^0(x)}-\eta\in\W$. Then $F^0\left(\frac{x}{\tilde{F}^0(x)} -\eta \right)=1$, by the $1$-homogeneity of $F^0$, we have
\begin{align}\label{equ:F0=tildeF0}
	F^0\left( x-\tilde{F}^0(x)\eta \right) = \tilde{F}^0(x), \quad \forall x\in\mathbb{R}^{n+1}.
\end{align}

 We write $x=(x_1,\cdots,x_{n+1}),\eta=(\eta_1,\cdots,\eta_{n+1}), X=(X_1,\cdots,X_{n+1}), Y=(Y_1,\cdots,Y_{n+1}),$ and $Z=(Z_1,\cdots,Z_{n+1})$  in $\mathbb{R}^{n+1}$ with a fixed Cartesian coordinate. And we denote $i,j,k,l,p,q,m\in\{1,\cdots,n+1\},$ and $\tilde{F}^0_i(x)=\frac{\partial \tilde{F}^0}{\partial x_i}(x),\tilde{F}^0_{ij}(x)=\frac{\partial^2 \tilde{F}^0}{\partial x_i\partial x_j}(x), \tilde{F}^0_{ijk}(x)=\frac{\partial^3 \tilde{F}^0}{\partial x_i\partial x_j\partial x_k}(x),
 F^0_i({\xi})=\frac{\partial F^0}{\partial {\xi}_i}({\xi}),F^0_{ij}({\xi})=\frac{\partial^2 F^0}{\partial {\xi}_i\partial {\xi}_j}({\xi}), F^0_{ijk}({\xi})=\frac{\partial^3 F^0}{\partial {\xi}_i\partial {\xi}_j\partial {\xi}_k}({\xi})$ for any $\xi=(\xi_1,\cdots,\xi_{n+1})$.
From \eqref{equ:F0=tildeF0} and $\frac{\partial(x_l-\tilde{F}^0(x)\eta_l)}{\partial x_i}=\delta_{il}-\eta_l\tilde{F}^0_i(x)$, we have
\begin{align*}
\left(\delta_{il}-\eta_l\tilde{F}^0_i(x)\right)\cdot F^0_{l}(x-\tilde{F}^0(x)\eta)=\tilde{F}^0_i(x),
\end{align*}
which implies
\begin{align}\label{equ:A.F_i}
	\tilde{F}^0_i(x)=\frac{F^0_i(x-\tilde{F}^0(x)\eta)}{1+\<\eta,DF^0(x-\tilde{F}^0(x)\eta)\>}.
\end{align}
Taking the partial derivative with $x_j$ in \eqref{equ:A.F_i}, we  have
\begin{align}
	\tilde{F}^0_{ij}(x)=\frac{F^0_{ij}}{1+\<\eta,DF^0\>}-\frac{\eta_pF^0_{ip}F^0_j}{(1+\<\eta,DF^0\>)^2}-\frac{\eta_pF^0_{jp}F^0_i}{(1+\<\eta,DF^0\>)^2}+\frac{\eta_p\eta_qF^0_{pq}F^0_iF^0_j}{(1+\<\eta,DF^0\>)^3}.\label{equ:A.F_ij}
\end{align}
Here we denote $F^0_i=F^0_i(x-\tilde{F}^0(x)\eta),F^0_{ij}=F^0_{ij}(x-\tilde{F}^0(x)\eta),F^0_{ijk}=F^0_{ijk}(x-\tilde{F}^0(x)\eta)$, and $DF^0=DF^0(x-\tilde{F}^0(x)\eta)=(F^0_1,\cdots,F^0_{n+1})$.
Taking the partial derivative of $x_k$ in \eqref{equ:A.F_ij}, we  have
\begin{align}
\tilde{F}	&^0_{ijk}(x)=\(1+\<\eta,DF^0\>\)^{-1}
F^0_{ijk}\nonumber
	\\
&	- \(1+\<\eta,DF^0\>\)^{-2}\eta_l\(F^0_{ijl}F^0_k +F^0_{ikl}F^0_j+F^0_{kjl}F^0_i+ F^0_{il}F^0_{jk}+F^0_{jl}F^0_{ik}+F^0_{kl}F^0_{ij}
	\)\nonumber
	\\
&	+\(1+\<\eta,DF^0\>\)^{-3}\eta_p\eta_q\(F^0_{ipq}F^0_kF^0_j+F^0_{jpq}F^0_kF^0_i+F^0_{kpq}F^0_iF^0_j\nonumber
\\
& +F^0_{pq}(F^0_{ik}F^0_j+F^0_{jk}F^0_i+F^0_{ij}F^0_k)
+2F^0_{ip}F^0_{jq}F^0_k +2F^0_{jp}F^0_{kq}F^0_i +2F^0_{kp}F^0_{iq}F^0_j
\)\nonumber
\\
&-\(1+\<\eta,DF^0\>\)^{-4}\eta_p\eta_q\eta_l\(F^0_{pql}F^0_iF^0_jF^0_k +3F^0_{pq}(F^0_{il}F^0_jF^0_k+F^0_{jl}F^0_kF^0_i+F^0_{kl}F^0_iF^0_j )
\)\nonumber
\\
&+3\(1+\<\eta,DF^0\>\)^{-5}\eta_p\eta_q\eta_l\eta_mF^0_kF^0_iF^0_jF^0_{pq}F^0_{lm}.\label{equ:A.F_ijk}
\end{align}
Next, we take  $x=\tilde{z}\in\widetilde{\W}$, then $x-\tilde{F}^0(x)\eta=\tilde{z}-\eta=z\in\W$. We know that $DF^0(z)$ is  perpendicular to $T_z\W$ with respect to standard Euclidean metric, since $\W=\{x\in\mathbb{R}^{n+1}:F^0(x)=1\}$. Similarly, $D\tilde{F}^0(\tilde{z})$ is  perpendicular to $T_{\tilde{z}}\widetilde{\W}$ with respect to standard Euclidean metric. Then combining with \eqref{equ:A.F_ij}, we have \begin{align*}
	\tilde{G}(\tilde{z})(X,Y)=&\tilde{F}^0(\tilde{z})\tilde{F}^0_{ij}(z)X_iY_j+\tilde{F}^0_i(z)\tilde{F}^0_j(z)X_iY_j=\tilde{F}^0_{ij}(z)X_iY_j
	\\
	=&\frac{F^0_{ij}(z)X_iY_j}{1+\<\eta,DF^0(z)\>}-\frac{\eta_pF^0_{ip}(z)F^0_j(z)X_iY_j}{(1+\<\eta,DF^0(z)\>)^2}-\frac{\eta_pF^0_{jp}(z)F^0_i(z)X_iY_j}{(1+\<\eta,DF^0(z)\>)^2}\nonumber
	\\
	&+\frac{\eta_p\eta_qF^0_{pq}(z)F^0_i(z)F^0_j(z)X_iY_j}{(1+\<\eta,DF^0(z)\>)^3}
	\\
	=&\frac{F^0_{ij}(z)X_iY_j}{1+\<\eta,DF^0(z)\>}
	\\
	=&\frac{F^0(z)F^0_{ij}(z)X_iY_j+F^0_i(z)F^0_j(z)X_iY_j}{1+\<\eta,DF^0(z)\>}
	\\
	=&\frac{G(z)(X,Y)}{1+\<\eta,DF^0(z)\>},
\end{align*}
for $X,Y\in T_z\W=T_{\tilde{z}}\widetilde{\W}$.
Similarly, combining with \eqref{equ:A.F_ijk}, we can calculate
\begin{align*}
	&\tilde{Q}(\tilde{z})(X,Y,Z)=\tilde{F}^0_{ijk}(\tilde{z})X_iY_jZ_k
	\\
	=&\(1+\<\eta,DF^0(z)\>\)^{-1}
	F^0_{ijk}(z)X_iY_jZ_k
		\\
		&- \(1+\<\eta,DF^0(z)\>\)^{-2}\( F^0_{il}(z)F^0_{jk}(z)+F^0_{jl}(z)F^0_{ik}(z)+F^0_{kl}(z)F^0_{ij}(z)\)\eta_lX_iY_jZ_k
	\\
	=&\(1+\<\eta,DF^0(z)\>\)^{-1}Q(z)(X,Y,Z)-\(1+\<\eta,DF^0(z)\>\)^{-2}\(G(z)(Z,Y)G(z)(X,\eta)
\\
	&+G(z)(X,Y)G(z)(Z,\eta)+G(z)(X,Z)G(z)(Y,\eta)
	\),
\end{align*}
 for $X,Y,Z\in T_z\W=T_{\tilde{z}}\widetilde{\W}$.
 Since $\sum_{i}{F}^0_{ij}(z)z_i=0$  and $\sum_{j} F^0_j(z)z_j=F^0(z)=1$, we have $G(z)(z,\eta)=\<DF^0(z),\eta\>$. Then we complete the proof.
\end{proof}
\begin{remark}
	Given  a Wulff shape $\W$  enclosing the origin $O$, a constant vector $\eta\in\mathbb{R}^{n+1}$, the following three statements are equivalent:
	\begin{itemize}
		\item[(i)] The translated Wulff shape $\widetilde{\W}=\mathcal{T}_{\eta}\W$ contains the origin $O$ in its interior;
		\item[(ii)] $F^0(-\eta)<1$;
		\item[(iii)] $1+G(z)(z,\eta)>0$ for all $z\in\W$.
	\end{itemize}
\end{remark}
\begin{proof}
	If $\eta=0$, it is trivial. We just  check the case  $\eta\neq 0$.
	
	(i)$\Leftrightarrow$(ii). This conclusion follows directly by $\widetilde{\W}=\{x\in\mathbb{R}^{n+1}:F^0(x-\eta)=1\}$. 
	
	(ii)$\Rightarrow$(iii). For any $z\in\W$, we have
	\begin{align*}
		1+G(z)(z,\eta)=&1+\<D^2F^0(z)z,\eta\>+\<DF^0(z),z\>\<DF^0(z),\eta\>
		\\=&1-\<DF^0(z),-\eta\>
		\\
		\geq&1-F(DF^0(z))F^0(-\eta)\\
		=&1-F^0(-\eta)\overset{\text{(ii)}}{>}0,
	\end{align*}
	where we use the following Cauchy-Schwarz inequality (see e.g.,
\cite{He-Li_Ma-Ge-09}*{Prop 2.4(2)} or  \cite{Jia-Wang-Xia-Zhang2023}*{Proposition 2.1(iv)}):	$$\<x,z\>\leq F^0(x)F(z).$$

(iii)$\Rightarrow$(ii).  Taking $z=\frac{-\eta}{F^0(-\eta)}\in\W$ into (iii), we have
\begin{align*}
	0<1+G(z)(z,\eta)=1+G(z)(z,-F^0(-\eta)z)=1-F^0(-\eta).
\end{align*}
That means $F^0(-\eta)<1$.
\end{proof}

At last, we give several examples of Wulff shapes satisfying Condition \ref{condition}.

\begin{Examples}\label{exam:elliptic}
	We take an ellipse $\W=\{(x,y,z):\frac{x^2}{A}+\frac{y^2}{B}+\frac{z^2}{C}=1\}\subset\mathbb{R}^3$, with constants $A,B,C>0$. Assume $\eta=\omega_0E_{n+1}^F=\omega_0E_{n+1}=(0,0,\omega_0)$ with $\omega_0\in(-C^{1/2},C^{1/2})$, then the translated Wulff shape $\mathcal{T}_{\eta}\W$ is $\widetilde{\W}=\{(x,y,z):\frac{x^2}{A}+\frac{y^2}{B}+\frac{(z-\omega_0)^2}{C}=1\}$.
	
	It's easy to check that $F^0=\left(\frac{x^2}{A}+\frac{y^2}{B}+\frac{z^2}{C}\right)^{\frac{1}{2}}$ and $Q_{ijk}(\xi)\equiv0$ for $\xi=(x,y,z)\in\W$ with $ i,j,k=1,2,3.$ Then $Q(\xi)(e_1,e_1,e_2)=0$ for $\xi=(x,y,-\omega_0)\in\W,{e}_{1}(\xi),{e}_{2}(\xi)\in T_\xi\W, e_1\in T_\xi\partial(\W\cap\{z=-\omega_0\})$ and $G(\xi)(e_1,e_2)=0$.
	
	For  $\tilde{\xi}=\mathcal{T}_{\eta}\xi=(x,y,0)\in\widetilde{\W}\cap \{z=0\}$, we can take $e_1=(\frac{y}{B},-\frac{x}{A},0),e_2=\left(\omega_0x,\omega_0y,C(1-
	\frac{\omega_0^2}{C})\right)$, it's easy to check that $\<e_1,\nu\>=0,\ \<e_2,\nu\>=0,\ G(\xi)(e_1,e_2)=0$, where $\nu=\frac{(\frac{x}{A},\frac{y}{B},\frac{-\omega_0}{C})}{((\frac{x}{A})^2+(\frac{y}{B})^2+(\frac{-\omega_0}{C})^2)^{1/2}}$.
 Then	we have $G(\xi)(e_1,e_1)=\frac{1}{AB}\left(1-\frac{\omega_0^2}{C}\right)>0,G(\xi)(e_2,\eta)=\omega_0(1-
 \frac{\omega_0^2}{C}),1+G(\xi)(\xi,\eta)=1+\frac{-\omega_0^2}{C}>0$, then from \eqref{equ:A.Qaan} we have $\tilde{Q}(\tilde{\xi})(e_1,e_1,e_2)=-\frac{\omega_0}{AB}$. If $\omega_0\neq 0$, we have $\tilde{Q}(\tilde{\xi})(e_1,e_1,e_2)\neq0$.
\end{Examples}

The ellipsoid in Example \ref{exam:elliptic} is so special. Let's calculate another example. 
\begin{Examples}\label{exam:x^4}
 	Given a Minkowski norm in $\mathbb{R}^3$:
 \begin{align*}
 	F^0=\((x^2+y^2+z^2)(x^2+ y^2)+z^4\)^{\frac{1}{4}},
 \end{align*}
 it uniquely determines  a Wulff shape $\W=\{(x,y,z)\in\mathbb{R}^3:F^0(x,y,z)=1\}$. 

 (1) We take a constant $z_0\in(-1,1)$, a point $\xi\in\W\cap\{z=z_0\}$, vectors $e_1,e_2\in T_{\xi}\W$,  $e_1\in T_{\xi}\partial(\W\cap\{z=z_0\}),$  $ G(\xi)(e_1,e_2)=0$, and $e_2$ points outward of $ ({\W}\cap\{z\geq z_0\})\subset{\W}$. Then we can calculate that $Q(\xi)(e_1,e_1,e_2)=0$ for $z_0=0$; $Q(\xi)(e_1,e_1,e_2)<0$ for $z_0<0$; $Q(\xi)(e_1,e_1,e_2)>0$ for $z_0>0$.

 (2) Taking a constant $\omega_0\in (-1,1)$, for this $\W$, Condition \ref{condition} is equivalent to  $\omega_0\leq 0$.
\end{Examples}
 \begin{proof}
 	First, we check that the  $F^0(\xi)$  is a Minkowski norm of $\mathbb{R}^3$. It is easy to see that $F^0$ is a $1$-homogeneous function which satisfies $F^0(\xi)>0$ for $\xi\neq 0$.
 	We denote $G(\xi)=D^2\left(\frac{(F^0)^2}{2}\right)(\xi)$, for $\xi=(x,y,z)\in \mathbb{R}^3$. We calculate the first, second, and third order leading principal minors of the $(3\times 3)$ matrix $ G(\xi)$ as follows:
 	\begin{align*}
 		D_1=&\frac{\partial^2 \frac12(F^0)^2(\xi)}{\partial x^2}
 		\\
 		=&
 		\frac{2 x^6 + 6 x^4 y^2 + 3 x^4 z^2 + 6 x^2 y^4 + 6 x^2 y^2 z^2 + 6 x^2 z^4 + 2 y^6 + 3 y^4 z^2 + 3 y^2 z^4 + z^6}{2\left((x^2 + y^2) (x^2 + y^2 + z^2) + z^4\right)^{\frac{3}{2}}}
 		\\
 		\geq&
 		 \frac{2x^6+ 2y^6 + z^6}{2\left((x^2 + y^2) (x^2 + y^2 + z^2) + z^4\right)}
 		\\ 		
 		>&0,\quad \text{for}\ \xi\neq (0,0,0),
 \\
 		D_2=&\frac{\partial^2 \frac12(F^0)^2(\xi)}{\partial x^2}\frac{\partial^2 \frac12(F^0)^2(\xi)}{\partial y^2}-\left(\frac{\partial^2 \frac12(F^0)^2(\xi)}{\partial x\partial y}\right)^2
 		\\
 		=&\frac{1}{4}\left(4 x^8 + 16 x^6 y^2 + 8 x^6 z^2 + 24 x^4 y^4 + 24 x^4 y^2 z^2 + 15 x^4 z^4 + 16 x^2 y^6 + 24 x^2 y^4 z^2\right.
 		\\
 		&\left.  + 30 x^2 y^2 z^4 + 8 x^2 z^6 + 4 y^8 + 8 y^6 z^2 + 15 y^4 z^4 + 8 y^2 z^6 + z^8\right)
 		\\
 		&\cdot \left(x^4 + 2 x^2 y^2 + x^2 z^2 + y^4 + y^2 z^2 + z^4\right)^{-2}
 		\\
 		\geq &\frac{1}{4}\left(4 x^8 + 4y^8 +  z^8\right)\cdot \left(x^4 + 2 x^2 y^2 + x^2 z^2 + y^4 + y^2 z^2 + z^4\right)^{-2}
 		\\
 		>&0,\quad \text{for}\ \xi\neq (0,0,0),
\\
 		D_3=&\det(G(\xi))
 		\\
 		=&\frac{1}{8}\left(4 x^{10} + 20 x^8 y^2 + 28 x^8 z^2 + 40 x^6 y^4 + 112 x^6 y^2 z^2 + 43 x^6 z^4 + 40 x^4 y^6 + 168 x^4 y^4 z^2 \right.
 		\\
 		&+ 129 x^4 y^2 z^4 + 41 x^4 z^6 + 20 x^2 y^8 + 112 x^2 y^6 z^2 + 129 x^2 y^4 z^4 + 82 x^2 y^2 z^6 + 17 x^2 z^8
 		\\
 		&\left.  + 4 y^{10} + 28 y^8 z^2 + 43 y^6 z^4 + 41 y^4 z^6 + 17 y^2 z^8 + 2 z^{10}\right)
 		\\
 		&\cdot\left(x^4 + 2 x^2 y^2 + x^2 z^2 + y^4 + y^2 z^2 + z^4\right)^{-\frac{5}{2}}
 		\\
 		\geq&
 		\frac{2x^{10}+ 2y^{10}+ z^{10}}{4\left(x^4 + 2 x^2 y^2 + x^2 z^2 + y^4 + y^2 z^2 + z^4\right)^{\frac{5}{2}}}
 		\\
 		>&0,\quad \text{for}\ \xi\neq (0,0,0),
  \end{align*}
  then $F^0$ satisfies the uniformly elliptic condition:  $D^2(F^0)^2$ is positive definite in $\mathbb{R}^3\setminus\{0\}$. By using \cite{Xia-phd}*{Proposition 1.4} and  \cite{Xia-phd}*{Proposition 1.2(ii)}, we know $D^2F^0$  is a positive semi-definite matrix in $\mathbb{R}^3$, that means $F^0$ is a convex function. Thus $F^0$ is a Minkowski norm. Then the dual norm $F$ is also a Minkowski norm (see \cite{Xia2017}*{Page 110} or \cite{Shen-Lecture-2001}*{Lemma 3.12}), and convex hypersurface  $\W=\{\xi\in\mathbb{R}^3:F^0(\xi)=1\}$ is the Wulff shape with respect to $F$ (see \cite{Xia-phd}*{Page 102-103}).

  (1) Next, we take  vector fields $e_1,e_2\in T_{\xi}\W$ such that  $e_1\in T_{\xi}\partial(\W\cap\{z=z_0\})$ and $ G(\xi)(e_1,e_2)=0$ if  $\xi\in\W\cap\{z=z_0\}$.
   Set $\nu(\xi)=(\nu_1,\nu_2,\nu_3)=\frac{DF^0(\xi)}{|DF^0(\xi)|}$, which is the unit outward normal vector of $\W\subset\mathbb{R}^3$ (since $\sgn({\nu}_3)=\sgn(z_0)$). Taking $e_1=(e_1^1,e_1^2,e_1^3)=(\nu_2,-\nu_1,0)\in\partial\mathbb{R}_+^3$, we then have $e_1\in T_{\xi}\partial(\W\cap\{z=z_0\})$ since $\<e_1,\nu\>=0$, and \begin{equation}
  	e_2=(e_2^1,e_2^2,e_2^3)= \nu\times(G(\xi)e_1)=\left|
  	\begin{array}{ccc}
  		\nu_1 & \nu_2 & \nu_3
  		\\
  		\hat{e}_1^1 & \hat{e}_1^2 & \hat{e}_1^3
  		\\
  		\mathbf{i} & \mathbf{j} & \mathbf{k}
  	\end{array}
  	\right|,\label{equ:A-en}
  \end{equation}
  where we denote vector $G(\xi)e_1=(	\hat{e}_1^1, \hat{e}_1^2 , \hat{e}_1^3)$. Obviously $\<e_2,\nu\>=0$, $G(\xi)(e_1,e_2)=\<G(\xi)e_1,e_2\>=0$, and $e_2^3(\xi)=-G(\xi)(e_1,e_1)<0$ (i.e., $e_2$ points outward of $\W\cap\{z\geq z_0\}\subset\W$).

  Then we calculate $Q_{112}(\xi)$ as follows:
  \begin{align*}
  	Q_{112}(\xi)=&Q(\xi)(e_1,e_1,e_2)=\sum_{i,j,k=1}^{3}Q_{\xi_i\xi_j\xi_k}(\xi)e_1^ie_1^je_2^k=\sum_{i,j,k=1}^{3} \frac{\partial^3\frac12(F^0)^2(\xi)}{\partial \xi_i \partial \xi_j \partial \xi_k}e_1^ie_1^je_2^k
  	\\
  	=&\frac{3}{2} z^3 \left(x^2 + y^2\right)^2 \left(2 x^2 + 2 y^2 + z^2\right)^4
    \left(x^4 + 2 x^2 y^2 + x^2 z^2 + y^4 + y^2 z^2 + z^4\right)^{-1}
  	\\
  &\cdot  \left( 4 x^6 + 12 x^4 y^2 + 5 x^4 z^2 + 12 x^2 y^4 + 10 x^2 y^2 z^2 + 5 x^2 z^4 + 4 y^6 + 5 y^4 z^2\right.
  \\
  &\left.  + 5 y^2 z^4 + 4 z^6\right)^{-2}	,
  \end{align*}
  which implies $Q_{112}(\xi)=0$ for $\xi\in\W\cap\{z=0\}$, $Q_{112}(\xi)>0$ for $\xi\in\W\cap\{z>0\}$, and $Q_{112}(\xi)<0$ for $\xi\in\W\cap\{z<0\}$.

  (2) In \eqref{equ:A-en}, we have determined the vector $e_2(\xi)$ (represented by $F^0$), which is  the same direction as $A_F(\nu)\mu$ (since Remark \ref{rk3.5} (1)),  but  different in length. In order to check $\eqref{equ:condition}$, we need to calculate $F(\nu)\cdot A_F(\nu)\mu$ (represented by $F^0$). 

  We denote  $\{\epsilon_i\}_{i=1}^{3}$ as the orthogonal frame on $\mathbb{R}^{3}$. For any $\xi=\xi_i\epsilon_i\in\mathbb{R}^{3}$, we have $F^0(\xi)D_iF(DF^0(\xi))=\xi_i$ (see \cite{Xia-phd}*{Proposition 1.3(ii)}). Taking the partial derivative with respect to $\xi_j$ on both sides of the above equation yields
  \begin{align*}
  	D_jF^0|_{\xi}\cdot D_iF|_{DF^0(\xi)}+F^0(\xi)D^2_{ik}F|_{DF^0(\xi)}\cdot D^2_{kj}F^0|_{\xi}=\delta_{ij}.
  \end{align*}
  Taking $\xi\in \W$, we have $F^0(\xi)=1, D_iF(DF^0(\xi))=\xi_i$, and $  F(\nu(\xi))=\<\xi,\nu(\xi)\>=\frac{\<\xi,DF^0(\xi)\>}{|D^0F(\xi)|}=\frac{1}{|D^0F(\xi)|}$, combining with \cite{Xia-phd}*{Proposition 1.2 (i) and (ii)}, we obtain
 \begin{align*}
 	\delta_{ij}=&D_jF^0(\xi){\xi}_i+D^2_{ik}F|_{DF^0(\xi)}\cdot D^2_{kj}F^0|_{\xi}
 	\\
 	=&D_jF^0(\xi){\xi}_i+D^2_{ik}F|_{DF^0(\xi)}\cdot \left(D^2_{kj}F^0({\xi})+ D_kF^0(\xi)D_jF^0(\xi)\right)
 	\\
 	=&D_jF^0(\xi){\xi}_i+D^2_{ik}F(\frac{\nu} {F(\nu)})\cdot G(\xi)
 	\\
 	=&D_jF^0(\xi){\xi}_i+F(\nu)D^2_{ik}F|_{\nu}\cdot G(\xi).
 \end{align*}
  Then we have $F(\nu)D^2F|_{\nu}=\(I-\xi^{T}\cdot DF^0(\xi)\)\cdot \(G(\xi)\)^{-1}$.
  We denote $\tilde{e}_2:=(\tilde{e}_2^1,\tilde{e}_2^2,\tilde{e}_2^3):=F(\nu)\<\mu,E_{n+1}\>A_F(\nu)\mu=F(\nu)D^2F(\nu)\(\<\mu,E_{n+1}\>\mu+\<\nu,E_{n+1}\>\nu\)=F(\nu)D^2F(\nu)E_{n+1}=\(I-\xi^{T}\cdot DF^0(\xi)\)\cdot \(G(\xi)\)^{-1}E_{n+1}$. Then we can calculate
  \begin{align*}
  	&z-\left.\frac{Q(\xi)(e_{1},e_{1},A_F(\nu)\mu)\cdot \<E_{n+1},\mu           \> \cdot F(\nu)}{G(\xi)(e_{1},e_{1})}\right|_{\xi_3=-z}
  	=z+\left.\frac{Q(\xi)(e_1,e_1,\tilde{e}_2)}{G(\xi)(e_{1},e_{1})}\right|_{\xi_3=z}
  	\\
  	=&z+\left.\sum_{i,j,k=1}^{3}\frac{Q_{\xi_i\xi_j\xi_k}(\xi)e_1^ie_1^j\tilde{e}_2^k}{G(\xi)(e_{1},e_{1})}\right|_{\xi_3=-z}
  	\\
  	=&z-6 z^3 (x^2 + y^2)\cdot(x^4 + 2 x^2 y^2 + x^2 z^2 + y^4 + y^2 z^2 + z^4)^{-\frac{1}{2}}\left(2 x^4 + 4 x^2 y^2 + 11 x^2 z^2 \right.
  	\\
  	&\left.+ 2 y^4 + 11 y^2 z^2 + 2 z^4\right)^{-1}
  	\\
  	=&z-6 z^3 (x^2 + y^2)\left(2+9 (x^2 + y^2) z^2 \right)^{-1}
  	\\
  	=&z\left(2+3(x^2 + y^2) z^2\right)\left(2+9 (x^2 + y^2) z^2 \right)^{-1}, 
  	  \end{align*}
  	  we can see that $z-\left.\frac{Q(\xi)(e_{1},e_{1},A_F(\nu)\mu)\cdot \<E_{n+1},\mu           \> \cdot F(\nu)}{G(\xi)(e_{1},e_{1})}\right|_{\xi_3=-z}\leq 0$ if and only if $z\leq 0$.
  \end{proof}
  \begin{remark}\label{rk:A.2}
  	Similarly, if we take $F^0=\left((x^2+2y^2+z^2)(x^2+ 2y^2)+z^4\right)^{\frac{1}{4}}$, one can check that $F^0$  is also a  Minkowski norm, and $\W=\{\xi\in\mathbb{R}^3:F^0(\xi)=1\}$ still satisfies (1) and (2) in Example \ref{exam:x^4}.
  \end{remark}

	\begin{Examples}\label{exam:x^4T}
		Given a constant $z_0\in(-1,1)$,  we can take a  Wulff shape ${\W}=\{(x,y,z):\((x^2+y^2+(z+z_0)^2)(x^2+ y^2)+(z+z_0)^4\)=1\}$. If we take  $\omega_0=z_0$, then $\omega_0$  and $\W$ satisfy Condition \ref{condition}.
	\end{Examples}
     \begin{proof}
     	We will prove it by using Lemma \ref{lem:Q-translate} (some relationship of $\W$ and the translated Wulff shape).
     	It's easy to check that $E_{n+1}^F=E_{n+1}$ since $\nu^{-1}((0,0,1))=(0,0,1)$.
     	We take a translated Wulff shape $\widetilde{\W}=\W+z_0E_{n+1}^F$, that is $\widetilde{\W}=\{(x,y,z):\((x^2+y^2+z^2)(x^2+ y^2)+z^4\)=1\}$. In Example \ref{exam:x^4} (1),  we have proved that $\tilde{Q}(\tilde{\xi})({e_1,e_1,e_2})=0$, for  $\tilde{\xi}=\xi+z_0E_{n+1}^F\in\widetilde{\W}\cap\{(x,y,z)\in\mathbb{R}^3:z=0\}$,  $\xi=(x,y,z)\in{\W}\cap\{z=-\omega_0\} $,  $e_1\in T_{\tilde{\xi}}\partial(\widetilde{\W}\cap\{z=0\})=T_{{\xi}}\partial({\W}\cap\{z=-\omega_0\})$ and $e_2=\widetilde{A_F}(\nu)\mu={A_F}(\nu)\mu$. By \eqref{Q-translate}, we know $\tilde{Q}_{112}(\tilde{\xi})=\frac{1}{1+\omega_0G(\xi)(\xi,E_{n+1}^F)}\(
     	Q_{112}(\xi)
     	-\frac{\omega_0G(\xi)(e_1,e_1)}{F(\nu)\<\mu,E_{n+1}\>}
     	\)$. Therefore, we have  $
     	Q_{112}(\xi)
     	=\frac{\omega_0G(\xi)(e_1,e_1)}{F(\nu)\<\mu,E_{n+1}\>}$ which satisfies \eqref{equ:condition} in Condition \ref{condition}.
     	     \end{proof}
     	
   \begin{remark}\label{rk:A.3}
   	 Similar to Example \ref{exam:x^4T},  one can check  more Wulff shapes $\W$ and corresponding $\omega_0$ which  satisfy Condition \ref{condition}, such as:
 \begin{itemize}
 	\item [(1)]  For ${\W}=\{(x,y,z):\left(2 x^2+y^2+(z+z_0)^2\right)^2+\left(x^2+2 y^2+(z+z_0)^2\right)^2+\left(x^2+y^2+2 (z+z_0)^2\right)^2=1\}$ with a constant $z_0\in(-6^{1/4},6^{1/4})$, we take $\omega_0=z_0$;
    \item  [(2)] For ${\W}=\{(x,y,z):x^4+y^4+(z+z_0)^4+x^2y^2+ x^2(z+z_0)^2+y^2 (z+z_0)^2=1\}$ with $z_0\in(-1,1)$, we take $\omega_0=z_0$.
 \end{itemize}
   \end{remark}
   \begin{remark}
   	In Example \ref{exam:x^4T} and  Remark \ref{rk:A.3}, we can take $z_0>0$. 
   	For these Wulff shapes, the result of preserving convexity in Theorem \ref{thm:convex preservation} and the   Alexandrov-Fenchel inequalities in Theorem \ref{thm:AF-neq} can hold  for   $\omega_0=z_0>0$ (that means anisotropic  contact angle   $\theta>\frac{\pi}{2}$).   	
   \end{remark}
}
\section*{Reference}



\begin{biblist}
\bib{And01}{article}{
	title={Volume-Preserving Anisotropic Mean Curvature Flow},
	author={Andrews, Ben},
	JOURNAL = {Indiana Univ. Math. J.},
	FJOURNAL = {Indiana University Mathematics Journal},
	volume={50},
	number={2},
	pages={783-827},
	year={2001},
}
\bib{Andrews}{article}{
	AUTHOR = {Andrews, Ben },
	AUTHOR = {McCoy, James},
	AUTHOR = {Zheng, Yu},
	TITLE = {Contracting convex hypersurfaces by curvature},
	JOURNAL = {Calc. Var. Partial Differential Equations},
	FJOURNAL = {Calculus of Variations and Partial Differential Equations},
	VOLUME = {47},
	YEAR = {2013},
	NUMBER = {3-4},
	PAGES = {611--665},
	ISSN = {0944-2669},
	MRCLASS = {53C44 (35K55 58J35)},
	MRNUMBER = {3070558},
	MRREVIEWER = {Kin Ming Hui},
	DOI = {10.1007/s00526-012-0530-3},
	URL = {https://doi.org/10.1007/s00526-012-0530-3},
}



\bib{Jost-1988}{article}{
	AUTHOR = {Barbosa, J. Lucas},
	AUTHOR = {do Carmo, Manfredo },
	AUTHOR = {Eschenburg, Jost},
	TITLE = {Stability of hypersurfaces of constant mean curvature in
		{R}iemannian manifolds},
	JOURNAL = {Math. Z.},
	FJOURNAL = {Mathematische Zeitschrift},
	VOLUME = {197},
	YEAR = {1988},
	NUMBER = {1},
	PAGES = {123--138},
	ISSN = {0025-5874},
	MRCLASS = {53C42},
	MRNUMBER = {917854},
	MRREVIEWER = {Johan Deprez},
	DOI = {10.1007/BF01161634},
	URL = {https://doi.org/10.1007/BF01161634},
}

	\bib{Brakke1978}{book}{
	AUTHOR = {Brakke, Kenneth A.},
	TITLE = {The motion of a surface by its mean curvature},
	SERIES = {Mathematical Notes},
	VOLUME = {20},
	PUBLISHER = {Princeton University Press, Princeton, NJ},
	YEAR = {1978},
	PAGES = {i+252},
	ISBN = {0-691-08204-9},
	MRCLASS = {49F22 (35K99 49F20 58D25)},
	MRNUMBER = {485012},
	MRREVIEWER = {Jean E. Taylor},
}

\bib{Ding-Li-2022}{article}{
	author={Ding, Shanwei},
	author={Li, Guanghan},
	title={A class of curvature flows expanded by support function and
		curvature function in the Euclidean space and hyperbolic space},
	journal={J. Funct. Anal.},
	volume={282},
	date={2022},
	number={3},
	pages={Paper No. 109305, 38},
	issn={0022-1236},
	review={\MR{4339010}},
	doi={10.1016/j.jfa.2021.109305},
}
\bib{Gao-Li2024}{article}{
	author={Gao, Jinyu},
	author={Li, Guanghan},
	title={Anisotropic Alexandrov--Fenchel Type Inequalities and
		Hsiung--Minkowski Formula},
	journal={J. Geom. Anal.},
	volume={34},
	date={2024},
	number={10},
	pages={Paper No. 312},
	issn={1050-6926},
	review={\MR{4784920}},
	doi={10.1007/s12220-024-01759-7},
}
\bib{Arxiv}{article}{
title={Generalized Minkowski formulas and rigidity results for anisotropic capillary hypersurfaces},
author={Gao, Jinyu},
author={Li, Guanghan},
year={2024},
eprint={2401.12137},
}


\bib{Gao-Luo-Xu2024}{article}{
	AUTHOR = {Gao, Zhenghuan},
	AUTHOR = {Lou, Bendong},
	AUTHOR = { Xu, Jinju},
	TITLE = {Uniform gradient bounds and convergence of mean curvature
		flows in a cylinder},
	JOURNAL = {J. Funct. Anal.},
	FJOURNAL = {Journal of Functional Analysis},
	VOLUME = {286},
	YEAR = {2024},
	NUMBER = {5},
	PAGES = {Paper No. 110283},
	ISSN = {0022-1236},
	MRCLASS = {35K93 (35B40 35B45 35C07 53E10)},
	MRNUMBER = {4682454},
	DOI = {10.1016/j.jfa.2023.110283},
	URL = {https://doi.org/10.1016/j.jfa.2023.110283},
}
\bib{Bo-Shi-Sui}{article}{
	AUTHOR = {Guan, Bo},
	AUTHOR = {Shi, Shujun},
	AUTHOR = {Sui, Zhenan},
	TITLE = {On estimates for fully nonlinear parabolic equations on
		{R}iemannian manifolds},
	JOURNAL = {Anal. PDE},
	FJOURNAL = {Analysis \& PDE},
	VOLUME = {8},
	YEAR = {2015},
	NUMBER = {5},
	PAGES = {1145--1164},
	ISSN = {2157-5045},
	MRCLASS = {58J35 (35B45 35K55)},
	MRNUMBER = {3393676},
	MRREVIEWER = {Vladimir Grebenev},
	DOI = {10.2140/apde.2015.8.1145},
	URL = {https://doi.org/10.2140/apde.2015.8.1145},
}

\bib{Guan-Li-2009}{article}{
author={Guan, Pengfei},
author={Li, Junfang},
title={The quermassintegral inequalities for $k$-convex starshaped
	domains},
journal={Adv. Math.},
volume={221},
date={2009},
number={5},
pages={1725--1732},
issn={0001-8708},
review={\MR{2522433}},
doi={10.1016/j.aim.2009.03.005},



}


\bib{Guan-Li-2015}{article}{
	AUTHOR = {Guan, Pengfei},
	AUTHOR = {Li, Junfang},
	TITLE = {A mean curvature type flow in space forms},
	JOURNAL = {Int. Math. Res. Not. IMRN},
	FJOURNAL = {International Mathematics Research Notices. IMRN},
	YEAR = {2015},
	NUMBER = {13},
	PAGES = {4716--4740},
	ISSN = {1073-7928},
	MRCLASS = {53C44},
	MRNUMBER = {3439091},
	MRREVIEWER = {Bang Xiao},
	DOI = {10.1093/imrn/rnu081},
	URL = {https://doi.org/10.1093/imrn/rnu081},
}


\bib{Guan-Li-Wang-2019}{article}{
	AUTHOR = {Guan, Pengfei},
	AUTHOR = {Li, Junfang},
	AUTHOR = {Wang, Mu-Tao},
	TITLE = {A volume preserving flow and the isoperimetric problem in
		warped product spaces},
	JOURNAL = {Trans. Amer. Math. Soc.},
	FJOURNAL = {Transactions of the American Mathematical Society},
	VOLUME = {372},
	YEAR = {2019},
	NUMBER = {4},
	PAGES = {2777--2798},
	ISSN = {0002-9947},
	MRCLASS = {53E10 (53C42)},
	MRNUMBER = {3988593},
	MRREVIEWER = {Greg\'{o}rio Manoel Silva Neto},
	DOI = {10.1090/tran/7661},
	URL = {https://doi.org/10.1090/tran/7661},
}

\bib{Guo-Xia}{article}{
author={Guo, Jinyu},
author={Xia, Chao},
title={Stable anisotropic capillary hypersurfaces in a half-space},
 year={2024},
eprint={2301.03020},
archivePrefix={arXiv},
primaryClass={math.DG}
}

\bib{HL08}{article} {
	AUTHOR = {He, Yijun},
AUTHOR = {Li, Haizhong},
	TITLE = {Stability of hypersurfaces with constant {$(r+1)$}-th
		anisotropic mean curvature},
	JOURNAL = {Illinois J. Math.},
	FJOURNAL = {Illinois Journal of Mathematics},
	VOLUME = {52},
	YEAR = {2008},
	NUMBER = {4},
	PAGES = {1301--1314},
	ISSN = {0019-2082},
	MRCLASS = {53C42 (49Q10)},
	MRNUMBER = {2595769},
	MRREVIEWER = {C\'{e}sar Rosales},
	URL = {http://projecteuclid.org/euclid.ijm/1258554364},
}
\bib{He-Li_Ma-Ge-09}{article}{
	AUTHOR = {He, Yijun},
	AUTHOR = {Li, Haizhong},
	AUTHOR = {Ma, Hui},
	AUTHOR = {Ge, Jianquan},
	TITLE = {Compact embedded hypersurfaces with constant higher order
		anisotropic mean curvatures},
	JOURNAL = {Indiana Univ. Math. J.},
	FJOURNAL = {Indiana University Mathematics Journal},
	VOLUME = {58},
	YEAR = {2009},
	NUMBER = {2},
	PAGES = {853--868},
	ISSN = {0022-2518},
	MRCLASS = {53C42},
	MRNUMBER = {2514391},
	MRREVIEWER = {C\'{e}sar Rosales},
	DOI = {10.1512/iumj.2009.58.3515},
	URL = {https://doi.org/10.1512/iumj.2009.58.3515},
}





	\bib{Huisken1984}{article}{
	AUTHOR = {Huisken, Gerhard},
	TITLE = {Flow by mean curvature of convex surfaces into spheres},
	JOURNAL = {J. Differential Geom.},
	FJOURNAL = {Journal of Differential Geometry},
	VOLUME = {20},
	YEAR = {1984},
	NUMBER = {1},
	PAGES = {237--266},
	ISSN = {0022-040X},
	MRCLASS = {53C45 (49F05 58F17)},
	MRNUMBER = {772132},
	MRREVIEWER = {R. Schneider},
	URL = {http://projecteuclid.org/euclid.jdg/1214438998},
}



\bib{Huisken1987}{article}{
	AUTHOR = {Huisken, Gerhard},
	TITLE = {The volume preserving mean curvature flow},
	JOURNAL = {J. Reine Angew. Math.},
	FJOURNAL = {Journal f\"{u}r die Reine und Angewandte Mathematik. [Crelle's
		Journal]},
	VOLUME = {382},
	YEAR = {1987},
	PAGES = {35--48},
	ISSN = {0075-4102},
	MRCLASS = {53A07 (35K99 53C40)},
	MRNUMBER = {921165},
	MRREVIEWER = {Michael Gr\"{u}ter},
	DOI = {10.1515/crll.1987.382.35},
	URL = {https://doi.org/10.1515/crll.1987.382.35},
}
\bib{Hu-Wei-yang-zhou2023}{article}{
	AUTHOR = {Hu, Yingxiang},
	AUTHOR = {Wei, Yong},
	AUTHOR = {Yang, Bo},
	AUTHOR = {Zhou, Tailong},
	TITLE = {On the mean curvature type flow for convex capillary
		hypersurfaces in the ball},
	JOURNAL = {Calc. Var. Partial Differential Equations},
	FJOURNAL = {Calculus of Variations and Partial Differential Equations},
	VOLUME = {62},
	YEAR = {2023},
	NUMBER = {7},
	PAGES = {Paper No. 209, 23},
	ISSN = {0944-2669},
	MRCLASS = {53E10 (35K93 52A40 53C21)},
	MRNUMBER = {4626464},
	DOI = {10.1007/s00526-023-02554-y},
	URL = {https://doi.org/10.1007/s00526-023-02554-y},
}
\bib{Hu-Wei-yang-zhou2023-2}{article}{
	author={Hu, Yingxiang},
	author={Wei, Yong},
	author={Yang, Bo},
	author={Zhou, Tailong},
	title={A complete family of Alexandrov-Fenchel inequalities for convex
		capillary hypersurfaces in the half-space},
	journal={Math. Ann.},
	volume={390},
	date={2024},
	number={2},
	pages={3039--3075},
	issn={0025-5831},
	review={\MR{4801847}},
	doi={10.1007/s00208-024-02841-9},
}
\bib{Jia-Wang-Xia-Zhang2023}{article}{
	AUTHOR = {Jia, Xiaohan},
	AUTHOR = {Wang, Guofang},
	AUTHOR = {Xia, Chao},
	AUTHOR = {Zhang, Xuwen},
	title={Alexandrov's theorem for anisotropic capillary hypersurfaces in the half-space},
	journal={Arch. Ration. Mech. Anal. },
	VOLUME = {247},
	YEAR = {2023},
	NUMBER = {2},
	PAGES = {19-25},
}

\bib{Koiso2006}{article}{
	author={Koiso, Miyuki},
	author={Palmer, Bennett},
	title={Stability of anisotropic capillary surfaces between two parallel
		planes},
	journal={Calc. Var. Partial Differential Equations},
	volume={25},
	date={2006},
	number={3},
	pages={275--298},
	issn={0944-2669},
	review={\MR{2201674}},
	doi={10.1007/s00526-005-0336-7},
}

\bib{Koiso2007-1}{article}{
	author={Koiso, Miyuki},
	author={Palmer, Bennett},
	title={Anisotropic capillary surfaces with wetting energy},
	journal={Calc. Var. Partial Differential Equations},
	volume={29},
	date={2007},
	number={3},
	pages={295--345},
	issn={0944-2669},
	review={\MR{2321891}},
	doi={10.1007/s00526-006-0066-5},
}


\bib{Koiso2007-2}{article}{
	author={Koiso, Miyuki},
	author={Palmer, Bennett},
	title={Uniqueness theorems for stable anisotropic capillary surfaces},
	journal={SIAM J. Math. Anal.},
	volume={39},
	date={2007},
	number={3},
	pages={721--741},
	issn={0036-1410},
	review={\MR{2349864}},
	doi={10.1137/060657297},
}
\bib{Koiso2023}{article}{
	AUTHOR = {Koiso, Miyuki},
	TITLE = {Stable anisotropic capillary hypersurfaces in a wedge},
	JOURNAL = {Math. Eng.},
	FJOURNAL = {Mathematics in Engineering},
	VOLUME = {5},
	YEAR = {2023},
	NUMBER = {2},
	PAGES = {Paper No. 029, 22},
	MRCLASS = {58E12},
	MRNUMBER = {4431669},
	MRREVIEWER = {Futoshi Takahashi},
	DOI = {10.3934/mine.2023029},
	URL = {https://doi.org/10.3934/mine.2023029},
}


\bib{3-pde}{book}{
	AUTHOR = {Lady{z}enskaja, O. A. },
	AUTHOR = { Solonnikov, V. A.},
	AUTHOR = { Ural'ceva, N. N.},
	TITLE = {Linear and quasilinear equations of parabolic type},
	SERIES = {Translations of Mathematical Monographs, Vol. 23},
	NOTE = {Translated from the Russian by S. Smith},
	PUBLISHER = {American Mathematical Society, Providence, RI},
	YEAR = {1968},
	PAGES = {xi+648},
	MRCLASS = {35.62},
	MRNUMBER = {241822},
	MRREVIEWER = {B. Frank Jones, Jr.},
}

\bib{2-pde}{book}{
	AUTHOR = {Lieberman, Gary M.},
	TITLE = {Second order parabolic differential equations},
	PUBLISHER = {World Scientific Publishing Co., Inc., River Edge, NJ},
	YEAR = {1996},
	PAGES = {xii+439},
	ISBN = {981-02-2883-X},
	MRCLASS = {35-02 (35Bxx 35Dxx 35Kxx)},
	MRNUMBER = {1465184},
	MRREVIEWER = {Siegfried Carl},
	DOI = {10.1142/3302},
	URL = {https://doi.org/10.1142/3302},
}
	
\bib{Li-Ma2020}{article}{
	AUTHOR = {Li, Guanghan},
	AUTHOR = {Ma, Kuicheng},
	TITLE = {The mean curvature type flow in {L}orentzian warped product},
	JOURNAL = {Math. Phys. Anal. Geom.},
	FJOURNAL = {Mathematical Physics, Analysis and Geometry. An International
		Journal Devoted to the Theory and Applications of Analysis and
		Geometry to Physics},
	VOLUME = {23},
	YEAR = {2020},
	NUMBER = {2},
	PAGES = {Paper No. 15, 15},
	ISSN = {1385-0172},
	MRCLASS = {53E10 (53C50)},
	MRNUMBER = {4093926},
	MRREVIEWER = {Abimbola Abolarinwa},
	DOI = {10.1007/s11040-020-09338-2},
	URL = {https://doi.org/10.1007/s11040-020-09338-2},
}
\bib{Ma-Ma-Yang}{article}{
	title={Anisotropic capillary hypersurfaces in a wedge},
	author={Hui Ma and Jiaxu Ma and Mingxuan Yang},
	year={2024},
	eprint={2403.19815},
	archivePrefix={arXiv},
	primaryClass={math.DG},
	url={https://arxiv.org/abs/2403.19815},
}
	\bib{Mei-Wang-Weng}{article}{
	TITLE = {A constrained mean curvature flow and {A}lexandrov-{F}enchel
		inequalities},
	JOURNAL = {Int. Math. Res. Not. IMRN},
	FJOURNAL = {International Mathematics Research Notices. IMRN},
	YEAR = {2024},
	NUMBER = {1},
	PAGES = {152--174},
	ISSN = {1073-7928},
	MRCLASS = {53E10 (52A40)},
	MRNUMBER = {4686648},
	DOI = {10.1093/imrn/rnad020},
	URL = {https://doi.org/10.1093/imrn/rnad020},
	author={Mei, Xinqun},
	author={Wang, Guofang},
	author={Weng, Liangjun},
}
\bib{Xia-arxiv}{article}{
	title={Alexandrov-Fenchel inequalities for convex hypersurfaces in the half-space with capillary boundary. II},
	author={Xinqun Mei and Guofang Wang and Liangjun Weng and Chao Xia},
	year={2024},
	eprint={2408.13655},
	archivePrefix={arXiv},
	primaryClass={math.MG},
	url={https://arxiv.org/abs/2408.13655},
}
	\bib{Mullins1956}{article}{
	AUTHOR = {Mullins, W. W.},
	TITLE = {Two-dimensional motion of idealized grain boundaries},
	JOURNAL = {J. Appl. Phys.},
	FJOURNAL = {Journal of Applied Physics},
	VOLUME = {27},
	YEAR = {1956},
	PAGES = {900--904},
	ISSN = {0021-8979},
	MRCLASS = {76.0X},
	MRNUMBER = {78836},
}

	\bib{Nina1992}{article}{
	AUTHOR = {Nazarov, A. I.},
	AUTHOR = {Ural\cprime tseva, N. N.},
	TITLE = {A problem with an oblique derivative for a quasilinear
		parabolic equation},
	JOURNAL = {Zap. Nauchn. Sem. S.-Peterburg. Otdel. Mat. Inst. Steklov.
		(POMI)},
	FJOURNAL = {Rossi\u{\i}skaya Akademiya Nauk. Sankt-Peterburgskoe Otdelenie.
		Matematicheski\u{\i} Institut im. V. A. Steklova. Zapiski Nauchnykh
		Seminarov (POMI)},
	VOLUME = {200},
	YEAR = {1992},
	NUMBER = {Kraev. Zadachi Mat. Fiz. Smezh. Voprosy Teor. Funktsi\u{\i}. 24},
	PAGES = {118--131, 189},
	ISSN = {0373-2703},
	MRCLASS = {35K55},
	MRNUMBER = {1192119},
	MRREVIEWER = {I. \L ojczyk-Kr\'{o}likiewicz},
	DOI = {10.1007/BF02364713},
	URL = {https://doi.org/10.1007/BF02364713},
}
	
\bib{Shen-Lecture-2001}{book}{
	AUTHOR = {Shen, Zhongmin},
	TITLE = {Lectures on {F}insler geometry},
	PUBLISHER = {World Scientific Publishing Co., Singapore},
	YEAR = {2001},
	PAGES = {xiv+307},
	ISBN = {981-02-4531-9},
	MRCLASS = {53B40 (53C60)},
	MRNUMBER = {1845637},
	MRREVIEWER = {David Bao},
	DOI = {10.1142/9789812811622},
	URL = {https://doi.org/10.1142/9789812811622},
}
	
\bib{Scheuer-Wang-Xia-2022}{article}{
	AUTHOR = {Scheuer, Julian},
	AUTHOR = {Wang, Guofang},
	AUTHOR = {Xia, Chao},
	TITLE = {Alexandrov-{F}enchel inequalities for convex hypersurfaces
		with free boundary in a ball},
	JOURNAL = {J. Differential Geom.},
	FJOURNAL = {Journal of Differential Geometry},
	VOLUME = {120},
	YEAR = {2022},
	NUMBER = {2},
	PAGES = {345--373},
	ISSN = {0022-040X},
	MRCLASS = {53A10 (53C21 53C24 53E10)},
	MRNUMBER = {4385120},
	MRREVIEWER = {Haizhong Li},
	DOI = {10.4310/jdg/1645207496},
	URL = {https://doi.org/10.4310/jdg/1645207496},
}


\bib{Convex-book}{book}{
	author={Schneider, Rolf},
	title={Convex bodies: the Brunn-Minkowski theory},
	series={Encyclopedia of Mathematics and its Applications},
	volume={151},
	edition={Second expanded edition},
	publisher={Cambridge University Press, Cambridge},
	date={2014},
	pages={xxii+736},
	isbn={978-1-107-60101-7},
	review={\MR{3155183}},
}

	

	
	
\bib{Wang-Weng-2020}{article}{
	AUTHOR = {Wang, Guofang},
	AUTHOR = {Weng, Liangjun},
	TITLE = {A mean curvature type flow with capillary boundary in a unit
		ball},
	JOURNAL = {Calc. Var. Partial Differential Equations},
	FJOURNAL = {Calculus of Variations and Partial Differential Equations},
	VOLUME = {59},
	YEAR = {2020},
	NUMBER = {5},
	PAGES = {Paper No. 149, 26},
	ISSN = {0944-2669},
	MRCLASS = {53E10 (35K93)},
	MRNUMBER = {4137803},
	MRREVIEWER = {Yoshihiro Tonegawa},
	DOI = {10.1007/s00526-020-01812-7},
	URL = {https://doi.org/10.1007/s00526-020-01812-7},
}



	
	\bib{Wang-Weng-Xia}
	{article}{
		   TITLE = {Alexandrov-{F}enchel inequalities for convex hypersurfaces in
			the half-space with capillary boundary},
		JOURNAL = {Math. Ann.},
		FJOURNAL = {Mathematische Annalen},
		VOLUME = {388},
		YEAR = {2024},
		NUMBER = {2},
		PAGES = {2121--2154},
		ISSN = {0025-5831},
		MRCLASS = {53E40 (35K96 53C21 53C24)},
		MRNUMBER = {4700391},
		DOI = {10.1007/s00208-023-02571-4},
		URL = {https://doi.org/10.1007/s00208-023-02571-4},
		author={Wang, Guofang},
		author={Weng, Liangjun},
		author={Xia, Chao},
	}
	
\bib{Wang-Xia-2022}{article}{
	AUTHOR = {Wang, Guofang},
	AUTHOR = {Xia, Chao},
	TITLE = {Guan-{L}i type mean curvature flow for free boundary
		hypersurfaces in a ball},
	JOURNAL = {Comm. Anal. Geom.},
	FJOURNAL = {Communications in Analysis and Geometry},
	VOLUME = {30},
	YEAR = {2022},
	NUMBER = {9},
	PAGES = {2157--2174},
	ISSN = {1019-8385},
	MRCLASS = {53E10},
	MRNUMBER = {4631178},
	MRREVIEWER = {Yue He},
}

\bib{Wei-Xiong2021}{article}{
	AUTHOR = {Wei, Yong },
	AUTHOR = {Xiong, Changwei},
	TITLE = {A volume-preserving anisotropic mean curvature type flow},
	JOURNAL = {Indiana Univ. Math. J.},
	FJOURNAL = {Indiana University Mathematics Journal},
	VOLUME = {70},
	YEAR = {2021},
	NUMBER = {3},
	PAGES = {881--905},
	ISSN = {0022-2518},
	MRCLASS = {53E10},
	MRNUMBER = {4284100},
	MRREVIEWER = {James Alexander McCoy},
	DOI = {10.1512/iumj.2021.70.8337},
	URL = {https://doi.org/10.1512/iumj.2021.70.8337},
}	


\bib{Wei-Xiong2022}{article}{
	AUTHOR = {Wei, Yong },
AUTHOR = {Xiong, Changwei},
	TITLE = {A fully nonlinear locally constrained anisotropic curvature
		flow},
	JOURNAL = {Nonlinear Anal.},
	FJOURNAL = {Nonlinear Analysis. Theory, Methods \& Applications. An
		International Multidisciplinary Journal},
	VOLUME = {217},
	YEAR = {2022},
	PAGES = {Paper No. 112760, 29},
	ISSN = {0362-546X},
	MRCLASS = {53E10 (53C21)},
	MRNUMBER = {4361845},
	MRREVIEWER = {Xiaolong Li},
	DOI = {10.1016/j.na.2021.112760},
	URL = {https://doi.org/10.1016/j.na.2021.112760},
}
	
\bib{Weng-Xia-2022}{article}{
	AUTHOR = {Weng, Liangjun},
	AUTHOR = {Xia, Chao},
	TITLE = {Alexandrov-{F}enchel inequality for convex hypersurfaces with
		capillary boundary in a ball},
	JOURNAL = {Trans. Amer. Math. Soc.},
	FJOURNAL = {Transactions of the American Mathematical Society},
	VOLUME = {375},
	YEAR = {2022},
	NUMBER = {12},
	PAGES = {8851--8883},
	ISSN = {0002-9947},
	MRCLASS = {53C21 (35K96 52A40)},
	MRNUMBER = {4504655},
	MRREVIEWER = {Changwei Xiong},
	DOI = {10.1090/tran/8756},
	URL = {https://doi.org/10.1090/tran/8756},
}










\bib{Xia-phd}{book}{
	AUTHOR = {Xia, Chao},
	TITLE = {On a Class of Anisotropic Problem},
	SERIES = {PhD Thesis},
	PUBLISHER = {Albert-Ludwigs University Freiburg},
	YEAR = {2012},
}


\bib{Xia13}{article}{
	author={Xia, Chao},
	title={On an anisotropic Minkowski problem},
	journal={Indiana Univ. Math. J.},
	volume={62},
	number={5},
	pages={1399–1430},
	year={2013},	
}




\bib{Xia-2017-convex}{article}{
	AUTHOR = {Xia, Chao},
	TITLE = {Inverse anisotropic curvature flow from convex hypersurfaces},
	JOURNAL = {J. Geom. Anal.},
	FJOURNAL = {Journal of Geometric Analysis},
	VOLUME = {27},
	YEAR = {2017},
	NUMBER = {3},
	PAGES = {2131--2154},
	ISSN = {1050-6926},
	MRCLASS = {53C44 (52A20)},
	MRNUMBER = {3667425},
	MRREVIEWER = {Alina Stancu},
	DOI = {10.1007/s12220-016-9755-2},
	URL = {https://doi.org/10.1007/s12220-016-9755-2},
}

\bib{Xia2017}{article}{
	title={Inverse anisotropic mean curvature flow and a Minkowski type inequality},
	author={Xia, Chao},
	JOURNAL = {Adv. Math.},
	FJOURNAL = {Advances in Mathematics},
	volume={315},
	pages={102-129},
	year={2017},
}

\end{biblist}
\end{document}